\documentclass[10pt,reqno]{amsart}
\usepackage{amssymb,mathrsfs,color}
\usepackage{enumerate}
\usepackage{graphicx}
\usepackage{xcolor} 
\usepackage{geometry}
\usepackage{amsfonts,amssymb,amsmath}
\usepackage{slashed}
\usepackage{ mathrsfs }
\usepackage[titletoc,title]{appendix}
\usepackage{wrapfig}

\usepackage{cancel}

\usepackage{cite}

\usepackage{nth}

\usepackage{hyperref}

\usepackage{marginnote}

\definecolor{green}{rgb}{0,0.8,0} 

\definecolor{deepgreen}{cmyk}{0,4,0,0}

\newcommand{\pmat}[1]{\begin{pmatrix} #1 \end{pmatrix}}

\newcommand{\Del}[1]{}

\numberwithin{equation}{section}

\newtheorem{theorem}{Theorem}[section]
\newtheorem{corollary}[theorem]{Corollary}
\newtheorem{lemma}[theorem]{Lemma}
\newtheorem{proposition}[theorem]{Proposition}
\newtheorem{remark}[theorem]{Remark}
\newtheorem{definition}[theorem]{Definition}

\newcommand{\lem}[1]{Lemma~\ref{#1}}
\newcommand{\lems}[1]{Lemmas~\ref{#1}}
\newcommand{\prop}[1]{Proposition~\ref{#1}}
\newcommand{\props}[1]{Propositions~\ref{#1}}
\newcommand{\rem}[1]{Remark~\ref{#1}}

\newcommand{\thm}[1]{Theorem~\ref{#1}}

\newcommand{\defn}[1]{Definition~\ref{#1}}

\newcommand{\cor}[1]{Corollary~\ref{#1}}

\newcommand{\mwhere}{{\ \ \text{where} \ \ }}
\newcommand{\mand}{{\ \ \text{and} \ \  }}

\usepackage{pgfplots}

\newcommand{\txtg}{{\text g}}

\renewcommand{\Re}{\mathrm{Re}}

\newcommand{\pv}{\mathrm{p.v.}}
\newcommand{\cross}{\times}

\newcommand{\bfa}{{\bf a}}

\newcommand{\bfd}{{\bf d}}
\newcommand{\bfe}{{\bf e}}
\newcommand{\bff}{{\bf f}}

\newcommand{\bfn}{{\bf n}}

\newcommand{\bfp}{{\bf p}}

\newcommand{\bfu}{{\bf u}}
\newcommand{\bfv}{{\bf v}}

\newcommand{\bfx}{{\bf x}}
\newcommand{\bfy}{{\bf y}}
\newcommand{\bfz}{{\bf z}}

\newcommand{\bfD}{{\bf D}}
\newcommand{\bfE}{{\bf E}}

\newcommand{\bfJ}{{\bf J}}
\newcommand{\bfK}{{\bf K}}

\newcommand{\bfP}{{\bf P}}

\newcommand{\bfS}{{\bf S}}

\newcommand{\bfzeta}{\boldsymbol{\zeta}}

\newcommand{\bfmu}{\boldsymbol{\mu}}
\newcommand{\bfnu}{\boldsymbol{\nu}}
\newcommand{\bfxi}{\boldsymbol{\xi}}

\newcommand{\bfphi}{\boldsymbol{\phi}}

\newcommand{\bfpsi}{\boldsymbol{\psi}}

\newcommand{\barD}{{\overline D}}

\newcommand{\barOmega}{\overline{\Omega}}

\renewcommand{\hbar}{{\underline h}}

\newcommand{\bbB}{\mathbb B}

\newcommand{\bbK}{\mathbb K}

\newcommand{\bbR}{\mathbb R}
\newcommand{\bbS}{\mathbb S}

\newcommand{\calB}{\mathcal B}
\newcommand{\calC}{\mathcal C}
\newcommand{\calD}{\mathcal D}
\newcommand{\calE}{\mathcal E}

\newcommand{\calL}{\mathcal L}

\newcommand{\calP}{\mathcal P}

\newcommand{\calR}{\mathcal R}

\newcommand{\calY}{\mathcal Y}

\newcommand{\frakb}{\mathfrak b}

\newcommand{\frakD}{\mathfrak D}

\newcommand{\frakR}{\mathfrak R}

\newcommand{\tilf}{{\tilde{f}}}

\newcommand{\tilg}{{\tilde{g}}}

\newcommand{\tilh}{{\tilde{h}}}

\newcommand{\tilr}{{\tilde{r}}}

\newcommand{\tilA}{{\tilde{A}}}

\newcommand{\tilF}{{\tilde{F}}}

\newcommand{\tilJ}{{\tilde{J}}}

\newcommand{\tilN}{{\tilde{N}}}

\newcommand{\scE}{{\mathscr{E}}}

\newcommand{\scU}{{\mathscr{U}}}

\newcommand{\vecf}{{\vec f}}

\newcommand{\vecu}{{\vec u}}

\newcommand{\vecQ}{{\vec Q}}

\newcommand{\sd}{\slashed{d}}

\newcommand{\sg}{\slashed{g}}
\newcommand{\AV}{\mathrm{AV}}
\newcommand{\AVE}{\widetilde{\mathrm{AV}}}

\newcommand{\snabla}{{\slashed{\nabla}}}
\newcommand{\rphi}{{\mathring{\phi}}}

\newcommand{\Hone}{H_{\partial\calB_1}}
\newcommand{\Kone}{K_{\partial\calB_1}}

\newcommand{\Sone}{S_{\partial\calB_{1}}}

\newcommand{\vect}{\mathrm{Vec~}}
\newcommand{\circu}{{\mathring{u}}}

\newcommand{\circf}{{\mathring{f}}}
\newcommand{\vecbff}{{\vec{\bff}}}
\newcommand{\circbff}{{\mathring{\bff}}}
\newcommand{\ringsDelta}{{\mathring{\slashed{\Delta}}}}

\newcommand{\circbfphi}{\mathring{\bfphi}}

\newcommand{\scEorbital}{\scE_{{\mathrm{orbital}}}}
\newcommand{\scEtidal}{\scE_{{\mathrm{tidal}}}}

\begin{document}
\title[Tidal Energy]{On Tidal Energy in Newtonian Two-Body Motion}

\author{Shuang Miao
\and  Sohrab Shahshahani}

\begin{abstract}
In this work, which is based on an essential linear analysis carried out by Christodoulou \cite{Ch}, we study the evolution of tidal energy for the motion of two gravitating incompressible fluid balls with free boundaries obeying the Euler-Poisson equations. The orbital energy is defined as the mechanical energy of the two bodies' center of mass. According to the classical analysis of Kepler and Newton, when the fluids are replaced by point masses, the conic curve describing the trajectories of the masses is a hyperbola when the orbital energy is positive and an ellipse when the orbital energy is negative. The orbital energy is conserved in the case of point masses. If the point masses are initially very far, then the orbital energy is positive, corresponding to hyperbolic motion. However, in the motion of fluid bodies the  orbital energy is no longer conserved because part of the conserved energy is used in deforming the boundaries of the bodies. In this case the total energy $\tilde{\scE}$ can be decomposed into a sum $\tilde{\scE}:=\widetilde{\scEorbital}+\widetilde{\scEtidal}$, with $\widetilde{\scEtidal}$ measuring the energy used in deforming the boundaries, such that if $\widetilde{\scEorbital}<-c<0$ for some absolute constant $c>0$, then the orbit of the bodies must be bounded. In this work we prove that under appropriate conditions on the initial configuration of the system, the fluid boundaries and velocity remain regular up to the point of the first closest approach in the orbit, and that the tidal energy $\widetilde{\scEtidal}$ can be made arbitrarily large relative to the total energy $\tilde{\scE}$. In particular under these conditions $\widetilde{\scEorbital}$, which is initially positive, becomes negative before the point of the first closest approach.
\end{abstract}

\maketitle

\section{Introduction}\label{sec: intro}

Consider two point masses of equal mass approaching each other from a very far distance with initial velocities $v_0$ and $-v_0$. According to Newton's laws, the masses accelerate due to the gravitational force between them. The mechanical energy (per unit mass) of the each point mass is

\begin{align}\label{eq: scE1 def}
\scE_1:=\frac{1}{2}|v|^2-\frac{GM}{4r_1},
\end{align}
where $r_1$ is the distance of each point mass to the center of mass of the system, $|v|$ is the speed of each point mass, and $G$ is the gravitational constant

\begin{align*}
G\approx 6.67\times 10^{-11}\frac{\mathrm{(meters)}^3}{\mathrm{(seconds)}^2\mathrm{(kilograms)}}.
\end{align*}
The mechanical energy $\scE_1$ is conserved in time. As shown by Kepler, the orbit of the point masses is described by a conic curve whose shape is determined by the sign of $\scE_1$: If $\scE_1<0$, then the orbit is an ellipse, if $\scE_1>0$, then the orbit is a hyperbola, and if $\scE_1=0$, then the orbit is a parabola. If the initial distance of the two point masses is very large, $\lim_{t\to-\infty}r_1(t)=\infty$, then $\scE_1>0$ and therefore the orbit of the two point masses is hyperbolic.

Suppose now that the point masses are replaced by two fluid bodies of equal mass and density, $\calB_1$ and $\calB_2$, which are initially round spheres of radius $R$. More precisely, suppose $\calB_1$ and $\calB_2$ are fluid bodies of constant density $\rho$ and volume $\frac{4\pi}{3}R^3$, with free boundaries, which satisfy the incompressible Euler equations

\begin{align}\label{eq: fluid bodies}
\begin{split}
 \begin{cases}
 \bfv_t+\bfv\cdot\nabla\bfv=-\nabla \bfp -\nabla(\bfpsi_{1}+\bfpsi_{2}),\quad\quad\,\,\, \mathrm{in } ~\calB_j(t)\\
 \nabla\cdot\bfv=0,\quad \nabla\cross\bfv=0,\quad\quad\quad\quad\quad\quad\quad\,\,\,\, \mathrm{in}~\calB_j(t)\\
 \bfp=0, \qquad\qquad\qquad\qquad\quad\quad\quad\quad\quad\quad\quad \mathrm{on } ~\partial\calB_j(t)\\
 (1,\bfv)\in T(t,\partial\calB_j(t)),
\quad\qquad\qquad\quad\quad\quad\,\,\, \,\,\mathrm{on} ~\partial\calB_j(t)
 \end{cases} ,
\end{split}
\end{align}
$j=1,2$. Here $\bfv$ denotes the fluid velocity, $\bfp$ denotes the fluid pressure, and $\bfpsi_j$, $j=1,2$, are the gravitational potentials

\begin{align*}
\bfpsi_j(t,\bfx):=-G\rho\int_{\calB_j(t)}\frac{d\bfy}{|\bfx-\bfy|},
\end{align*}
so that

\begin{align*}
\Delta\bfpsi_j(t,\bfx)=4\pi G\rho \chi_{\calB_j(t)}.
\end{align*}
For simplicity we have assumed that the fluids are irrotational and that the surface tensions are zero. For this system the \emph{total energy} of each body,

\begin{align}\label{eq: scE def}
\scE(t):=\frac{1}{2}\int_{\calB_1(t)}|\bfv(t,\bfx)|^{2}d\bfx+\frac{1}{2}\int_{\calB_1(t)}\bfpsi_1(t,\bfx)d\bfx+\frac{1}{2}\int_{\calB_1(t)}\bfpsi_2(t,\bfx)d\bfx,
\end{align}
is conserved during the evolution. Note that since the bodies are no longer point masses, $\scE$ would not agree with $\scE_1$ above even if the bodies were perfect balls moving at speed $|v|$, because the contribution of $\bfpsi_1$ to the energy is now a nonzero constant. Indeed, when $\calB_1$ is a ball of radius $R$ centered at the origin,

\begin{align*}
\begin{split}
 \frac{1}{2}\int_{\calB_1(t)}\bfpsi_1(t,\bfx)d\bfx=-\frac{G\rho}{2}\int_{B_R(0)}\int_{B_R(0)}\frac{d\bfx \,d\bfy}{|\bfx-\bfy|}=-\frac{3GM|B_R(0)|}{5R}.
\end{split}
\end{align*}
Now as the bodies approach each other their boundaries deform from their initial spherical shape. It is conceivable that the energy required to create the deformation is so large that the energy left in the center of mass motion becomes negative. That is, more significantly, it is conceivable that the orbit of the two bodies eventually becomes bounded, a phenomenon which is impossible in the point mass case, and to which we refer as \emph{tidal capture}. To better understand the situation we decompose the \emph{modified total energy} as\footnote{Note that by the incompressibility assumption the volume $|\calB_1|$ is conserved during the evolution.}

\begin{align}\label{eq: energy decomposition}
\begin{split}
\tilde{\scE}:=\frac{1}{|\calB_1|}\scE+\frac{3GM}{5R}=\widetilde{\scEorbital}+\widetilde{\scEtidal},
\end{split}
\end{align}
where

\begin{align}\label{eq: tidal def}
\begin{split}
 \widetilde{\scEtidal}:= \frac{1}{2|\calB_1|}\int_{\calB_1}|\bfv(t,\bfx)|^2d\bfx- \frac{1}{2}|\bfx_1'|^2+\frac{1}{2|\calB_1|}\int_{\calB_1}\bfpsi_1(t,\bfx)d\bfx+\frac{3GM}{5R},
\end{split}
\end{align}
and 

\begin{align}\label{eq: orbital def}
\widetilde{\scEorbital}:=\frac{1}{|\calB_1|}\scE+\frac{3GM}{5R}-\widetilde{\scEtidal}=\frac{1}{2}|\bfx_1'|^2+\frac{1}{2|\calB_1|}\int_{\calB_1}\bfpsi_2d\bfx. 
\end{align}
Note that if the initial distance between the two bodies is sufficiently large, 

\begin{align}
\frac{1}{|\calB_{1}|}\scE+\frac{3GM}{5R}=\widetilde{\scEorbital}:=e_{0}>0,\quad \widetilde{\scEtidal}=0
\end{align}
initially when the bodies are perfect balls. We will see that, as in the case of point masses, the two bodies get closer until at some point they reach their minimum distance and then start to move apart. The question of interest for us is if the bodies will continue to get arbitrarily far as in the point mass case, or if by contrast their orbit will become bounded and their distance will start to decrease again at some later time. Suppose now that we have a lower bound 

\begin{align}\label{eq: tilde tidal m0}
\begin{split}
 \widetilde{\scEtidal}\geq m_0>e_{0}, 
\end{split}
\end{align}
for all times $t>T_1$ for some $T_1$. This implies

\begin{align}\label{negative orbital energy}
\frac{1}{2}|\bfx'_{1}(t)|^{2}+\frac{1}{2|\calB_{1}|}\int_{\calB_{1}}\bfpsi_{2}(t,\bfx)d\bfx=\widetilde{\scEorbital}(t)=\frac{1}{|\calB_{1}|}\scE(t)+\frac{3GM}{5R}-\widetilde{\scEtidal}(t)\leq e_{0}-m_{0}<0.
\end{align}
Since 

\begin{align*}
\lim_{|\bfx_{1}|\rightarrow\infty}\frac{1}{2|\calB_{1}|}\int_{\calB_{1}}\bfpsi_{2}(t,\bfx)d\bfx=0,
\end{align*}
and the first term on the left hand side in \eqref{negative orbital energy} is non-negative, this would imply that the orbit cannot be unbounded. Therefore, a uniform lower bound of the form \eqref{eq: tilde tidal m0} for all times $t>T_1$ for some $T_1$, implies that tidal capture occurs. In this work we will prove that if certain relations are satisfied in the initial configuration of the system, then \eqref{eq: tilde tidal m0} holds near and up to the point of the closest approach, where $m_0$ can be made arbitrarily large relative to the initial energy $e_{0}$. To state our result more precisely, we need to describe the setup in more detail.

Consider first the case of point masses. Suppose the two point masses $\bfx_1$ and $\bfx_2$ have equal mass $M$, and satisfy 

\begin{align*}
&\lim_{t\to-\infty}\bfx_1(t)=-\lim_{t\to-\infty}\bfx_2(t)=(-b,\infty,0)\in\bbR^3,\\
&\lim_{t\to-\infty}\bfx_1'(t)=-\lim_{t\to-\infty}\bfx'_2(t)=(0,-v_0,0),\qquad v_0>0.
\end{align*}
The following dimensionless parameter plays a crucial role in the analysis in this paper:

\begin{align}\label{eq: p def}
p:=\frac{GM}{bv_0^2}.
\end{align}
The differential equation describing the motion of the point masses is given by Newton's second law

\begin{align*}
\bfx_1''(t)=-\frac{GM\bfx_{1}(t)}{4|\bfx_1(t)|^3}.
\end{align*}
By direct differentiation we see that the energy, 

\begin{align*}
\scE_1:=\frac{1}{2}|\bfx_1'(t)|^2-\frac{GM}{4|\bfx_1(t)|},
\end{align*}
and the angular momentum,

\begin{align*}
\bfJ(t)=\bfx_1(t)\cross\bfx_1'(t),
\end{align*}
are constant in time. In particular, the orbit of the point masses is confined to the $x$-$y$ coordinate plane, and we will suppress the $z$-coordinate in the remainder of the discussion for point masses. As mentioned earlier, when $\scE_1>0$ the orbit of the two point masses is a hyperbola. Indeed, a direct calculation shows that

\begin{align}\label{eq: bfe}
\bfe:=\frac{4}{GM}\left(\bfx_1'\cross\bfJ-\frac{GM}{4|\bfx_1|}\bfx_1\right),
\end{align}
is a constant vector, that is $\frac{d}{dt}\bfe=0$. The constant scalar $e:=|\bfe|$ is known as the \emph{eccentricity} of the system, and is given by

\begin{align*}
e=\frac{4}{GM}\sqrt{\left(\frac{1}{4}GM\right)^2+2\scE_1 |\bfJ|^2}=\sqrt{1+\frac{32\scE_{1}|\bfJ|^{2}}{(GM)^{2}}}.
\end{align*}
In particular,

\begin{align*}
\scE_1<0 \iff e<1\quad \mand \quad \scE_1>0\iff e>1.
\end{align*}
Let $r:=|\bfx_1|$ and $\theta$ be the angle between $\bfx_1$ and $\bfe$. It follows from \eqref{eq: bfe} that

\begin{align*}
r=\frac{4|\bfJ|^2}{GM}\frac{1}{1+e\cos\theta}.
\end{align*}
Since $e>1$ in the configuration described above, this is the polar representation of a hyperbola. By construction, one asymptote of this hyperbola, corresponding to $t\to-\infty$, is the $y$-axis. The other asymptote, $L$, corresponding to $t\to\infty$ is determined by the choice of the parameters $b$, $v_0$, $R$, and $M$. Let $\alpha$ be the angle between $L$ and the negative $y$-axis so that $\pi+\alpha$ is the angle $L$ makes with the positive $y$-axis, measured counterclowckwisely from the positive $y$-axis. See Figure \ref{fig:label}. We call $\alpha$ the \emph{scattering angle} of $\bfx_1$. If $\varphi(t)$ is the angle from the positive $y$-axis to $\bfx_1$, then

\begin{align}\label{eq: alpha 1}
\pi+\alpha=\int_{-\infty}^\infty\frac{d\varphi}{dt}dt.
\end{align}
Since $|\bfx_1'|^2=\left(\frac{dr}{dt}\right)^2+r^2\left(\frac{d\varphi}{dt}\right)^2$ and $|\bfJ|=bv_0=r^2\frac{d\varphi}{dt}$,

\begin{align*}
\left(\frac{dr}{dt}\right)^2=2(\scE_1-\scU_1),
\end{align*}
where 

\begin{align*}
\scU_1:=-\frac{GM}{4r}+\frac{|\bfJ|^2}{2r^2}.
\end{align*}
Since the orbit of $\bfx_1$ is a hyperbola, $r$ will decrease until it reaches a minimum value $r_{+}$, which we call the \emph{distance at closest approach}, and then increase again to infinity. If $t_{+}$ denotes the time at which $r(t_{+})=r_{+}$, then

\begin{align*}
&\frac{dr}{dt}=-\sqrt{2(\scE_1-\scU_1)},\qquad\, t\leq t_{+},\\
&\frac{dr}{dt}=\sqrt{2(\scE_1-\scU_1)},\quad\qquad t\geq t_{+}.
\end{align*}
Using the change of variables $\lambda=b^{-1}r$, with $\lambda_{+}:=b^{-1}r_{+}$, and recalling the definition of the parameter $p$ from \eqref{eq: p def}, we can rewrite \eqref{eq: alpha 1} as

\begin{align*}
\begin{split}
 \alpha+\pi=&\int_{-\infty}^\infty  \frac{d\varphi}{dt}dt=-\int_{\infty}^{r_+}\frac{|\bfJ|}{bv_0r\sqrt{\lambda^2+\frac{p}{2}\lambda-1}}dr+\int_{r_+}^\infty\frac{|\bfJ|}{bv_0r\sqrt{\lambda^2+\frac{p}{2}\lambda-1}}dr\\
 =&2|\bfJ|\int_{r_+}^\infty\frac{dr}{bv_0r\sqrt{\lambda^2+\frac{p}{2}\lambda-1}}
 =2\int_{\lambda_+}^\infty\frac{d\lambda}{\lambda\sqrt{\lambda^2+\frac{p}{2}\lambda-1}}.
\end{split}
\end{align*}
The quadratic polynomial $\lambda^2+\frac{p}{2}\lambda-1$ factorizes as $(\lambda-\lambda_{+})(\lambda-\lambda_{-})$, where $\lambda_{\pm}=-\frac{p}{4}\pm\sqrt{\frac{p^2}{16}+1}.$ Finally, denoting the speed at closest approach by $v_{+}:=|\bfx_1'(t_{+})|$, by conservation of energy $\scE_1$ we have $v_{+}^2=\frac{GM}{2r_{+}}\left(1+\frac{2v_0^2r_{+}}{GM}\right)$. Summarizing we get

\begin{align*}
&\lambda_{+}=\frac{1}{\frac{p}{4}+\sqrt{\frac{p^2}{16}+1}},\quad r_{+}=b\lambda_{+},\quad \alpha=2\int_{\lambda_{+}}^\infty\frac{d\lambda}{\lambda\sqrt{\lambda^2+\frac{p}{2}\lambda-1}}-\pi,\quad v_{+}^2=\frac{GM}{2r_{+}}\left(1+\frac{2v_0^2r_{+}}{GM}\right).
\end{align*}
Now consider the two limit cases where $p$ is very small and very large. First, when $p\to0$ (e.g., $bv_0^2\gg GM$) it is not surprising that the gravitational force will almost not affect the trajectories of the masses. Indeed, in this case $\lambda_+\to1$ so $r_+\to b$ and $\alpha\to 0$ and the motion  of the two masses is almost along a straight line. The more interesting case is when $p\to\infty$ (e.g., $GM\gg bv_0^2$) and in this case we have 

\begin{align*}
\begin{split}
p\lambda_+ \to 2,\quad \frac{pr_+}{b}\to2,\quad \alpha\to\pi, \quad \frac{2v_+^2 r_+}{GM}\to 1.
\end{split}
\end{align*}

\begin{figure}
\begin{tikzpicture}
  \draw[thick,->,>=latex] (-2.5,0)--(2.5,0) node[above] {$x$};
  \draw[thick,->,>=latex] (0,-4.5)--(0,4.5) node[right] {$y$};
  \draw[domain=-230:50,scale=10,samples=500,blue] plot (\x-22:{.1/(1+(1.01)*cos(\x+90))});
  \draw[domain=-230:50,scale=10,samples=500,red] plot (\x+158:{.1/(1+(1.01)*cos(-\x-90))});
  \draw[scale=1,domain=0.:4,dashed,variable=\x,blue] plot ({\x},{\x-2}) node[right] {$L$};
  \node [blue] at (-1.36,4)   {\textbullet};
   \node [blue] at (-1.36,4) [left]  {$\bfx_1$};
  \node [red] at (1.36,-4) {\textbullet};
  \node [red] at (1.36,-4) [right]  {$\bfx_2$};
  \draw [thick,decoration={brace},decorate,blue] (-1.35,4.1) -- (0,4.1) node[midway,above,yshift=.1cm] {$b$};
  \draw [thick,decoration={brace,mirror},decorate,red] (0,-4.1) -- (1.35,-4.1) node[midway,below,yshift=-.1cm] {$b$};
  \draw[blue] (0,-2.3) arc (-90:45:0.3) node[midway,right] {$\alpha$};
  \draw[ultra thick,->, blue] (-1.36,4)--(-1.36,3.2) node[left] {$-v_0$};
  \draw[ultra thick,->, red] (1.36,-4)--(1.36,-3.2) node[right] {$v_0$};
\end{tikzpicture}
\label{fig:label}
\caption{}
\end{figure}

Suppose now that the point masses are replaced by fluid bodies obeying \eqref{eq: fluid bodies}. Assume that at time $t=T_0$ the two bodies $\calB_1(T_0)$ and $\calB_2(T_0)$ are balls of radius $R$ centered at $(x,y,z)=(-b,R_0,0)$ and $(x,y,z)=(b,-R_0,0)$, respectively, with $R_0\gg b\gg R$. For future reference let $R_1:=\sqrt{b^2+R_0^2}$ be the initial distance of the center of mass of the entire system to the initial position of the center of mass of each body. We denote the center of masses of the two bodies by $\bfx_1$ and $\bfx_2=-\bfx_1$ and suppose that the initial velocities of $\bfx_1$ and $\bfx_2$ are 

\begin{align*}
\bfx_1'(T_0)=\bfv_0:=(0,-v_0,0),\qquad \bfx_2'(T_0)=-\bfv_0=(0,v_0,0),\qquad v_0>0.
\end{align*}
Note that, by symmetry, the center of mass of the entire system is always at the origin. We will discuss the local existence of a solution to equation \eqref{eq: fluid bodies} momentarily, but let us assume for now that there exists a classical solution on some interval $I:=[T_0,T_1)$. As mentioned earlier, a direct computation shows that the total energy $\scE$ defined in \eqref{eq: scE def} is conserved in $I$. Since the center of mass of each body is given by

\begin{align*}
\bfx_j=\frac{\rho}{M}\int_{\calB_j}\bfx d\bfx,
\end{align*}
 we have

\begin{align*}
\bfx_j''=\frac{\rho}{M}\int_{\calB_j}(\bfv_t+\bfv\cdot\nabla\bfv)d\bfx=-\frac{\rho}{M}\int_{\calB_j}(\nabla\bfpsi_1+\nabla\bfpsi_2)d\bfx.
\end{align*}
The modified total energy

\begin{align}\label{eq: tilde E def}
\tilde{\scE}=\frac{1}{|\calB_1|}\scE+\frac{3GM}{5R}
\end{align}
admits the decomposition

\begin{align*}
\tilde\scE=\widetilde{\scEtidal}+ \widetilde{\scEorbital},
\end{align*}
where $\widetilde{\scEtidal}$ and $\widetilde{\scEorbital}$ are defined in \eqref{eq: tidal def} and \eqref{eq: orbital def}. As in the point mass case, if the initial separation, $2R_1$, of the two bodies is sufficiently large then $\tilde{\scE}$ is positive. On the other hand, since the two bodies are initially unperturbed round balls, $\widetilde{\scEorbital}$ is positive, corresponding to hyperbolic motion in the point mass case. As discussed at the beginning of the introduction, our goal in this paper is to prove that the parameters of the problem can be set up so that $\widetilde{\scEtidal}$ becomes arbitrarily large relative to the modified total energy $\tilde{\scE}$  at the first point of closest approach, that is, when the bodies reach their minimum distance. To show this, we also need to prove that the solution of \eqref{eq: fluid bodies} does not develop singularities until the bodies get sufficiently close for $\widetilde{\scEtidal}$ to become large. To state the precise result we need to introduce some more notation. First, as in the point mass case we define

\begin{align*}
p:=\frac{GM}{bv_0^2},\qquad r_+:=\frac{2b}{p}.
\end{align*}
We also define

\begin{align*}
\eta=\eta(t):=\frac{R}{|\bfx_1(t)|},\qquad \eta_+:=\frac{R}{r_+},\qquad \beta:=\frac{b}{R}.
\end{align*}
The analysis in the point mass case shows that if $p\gg1$ the scattering angle approaches $\pi$. It is important to note that this can be achieved while keeping the distance at closest approach arbitrarily large relative to $R$. For instance, one can take $v_0\sim\sqrt{\frac{GM}{R}}\beta^{-9/10}$, $p\sim\beta^{4/5}$, and $\beta\gg1$. It follows that $r_{+}\sim R\beta^{1/5}\gg R$. It is in the context of the limiting scenario $p\gg 1$ and $r_{+}\gg R$ that we will study the evolution of fluid bodies.
We will choose $M$, $b$, and $v_0$ such that $b\gg r_+$ and $r_+ \gg R$, which in particular imply $\beta\gg1$. According to the analysis in the point mass case, the significance of the condition $b\gg r_+$ or equivalently $p\gg1$, is that the scattering angle is close to $\pi$ in the point mass analysis. The significance of the condition $r_+\gg R$ is that it will allow us to treat the deformation of the bodies perturbatively in the proof of a priori estimates. To relate the energy $\widetilde{\scEtidal}$ to the surface deformation, we introduce the Lagrangian parametrization of the surface. Let $\xi:\bbR\times S_R\to\calB_1$ be the Lagrangian parametrization of $\calB_1$ satisfying $\xi(T_0,p)=p$ for all $p$ in $S_R$, that is,

\begin{align*}
\begin{split}
 \xi_t(t,p)=\bfv(t,\xi(t,p)). 
\end{split}
\end{align*} 
Here $S_R$ is the round sphere of radius $R$. We define the \emph{height function} $h:\bbR\times S_R\to \bbR$ as

\begin{align*}
\begin{split}
 h(t,p)=|\xi(t,p)-\bfx_1(t)|-R. 
\end{split}
\end{align*}

The following theorem is the main result of this paper.


\begin{theorem}\label{thm: tidal capture}
Suppose $r_+\geq C R$ where $C>0$ is sufficiently large. Then $|\bfx_1(t)|$ is decreasing on any time interval $I_T:=[T_0,T)$ such that 
$
|\bfx_1(t)|\geq \frac{3}{2}\, r_+  \mathrm{~for~ all~} t\in I_T,
$
and a classical solution to \eqref{eq: fluid bodies} exists on the longest time interval on which $|\bfx_1|$ is decreasing. Moreover,  there exist universal constants $c_1$ and $c_2$ such that, with $r_0$ denoting the first local minimum of $|\bfx_1|$,
$
c_1\frac{GM}{R} \eta^6 \leq \widetilde{\scEtidal} \leq c_2 \frac{GM}{R} \eta^6  \mathrm{~if~} |\bfx_1(t)|\in(2r_0,r_0),
$
and $\widetilde{\scEtidal}$ is related to the height function $h$ as

\begin{align}\label{eq: tidal waves}
\widetilde{\scEtidal}\approx \frac{GM}{R^5}\| h\|_{H^1(S_R)}^2+\frac{1}{R^{2}}\|\partial_th\|_{L^2(S_R)}^2,
\end{align}
where the constants $c_1$ and $c_2$ as well as the implicit constants in \eqref{eq: tidal waves} are independent of the initial time $T_0$ and the initial separation $R_1$. In particular, if $\eta_+^5p^2\gtrsim1$, then for some $m>2$

\begin{align}\label{eq: energy comparison}
\begin{split}
\widetilde{\scEtidal}\geq m\, \tilde{\scE}
\end{split}
\end{align}
when $|\bfx_1(t)|\in(2r_0,r_0)$.
\end{theorem}


\begin{remark}\label{rem: initial data}

Condition \eqref{eq: energy comparison} in \thm{thm: tidal capture} shows that by choosing the parameters of the problem appropriately we can guarantee that the orbital energy $\widetilde{\scEorbital}$, which is initially positive, becomes negative, which as described above is a significant step in understanding the phenomenon of tidal capture. In fact, we can make $m$ in \eqref{eq: energy comparison} arbitrarily large. To see this, suppose we choose $v_0\sim \sqrt{\frac{GM}{R}}\beta^{-\alpha}$, which implies that $r_+\sim \beta^{2-2\alpha}R$, and $\eta_{+}^5p^2\sim \beta^{14\alpha-12}$. Then, for both of the conditions $r_{+}\gg R$ and $\eta_{+}^5p^2\gtrsim1$ to be satisfied, we should choose $\alpha$ in the range $[\frac{6}{7},1]$. In particular choosing $v_0=\sqrt{\frac{\mu GM}{R}}\beta^{-1}$ with $\mu\gg1$, we get $r_+\sim \mu R\gg R$, and 

\begin{align*}
\widetilde{\scEtidal}\sim \frac{GM}{R}\eta_+^6\sim \mu^{-6}\frac{GM}{R},\qquad \tilde{\scE}\sim|v_0|^2\sim \mu \beta^{-2}\frac{GM}{R}.
\end{align*}
Therefore choosing $\beta$ sufficiently large relative to $\mu$ we can make $m$ in \eqref{eq: energy comparison} arbitrarily large.

\end{remark}


\begin{remark}\label{rem: tidal waves}

The significance of the comparison \eqref{eq: tidal waves} can be explained as follows. As explained in \rem{rem: initial data} above, \thm{thm: tidal capture} implies that the tidal energy $\widetilde{\scEtidal}$ can be made arbitrarily large relative to the total energy $\tilde{\scE}$ at the time of closest approach. By \eqref{eq: tidal waves} the tidal energy is comparable to the $H^1\times L^2$ norm of $(h,\partial_th)$ of the height function (and this remains true as long as the tides remain small). Therefore, if $\|(h,\partial_th)\|_{H^1\times L^2}$ retains a nontrivial portion of its size at closest approach as the bodies move away from each other, we can argue as in the beginning of this introduction to conclude that tidal capture will indeed happen.

\end{remark}



\begin{remark}\label{rem: tidal waves}

It is important to have constants that do not depend on the initial time $T_0$ and the initial separation $R_1$ in \thm{thm: tidal capture}. The reason is that it makes sense for the bodies to be perfect balls only when their distance is infinite. The fact that for us the constants are independent of $T_0$ and $R_1$ allows us to take the limit $|T_0|,\,R_1\to\infty$ in \thm{thm: tidal capture}. It follows from our analysis that in this case the estimate $c_1\frac{GM}{R} \eta^6 \leq \widetilde{\scEtidal} \leq c_2 \frac{GM}{R} \eta^6 $ is valid for all times before the first closest approach, that is, before $|\bfx_1(t)|=r_0$.

\end{remark}


\subsection{Discussion of the Proof}
We divide the proof of \thm{thm: tidal capture} into two parts. The first part consists of proving a-priori estimates, which in particular show that the solution remains regular as long as $|\bfx_1|$ remains larger than $\frac{3}{2}r_+$. More importantly we derive a precise description of the evolution of the bodies up to closest approach. The second part of the proof is an analysis of the tidal energy $\widetilde{\scEtidal}$. Here we use the precise description and a-priori estimates from the first part of the proof to derive a lower bound for the tidal energy. This lower bound depends on the distance of the two bodies, and comparing it with the conserved total energy we will see that if $|\bfx_1|$ is sufficiently close to $r_{+}$  and $p^2\eta_{+}^5\gtrsim1$, then \eqref{eq: energy comparison} holds.


\subsubsection{A-priori estimates} Note that since the fluid velocity is a harmonic function inside $\calB_1$ and $\calB_2$, to prove regularity of the solution it suffices to prove regularity of the boundary and the velocity on the boundary. Let $\xi:\bbR\times S_R\to\partial\calB_1(t)$ be the Lagrangian parametrization of the boundary of the first body, that is, $\xi_t=\bfv(t,\xi)$, with $\xi(T_0,\cdot)=\mathrm{Id}$, and let $\zeta:=\xi-\bfx_1$ and $u:=\zeta_t$. Local existence of a solution starting at $t=T_0$ follows from the local well-posedness of the system\footnote{Local well-posedness in Sobolev spaces for the free boundary problem of the incompressible Euler equation with constant gravity was proved in \cite{Wu97, Wu99}. For the particular model at hand with only one body this was established in \cite{LinNor1,Nor1}, and the presence of a second body does not affect the local well-posedness of the system, since its acts as a lower order source term in the equation.}. However, as we are interested in understanding the dynamics of the motion up to the point of closest approach, local existence is far from sufficient for our purposes. To prove the necessary long time existence result, we derive a-priori bounds using energy estimates to  which we now turn. The first step  is to derive a quasilinear equation for $u$. Since $\bfp\equiv 0$ on $\partial\calB_1$, the gradient $\nabla\bfp\vert_{\partial\calB_1}$ points in the direction of the normal to the boundary. With $n$ denoting the exterior normal (in Lagrangian coordinates), let $a$ be defined by the relation

\begin{align*}
-\nabla\bfp(t,\xi)=an.
\end{align*}
Note that the positivity of $a$ corresponds to the Taylor sign condition (cf. \cite{GTay1}). With this notation, we write the quasilinear equation for $u$, which is derived in Section~\ref{sec: Lagrange}, in the schematic form

\begin{align}\label{eq: u eq intro}
\partial_{t}^2u+an\cross\nabla u +\frac{GM}{R^3}u-\frac{3GM}{R^3}\calP u= F+N.
\end{align}
We explain the notation. On the right hand side, the term $F$ consists of the principal contribution of the gravity from the second body $\calB_2$, and represents the contribution from the \emph{tidal accelreation} from the body $\calB_{2}$ acting on $\calB_{1}$. The error from considering only the principal contribution of the gravity of the second body as well as the genuinely nonlinear terms are contained in $N$. The operator $n\cross\nabla$, which acts componentwisely on $u$, is intrinsic to the surface $\partial\calB_1$. Indeed, for a function $f$ defined on $\calB_1$, $n\cross\nabla f$ depends only on the restriction of $f$ to $\partial\calB_1$ as the cross product with $n$ annihilates the normal component of $\nabla f$. In (orientation preserving) local coordinates $(\alpha,\beta)$ on $S_R$,

\begin{align*}
n\cross\nabla u=\frac{1}{|N|}(\xi_\beta u_\alpha-\xi_\alpha u_\beta),
\end{align*}
where $N=\xi_\alpha\cross\xi_\beta$. The operator $\calP$ is a non-local operator which maps $u$ to the projection of $(u\cdot n)n$ into the space of curl and divergence free vectorfields on $\partial\calB_1$\footnote{More precisely, vectorfields on $\partial\calB_1$ which can be written as the restriction of a curl and divergence free vectorfield in $\calB_1$. We will often simply refer to such vector fields as curl and divergence free.}. For the reader who is familiar with Clifford analysis, the projection $\calP$ admits the explicit representation

\begin{align*}
\calP u=\frac{1}{2}(I+\Hone)\left((u\cdot n)n\right),
\end{align*}
where $\Hone$ denotes the Hilbert transform

\begin{align}\label{def Hilbert}
\Hone f(\xi)=-\frac{\pv}{2\pi}\int_{\partial\calB_1}\frac{\xi'-\xi}{|\xi'-\xi|^3}n(\xi')f(\xi')dS(\xi').
\end{align}
We now discuss a novel part of our analysis which is the most delicate point in the proof of a-priori estimates, and is related to the correct bootstrap assumptions for the energy and the unknowns. Unlike the usual lifespan estimates for a small data quasilnear system, where one assumes a bootstrap bound of the form $\calE\leq 2C\epsilon$ on the energy $\calE$, where $C>0$ and $\epsilon\ll1$ are determined by the initial data, here we impose a bootstrap assumption of the form $\calE_u(t)\leq C(t)$ where $C(t)$ is a function that decays to zero as $t\to-\infty$, and $\calE_u(t)$ is the energy functional to be defined below. The reason for this choice is that the initial data for the problem are zero, so initially $\calE_u(T_0)=0$, and we need to prove the existence of a solution for infinite time rather than on a time interval depending on the size of the data. In fact, it is more natural to think of the decay in terms of the distance of the bodies from each other rather than the time of evolution. To be able to close this type of bootstrap assumption, we need to determine the correct decay rate for the energy. This is based on a careful analysis of the source term $F$ and the decay behavior of various \emph{small quantities}\footnote{That is, quantities that would be zero if $\calB_1$ were a non-accelerating ball of radius $R$.} appearing in the nonlinearity $N$. By analyzing the source term in the equations satisfied by $u$ and the height function $h$ (see Subsection \ref{intro: proof tidal capture}) we pose the bootstrap assumptions

\begin{align}\label{eq: bootstrap intro}
\|h(t)\|_{L^2(S_R)}\leq C R^2 \eta^3(t),\qquad \|u(t)\|_{L^2(S_R)}\leq C R \eta^4(t)|\bfx_1'(t)|.
\end{align}
Note that according to \eqref{eq: bootstrap intro} the various terms appearing in the nonlinearity have different decay rates, an observation which is crucial in controlling the contribution of the nonlinearity $N$ in the energy estimates. Moreover, as shown in the final section of this article where the tidal energy is analyzed, the decay rates above for $u$ and $h$ are in fact sharp. The assumptions \eqref{eq: bootstrap intro} then lead to the following bootstrap assumption for the energy: 

\begin{align}\label{eq: energy bootstrap intro}
\calE_u(t)\leq CR\,\eta^8(t)|\bfx_1'(t)|^2.
\end{align}
Recall that $\eta(t):=\frac{R}{r_1(t)}$, so \eqref{eq: bootstrap intro} and \eqref{eq: energy bootstrap intro} are formulated in terms of the position and speed of the center of masses of the bodies. Therefore, to be able to prove a-priori estimates we need to control the center of mass motion of the bodies to obtain bounds on $|\bfx_1(t)|$ and $|\bfx_1'(t)|$ when the bodies are sufficiently far. This requires analyzing the point mass mass system and the error resulting from approximating the motion of the bodies by that of point masses up to the point of closest approach.

To define the energy $\calE_u$ we take the inner product of  \eqref{eq: u eq intro} with $\frac{u_t}{a}$, which leads to the following definition:

\begin{align}\label{eq: energy def intro}
\calE_u:=\frac{1}{2}\int_{\partial\calB_1}\frac{|u_t|^2}{a} dS+\frac{1}{2}\int_{\partial\calB_1}(n\cross\nabla u)\cdot u\, dS+\frac{GM}{2R^3}\int_{\partial\calB_1}\frac{|u|^2}{a}dS-\frac{3GM}{2R^3}\int_{\partial\calB_1}\frac{(n\cdot u)^2}{a}dS.
\end{align} 
This energy satisfies

\begin{align*}
\frac{d\calE_u}{dt}=\int_{\partial\calB_1}(F+N)\frac{u_t}{a}dS+\frac{1}{2}\int_{\partial\calB_1} \frac{1}{|N|}\partial_t\left(\frac{|N|}{a}\right)|u_t|^2dS+\dots \,\,.
\end{align*}
We have not written out all the error terms on the right hand side as they are not relevant for our discussion of the main challenges in this introduction, but the precise statement can be found in \prop{prop: energy id}. The definition of our energy and the energy identity are similar to the ones used in \cite{Wu99, Wu11} to study water waves, but here several new ingredients are needed in using them to prove a-priori estimates for equation \eqref{eq: u eq intro}. These are needed to deal with the different geometry of the domain as well as the new linear and nonlinear contributions from the gravitational force which, in contrast to the water wave problem, is not a constant. The first issue to discuss is the coercivity of $\calE_u$. It is not hard to see that for a curl and divergence free vectorfield $f$ defined on $\calB_1$,

\begin{align*}
n\cross\nabla f=\nabla_n f
\end{align*}
where $\nabla_n$ denotes the Dirichlet-Neumann operator on $\calB_1$. This shows that with $\bfu:=\bfv-\bfx_1'$

\begin{align*}
\int_{\partial\calB_1}(n\cross\nabla u )\cdot u dS=\int_{\calB_1}|\nabla \bfu|^2d\bfx.
\end{align*}
However, in view of the negative sign of the last term on the right hand side of \eqref{eq: energy def intro}, it is not clear that the energy $\calE_u$ is positive in general. To show that this negative term can be controlled by the other positive terms in the definition of the energy, first note that since $\int_{\calB_1}\bfu \,d\bfx=0$, by the Poincar\'e estiamate

\begin{align*}
\int_{\calB_1}|\nabla \bfu|^2 d\bfx\geq C\int_{\calB_1}|\bfu|^2d\bfx.
\end{align*}
Combining this with a careful computation using the trace embedding $H^1(\calB_1)\hookrightarrow L^2(\partial\calB_1)$, we are able to show the coercivity of the energy, and that in particular

\begin{align*}
\calE_u\gtrsim \int_{\partial\calB_1}(n\cross\nabla u)\cdot u \,dS+\frac{GM}{R^3}\int_{\partial \calB_1}\frac{|u|^2}{a} dS.
\end{align*}
To derive this estimate we need to have bounds on the constants in the trace embedding and Poincar\'e estimates, but such bounds are available so long as the surface $\partial\calB_1$ is close to $S_R$ in the appropriate sense. These statements are made precise in \lems{lem: Poincare}, \ref{lem: trace}, and \ref{lem: pos energy}.

The last point to discuss regarding energy estimates is commuting derivatives with equation \eqref{eq: u eq intro} to estimate higher derivatives of $u$. Since the spatial domain of the Lagrangian variables is $S_R$, to prove higher regularity for $u$ we need to estimate $\Omega^ku$, where $\Omega^k$ denotes $k$ differentiations using any combination of the restriction of three rotational vectorfields

\begin{align*}
\Omega_{ij}=\bfx_i\partial_j-\bfx_j \partial_i,\quad 1\leq i <j\leq3,
\end{align*}
to $S_R$. However, to preserve the structure of the equation we need the derivative we commute to have a small commutator with the operator $\calP$ in \eqref{eq: u eq intro}, and for the derivative of $u$ to be approximately curl and divergence free. For this purpose we introduce a modified Lie derivative $\frakD u$ which satisfies these properties, and such that control of $u$ and $\frakD u$ give us control of $\Omega u$. To motivate our definition of the modified Lie derivative, suppose for the moment that $\partial\calB_1$ is a round sphere $S_R$ centered at the origin, and that $u$ is the restriction to $S_R$ of a curl and divergence free vectorfield, $\bfu$, defined in $B_R(0)$. It is then easy to see that $\calL_\Omega\bfu$ is also curl and divergence free, where $\calL_\Omega$ denotes the Lie derivative with respect to $\Omega$. Moreover, since the normal vector to $S_R$ is the radial vectorfield $\partial_r$,  we have $\calL_\Omega n=0$. Recalling the definition of $\calP$ as the projection of $(n\cdot u)n$ into the space of curl and divergence free vectorfields, we see that $\calL_\Omega$ commutes with $\calP$. Now in general, when $\calB_1$ is not $S_R$, the Lie derivative $\calL_\Omega$ is not well-defined. Instead, we observe that in $\bbR^3$, if $\bfe$ is the axis of rotation for $\Omega$, we have the simple relation

\begin{align*}
\calL_\Omega f= \Omega f-\bfe\cross f,
\end{align*}
for any vectorfield $f$. Motivated by this we define the differential operator

\begin{align*}
\frakD_\Omega=\Omega-\bfe\cross.
\end{align*}
This operator will then have a small commutator with $\calP$, and $\frakD^k u$ is almost curl and divergence free, in a sense that is made precise in Section~\ref{sec: Lagrange}.

\subsubsection{The Tidal Energy}\label{intro: proof tidal capture} Having proved the existence of a solution to \eqref{eq: fluid bodies} and estimates on the velocity and height function, $h=|\zeta|-R$, we can turn to the study of the dynamics of the equation and the proof of \eqref{eq: energy comparison}. Using the already-established a-priori estimates, it is not hard to show that

\begin{align*}
&\int_{\calB_1}|\bfv-\bfx_1'|^2d\bfx\approx R\|\partial_th\|_{L^2(S_R)}^2,\\
&\int_{\calB_1}\bfpsi_1 d\bfx+\frac{3GM|\calB_1|}{5R}\approx \frac{GM}{R^2}\|h\|_{L^2(S_R)}^2,
\end{align*}
 proving \eqref{eq: tidal waves}. Therefore, to prove \eqref{eq: energy comparison} we need to obtain a \emph{lower bound} on $$GMR^{-2}\|h\|_{L^2(S_R)}^2+R\|h_t\|_{L^2(S_R)}^2.$$ To obtain this lower bound, we study the linearized equation for the height function $h$ (see equation \eqref{eq: h intro 2} below) to derive lower bounds for $\|h\|_{L^2(S_R)}$ and $\|\partial_th\|_{L^2(S_R)}$ in the linearized setting. We then use the a-priori estimates to show that these lower bounds remain valid even after considering nonlinear effects. To derive the desired linearized equation we first derive an equation for $h$ in Lagrangian coordinates which we schematically write as

\begin{align}\label{eq: h intro}
\partial_t^2 h +\nabla_n(I-3\Kone)h = F+N,
\end{align}
where the source term $F$ contains the contribution of the gravitational force from $\calB_2$ and $N$ contains the genuinely nonlinear terms. Here $F$ and $N$ are not the same as in \eqref{eq: u eq intro}. $\nabla_n$ denotes the Dirichlet-Neumann map of $\calB_1$ and  the non-local operator $\Kone$ is the double-layered potential for $\calB_1$, defined as 

\begin{align*}
\Kone f(\xi)=\frac{\pv}{2\pi}\int_{\partial\calB_1}\frac{(\xi'-\xi)}{|\xi'-\xi|^3}\cdot n(\xi') f(\xi') dS(\xi')
\end{align*}
for any real-valued function $f$. To derive a lower bound for $\|\partial_th\|_{L^2(S_R)}$ we do not rely on energy estimates for \eqref{eq: h intro}  but use the fundamental solution for this equation instead. More precisely, we first transfer the equation to an equation on $\bbR\times S_R$, by replacing $\nabla_n$ and $\Kone$ by $\bfD$ and $\bfK$, respectively, where $\bfD$ and $\bfK$ are the Dirichlet-Neumann map and double-layered potential for $S_R$. The resulting equation is 

\begin{align}\label{eq: h intro 2}
\partial_t^2 h +\bfD(I-3\bfK)h=F+\tilN,
\end{align}
where $\tilN$ contains the new error terms which were created in passing to the equation on $\bbR\times S_R$. Note that some of these error terms (in fact, also some of the error terms in $N$) are of highest order in terms of regularity, that is equation \eqref{eq: h intro 2} is fully nonlinear. However, since we have already established higher regularity and a-priori estimates, regularity is not relevant in the analysis of \eqref{eq: h intro 2}. We can now use the fundamental solution of \eqref{eq: h intro 2} by decomposing $h$, $F$, and $\tilN$ into spherical harmonics $h_\ell$, $F_\ell$, and $\tilN_\ell$. Studying the source term $F$, we find that the main contribution to the equation comes from the second harmonic $F_2$, where $h_2$ satisfies

\begin{align}\label{eq: h2 intro}
\begin{split}
&\partial_t^2 h_2+a_2 h_{2}=F_2+\tilN_2,\\
&\lim_{t\to T_0}h_2(t)=\lim_{t\to T_0}\partial_th_2(t)=0,
\end{split}
\end{align}
for some constant coefficient $a_2>0$. The analysis of the system \eqref{eq: h2 intro}, using the explicit representation of $F_2$, is carried out in Section~\ref{sec: tidal capture}. This analysis relies heavily on the fact that the frequency of oscillation of the source term $F_2$ is much smaller than the natural frequency of the system \eqref{eq: h2 intro}. Using the explicit representation of $F_2$, we then derive the lower bounds

\begin{align*}
\|\partial_th_2\|_{L^2(S_R)}\gtrsim R\eta^4|\bfx_1'|, \qquad \|h_2\|_{L^2(S_R)}\gtrsim R^2\eta^3.
\end{align*}
Moreover, we show that the contribution of the other harmonics, $h_\ell$ and $\partial_t h_\ell$, $\ell\neq 2$, are smaller and do not affect these lower bounds, so

\begin{align*}
\|\partial_t h\|_{L^2(S_R)}\gtrsim R\eta^4|\bfx_1'|,  \qquad \|h\|_{L^2(S_R)}\gtrsim R^2\eta^3.
\end{align*}
Note that these lower bound are consistent with the bootstrap assumptions on $h$ and $\partial_t h$ in the discussion of a-priori estimates above, proving the sharpness of our bootstrap assumptions on the decay rates of $h$ and $u$. Now since the total energy is conserved during the evolution, a comparison of the implied lower bound on the tidal energy with the initial total energy allows us to conclude the proof of \thm{thm: tidal capture}.  In fact, it follows from the estimates above that the main contribution to the tidal energy up to the point of closest approach is from the the potential energy which contributes $GMR^{-5}\|h\|_{L^2(S_R)}^2$.

\subsection{Organization of the Paper}

In Section~\ref{sec: Lagrange} we derive the quasilinear system for $u:=\zeta_t$. Higher derivatives $\frakD^ku$ and the equations they satisfy are discussed in Subsection~\ref{subsec: Dk u}. The energy estimates are carried out in Section~\ref{sec: energy}. In Subsection~\ref{subsec: energy id} we define the energy and derive the basic energy identity. Here we also discuss the coercivity of the energy as discussed above. Subsection \ref{subsec: x_1 J} is devoted to the analysis of the motion of center of mass $\bfx_1$ of $\calB_1$, and estimates on $\frac{d|\bfx_1|}{dt}$ and $|\bfx_1'|$ are derived. In Subsection~\ref{subsec: bootstrap} we introduce the bootstrap assumptions and show how the various terms appearing in the nonlinearity can be controlled in terms of the energy, under the bootstrap assumptions. Finally, we close the energy estimates in Subsection~\ref{subsec: apriori}. Section~\ref{sec: tidal capture} is devoted to the analysis of the heigh function $h$ and the proof of \thm{thm: tidal capture}. The equation for $h$ is derived in Subsection~\ref{subsec: h}, where we also write the equation in spherical harmonics and obtain the key lower bounds on $h$ and $\partial_th$. This analysis is then used in Subsection~\ref{subsec: tidal capture} to complete the proof of \thm{thm: tidal capture}. There are two appendices to this article where several technical ingredients are discussed. In Appendix~\ref{app: Clifford} we recall a few basic definitions from Clifford analysis as well as the definition of layered-potentials. The notation and terminology introduce in this appendix are used throughout the paper. We also carry out some spherical harmonic decompositions which are used in Section~\ref{sec: tidal capture}. Appendix~\ref{app: CM} contains classical important singular integral estimates which are used regularly in the proof of a-priori estimates.

\subsection*{Acknowledgements}

This work is based on a linear analysis carried out by Demetrios Christodoulou \cite{Ch}. We thank him for pointing us to this interesting direction and for many insightful discussions. We thank Sijue Wu for many insightful discussions about free boundary problems for incompressible fluids. We also thank Lydia Bieri for many stimulating conversations.

\section{The Equations in Lagrangian Coordinates}\label{sec: Lagrange}

In this section we derive the main equations of motion in Lagrangian coordinates, starting with the system
\begin{align}\label{eq: Euler}
\begin{split}
 \begin{cases}
 \bfv_t+\bfv\cdot\nabla\bfv=-\nabla \bfp -\nabla\bfpsi,\quad\quad\, \mathrm{in } ~\calB(t)\\
 \nabla\cdot\bfv=0,\quad \nabla\cross\bfv=0,\quad\quad\quad\,\,\,\, \mathrm{in}~\calB(t)\\
 \bfp=0, \qquad\qquad\qquad\qquad\quad\quad\quad \mathrm{on } ~\partial\calB(t)\\
(1,\bfv)\in T(t,\partial\calB(t)),
\quad\qquad\qquad \,\mathrm{on} ~\partial\calB(t)
 \end{cases} .
\end{split}
\end{align}

For this we work with the time differentiated equation, taking the Lagrangian velocity as our main unknown. In deriving the desired equation, we will freely use the basic notation and concepts from Clifford analysis reviewed in Appendix~\ref{app: Clifford}.

\subsection{The Equation for u}

We start by introducing the notation

\begin{align*}
r_1:=|\bfx_1|,\quad \bfxi_1=\frac{\bfx_1}{|\bfx_1|},\quad \eta=\frac{R}{|\bfx_1|}.
\end{align*}
Let $\xi:\bbR\times S_R\to \partial\calB_1(t)$ be the Lagrangian parametrization of $\partial\calB_1=\partial\calB_1(t)$, such that $\xi(T_0,\cdot)-\bfx_1(T_0)$ is the identity map of $S_R\subseteq\bbR^3$, and define

\begin{align*}
\zeta:\bbR\times S_R\to \partial\calB_1-\bfx_1,\qquad \zeta(t,p)  =\xi(t,p)-\bfx_1.
\end{align*}
Occasionally we write

\begin{align*}
\zeta=\sum_{i=1}^3\zeta^i \bfe_i,
\end{align*}
where $\zeta^i=\bfx^i\circ\zeta=\bfe_i\cdot\zeta$. We will denote the exterior unit normal to $\partial\calB_1$ by $\bfn$ and let

\begin{align*}
n(t,p)=\bfn(t,\xi(t,p)).
\end{align*}
In arbitrary (orientation preserving) local coordinates $(\alpha,\beta)$ on $S_R$ we have 

\begin{align*}
n=\frac{N}{|N|},\quad\mwhere\quad N=\xi_\alpha\cross\xi_\beta=\zeta_\alpha\cross\zeta_\beta.
\end{align*}
If $\bff:\calB_1\to\bbR$ is a (possibly time-dependent) differentiable function, and $f=\bff\circ\xi$, then by a slight abuse of notation we write

\begin{align*}
\nabla f=(\nabla\bff)\circ\xi,\quad df=(df)\circ\xi,
\end{align*}
where $d$ denotes the exterior differentiation operators on $\partial\calB_1$. With this notation, and using the fact that $N=\zeta_\alpha\cross\zeta_\beta$

\begin{align*}
n\cross\nabla f:=(\bfn\cross\nabla\bff)\circ\xi=\frac{\xi_\beta f_\alpha-\xi_\alpha f_\beta}{|N|}=\frac{\zeta_\beta f_\alpha-\zeta_\alpha f_\beta}{|N|}.
\end{align*}
Note that this definition is independent of the extension $\bff$ of $\bff\vert_{\partial\calB_1}$ to the interior of $\calB_1$.

The fluid velocity in Lagrangian coordinates will be denoted by $v$, that is,

\begin{align*}
v(t,p)=\bfv(t,\xi(t,p))=\xi_t(t,p).
\end{align*}
We also let

\begin{align*}
u(t,p)=\bfv(t,\xi(t,p))-\bfx_1=\zeta_t(t,p).
\end{align*}
The gravity potential $\bfpsi$ is written as $\bfpsi=\bfpsi_1+\bfpsi_2$ with

\begin{align*}
\bfpsi_i(t,\bfx)=-G\rho\int_{\calB_j(t)}\frac{d\bfy}{|\bfx-\bfy|},
\end{align*}
and in Lagrangian coordinates we set

\begin{align*}
\psi_j(t,p):=\bfpsi_j(t,\xi(t,p)).
\end{align*}
Since $\bfp=0$ on $\partial\calB_1$,

\begin{align*}
-\nabla\bfp=\bfa\bfn,\quad\mwhere\quad\bfa:=-\nabla\bfp\cdot\bfn,
\end{align*}
and we let
\begin{align*}
a(t,p)=\bfa(t,\xi(t,p)).
\end{align*}

\begin{remark}
Suppose $\bff:\partial\calB_1\to\bbR$ is a function defined on the boundary of the fluid domain, and let $f=\bff\circ\xi$. By a slight abuse of notation, we often write integrals on the boundary $\partial\calB_1$ in terms of $f$ instead of $\bff$. For instance we write  $\int_{\partial\calB_1}f dS$ to mean $\int_{\partial\calB_1} \bff dS$, even though $f$ is a function with domain $S_R$.
\end{remark}


With the notation above the first equation in \eqref{eq: Euler} becomes

\begin{align}\label{eq: zeta eq 1}
\zeta_{tt}=an-\nabla\psi_1-(\nabla\psi_2+\bfx_1'').
\end{align}
To state the main result of this section we need to introduce some more notation. Let

\begin{align*}
h(t,p)=|\zeta(t,p)|-R,\quad\mand\quad \tilh(t,p)=|\zeta(t,p)|^2-R^2.
\end{align*}
Both $h$ and $\tilh$ vanish when $\partial\calB_1$ is a round sphere, and in general we expect them to be small during the evolution. In computations it is often convenient to replace $\zeta$ by $R^3|\zeta|^{-3}\zeta$. The reason for this is that the function $\frac{\bfx-\bfx_1}{|\bfx-\bfx_1|^3}$ has zero curl and divergence outside $\calB_1$ so $\Hone(|\zeta|^{-3}\zeta)=-|\zeta|^{-3}\zeta$. Here $\Hone$ denotes the Hilbert transform introduced in \eqref{def Hilbert}. The error that is generated from replacing $\zeta$ by $R^3|\zeta|^{-3}\zeta$ is encoded  in

\begin{align}\label{eq: mu def}
\mu:=1-\frac{R^3}{|\zeta|^3}.
\end{align}
Taking common denominators gives

 \begin{align}\label{eq: mu nu}
\mu=\frac{3}{2R^2}\tilh+\nu \tilh,
\end{align}
where

\begin{align}\label{eq: nu def}
\nu=\frac{|\zeta|^2+R|\zeta|+R^2}{|\zeta|^3(|\zeta|+R)}-\frac{3}{2R^2}.
\end{align}
Note that in view of the discussion above

\begin{align*}
(I+\Hone)\zeta=(I+\Hone)(\mu\zeta).
\end{align*}

We now state the main result of this section.


\begin{proposition}\label{prop: u eq}
$u$ satisfies

\begin{equation}\label{eq: u}
\begin{split}
\partial_t^2u+an\cross\nabla u +\frac{GM}{R^3}u-\frac{3GM}{2R^{3}}(I+\Hone)((u\cdot n)n)=&-F_t-\partial_tE_1+E_2+\partial_t\left(\frac{a}{|N|}\right)N\\
&+\partial_t\left(\frac{1}{|\calB_{1}|}\int_{\calB_{1}}\bfE_1(t,\bfx)d\bfx\right),
\end{split}
\end{equation}
where with $\bfzeta=\bfx-\bfx_1$ and $\bfzeta'=\bfy-\bfx_2$

\begin{align}\label{eq: F}
\begin{split}
  F:=\frac{GM\eta^3}{8R^3}(\zeta-3(\zeta\cdot\bfxi_1)\bfxi_1),
\end{split}
\end{align}

\begin{align}\label{eq: E1}
\begin{split}
E_1 =G\rho\,\eta^2\int_{\calB_2}\left[\frac{\eta(\zeta-(\bfy-\bfx_2))+2R\bfxi_1}{|\eta(\zeta-(\bfy-\bfx_2))+2R\bfxi_1|^3} -\frac{\bfxi_1}{4R^2}-\frac{\eta}{8R^3}\left(\zeta-(\bfy-\bfx_2)-3((\zeta-(\bfy-\bfx_2))\cdot\bfxi_1)\,\bfxi_1\right)\right]d\bfy,
\end{split}
\end{align}

\begin{align*}
\begin{split}
 \bfE_1(t,\bfx) =G\rho\,\eta^2\int_{\calB_2}\left[\frac{\eta(\bfzeta-\bfzeta')+2R\bfxi_1}{|\eta(\bfzeta-\bfzeta')+2R\bfxi_1|^3} -\frac{\bfxi_1}{4R^2}-\frac{\eta}{8R^3}\left(\bfzeta-\bfzeta'-3((\bfzeta-\bfzeta')\cdot\bfxi_1)\,\bfxi_1\right)\right]d\bfy,
\end{split}
\end{align*}

and 

\begin{align}\label{eq: E2}
\begin{split}
E_2:=&\frac{3GM}{2R^5}(I+\Hone)((u\cdot (\zeta-Rn))\zeta)+\frac{3GM}{2R^4}(I+\Hone)((u\cdot n)(\zeta-Rn))\\
&+\frac{GM}{2R^3}(I+\Hone)(\mu u)+\frac{GM}{2R^3}(I+\Hone)((\tilh_t\nu+\tilh\nu_t)\zeta)+\frac{GM}{2R^3}[\partial_t,\Hone](\mu\zeta).
\end{split}
\end{align}
\end{proposition}


\begin{remark}
Note that the left-hand side of equation \eqref{eq: u} is a pure vector, that is $(I+\Hone)((n\cdot u)n)$ has no real part. Indeed

\begin{align*}
\begin{split}
 \Hone ((n\cdot u )n)=\pv\int_{\partial\calB_1} Kn'n'(n'\cdot u')dS'=-\pv\int_{\partial\calB_1}K(n'\cdot u') dS'
\end{split}
\end{align*}
is a pure vector.
\end{remark}


To derive equation \eqref{eq: u} we will differentiate equation \eqref{eq: zeta eq 1} in time. To do this efficiently we need more convenient expressions for $\nabla\psi_1$ and $\nabla\psi_2$, which are derived in the following lemma.


\begin{lemma}\label{lem: nabla psi} 
\begin{enumerate}

\item Let $\nabla\psi_1(t,p):=\nabla\bfpsi_1(t,\xi(t,p))$. Then
\begin{align}\label{eq: nabla psi1}
\nabla\psi_1=\frac{GM}{2R^3}(I-\Hone)\zeta=\frac{GM}{R^3}\zeta-\frac{GM}{2R^3}(I+\Hone)\zeta.
\end{align}

\item Let $\bfx$ be any point in $\calB_2^c$. Then

\begin{align}\label{eq: nabla bfpsi2}
\nabla\bfpsi_2(t,\bfx)= \frac{GM\eta^2}{4R^2}\bfxi_1+\frac{GM\eta^3}{8R^3}\left((\bfx-\bfx_1)-3((\bfx-\bfx_1)\cdot\bfxi_1)\,\bfxi_1\right)+\bfE_1(t,\bfx),
\end{align}
where with $\bfzeta=\bfx-\bfx_1$ and $\bfzeta'=\bfy-\bfx_2$

\begin{align}\label{eq: bfE1}
\begin{split}
 \bfE_1(t,\bfx) =G\rho\,\eta^2\int_{\calB_2}\left[\frac{\eta(\bfzeta-\bfzeta')+2R\bfxi_1}{|\eta(\bfzeta-\bfzeta')+2R\bfxi_1|^3} -\frac{\bfxi_1}{4R^2}-\frac{\eta}{8R^3}\left(\bfzeta-\bfzeta'-3((\bfzeta-\bfzeta')\cdot\bfxi_1)\,\bfxi_1\right)\right]d\bfy.
\end{split}
\end{align}

\item Let $\nabla\psi_2(t,p):=\nabla\bfpsi_2(t,\xi(t,p))$ and $E_1(t,p):=\bfE_1(t,\xi(t,p))$. Then

\begin{align}\label{eq: nabla psi2}
\nabla\psi_2= \frac{GM\eta^2}{4R^2}\bfxi_1+\frac{GM\eta^3}{8R^3}\left(\zeta-3(\zeta\cdot\bfxi_1)\,\bfxi_1\right)+E_1.
\end{align}

\end{enumerate}

\end{lemma}


\begin{remark}
Note that the integrand in the definition \eqref{eq: bfE1} of $\bfE_1$ consists the $O(\eta^2)$ terms in the Taylor expansion of 

\begin{align*}
\begin{split}
 \frac{\eta(\bfzeta-\bfzeta')+2R\bfxi_1}{|\eta(\bfzeta-\bfzeta')+2R\bfxi_1|^3} 
\end{split}
\end{align*}
in $\eta$, and it is for this reason that $\bfE_1$ is regarded as an error term.
\end{remark}


\begin{proof}[Proof of Lemma~\ref{lem: nabla psi}]
We start with the equation for $\nabla\psi_1$. Recall that $\bfpsi_1$ satisfies

\begin{align*}
\Delta\bfpsi_1=4\pi G\rho \chi_{\calB_1}.
\end{align*}
It follows that 

\begin{align*}
\nabla\cdot \left(\nabla\bfpsi_1-\frac{4\pi G\rho}{3}\bfx\right)=0\quad\mand\quad \nabla\cross \left(\nabla\bfpsi_1-\frac{4\pi G\rho}{3}\bfx\right)=0
\end{align*}
in $\calB_1$ and hence

\begin{align*}
(I-\Hone)(\nabla\psi_1)=\frac{4\pi G\rho}{3}(I-\Hone)\xi.
\end{align*}
Similarly since $\nabla\bfpsi_1$ is curl and divergence-free outside of $\calB_1$,

\begin{align*}
(I+\Hone)\nabla\psi_1=0.
\end{align*}
It follows that

\begin{align*}
\nabla\psi_1=&\frac{1}{2}(I-\Hone)\nabla\psi_1 + \frac{1}{2}(I+\Hone)\nabla\psi_1=\frac{2\pi G\rho}{3}(I-\Hone)\xi=\frac{GM}{2R^3}(I-\Hone)\xi\\
=&\frac{GM}{R^3}\xi-\frac{GM}{2R^3}(I+\Hone)\xi=\frac{GM}{R^3}(\xi-\bfx_1)-\frac{GM}{2R^3}(I+\Hone)(\xi-\bfx_1)\\
=&\frac{GM}{R^3}\zeta-\frac{GM}{2R^3}(I+\Hone)\zeta,
\end{align*}
proving \eqref{eq: nabla psi1}.  Next we turn to the gravity of the second body. The formula for $\nabla\psi_2$ follows directly from that of $\nabla\bfpsi_2$ so we concentrate on the latter. If $\bfx\in\calB_2^c$ is an arbitrary point, then (suppressing the time variable from the notation)

\begin{align*}
\begin{split}
 \bfpsi_2(\bfx)=-G\rho\int_{\calB_2}\frac{d\bfy}{|\bfy-\bfx|}. 
\end{split}
\end{align*}
Since $\bfx\in\calB_2^c$, the integration kernel is nonsingular and we can differentiate inside the integral to get

\begin{align*}
\begin{split}
 \nabla\bfpsi_2(\bfx)=G\rho\int_{\calB_2}\frac{\bfx-\bfy}{|\bfx-\bfy|^3}d\bfy. 
\end{split}
\end{align*}
Introducing the notation $\bfzeta=\bfx-\bfx_1$ and $\bfzeta'=\bfy-\bfx_2$ and recalling that $\bfx_1=-\bfx_2$ we can rewrite this as

\begin{align}\label{eq: nabla psi temp 1}
\begin{split}
 \nabla\bfpsi_2(\bfx)=&G\rho\int_{\calB_2-\bfx_2}
 \frac{(\bfzeta-\bfzeta')+2\bfx_1}{|(\bfzeta-\bfzeta')+2\bfx_1|^3} d\bfzeta'\\
= &G\rho\eta^2\int_{\calB_2-\bfx_2}
\frac{\eta(\bfzeta-\bfzeta')+2R\bfxi_1}{|\eta(\bfzeta-\bfzeta')+2R\bfxi_1|^3}d\bfzeta'.
\end{split}
\end{align}
Now Taylor expansion in $\eta$ gives

\begin{align*}
\begin{split}
 G\rho\frac{\eta(\bfzeta-\bfzeta')+2R\bfxi_1}{|\eta(\bfzeta-\bfzeta')+2R\bfxi_1|^3}=\frac{G\rho}{4R^2}\bfxi_1+\frac{G\rho\,\eta}{8R^3}\left((\bfzeta-\bfzeta')-3((\bfzeta-\bfzeta')\cdot\bfxi_1)\,\bfxi_1\right)+\widetilde{\bfE}_1,
\end{split}
\end{align*}
where

\begin{align*}
\begin{split}
 \widetilde{\bfE}_1=G\rho\frac{\eta(\bfzeta-\bfzeta')+2R\bfxi_1}{|\eta(\bfzeta-\bfzeta')+2R\bfxi_1|^3} -\frac{G\rho}{4R^2}\bfxi_1-\frac{G\rho\,\eta}{8R^3}\left((\bfzeta-\bfzeta')-3((\bfzeta-\bfzeta')\cdot\bfxi_1)\,\bfxi_1\right)
\end{split}
\end{align*}
denotes the remainder of order $\eta^2$ in the Taylor expansion. Note that by definition,

\begin{align*}
\begin{split}
 \int_{\calB_2-\bfx_2}\bfzeta'd\bfzeta'=0,\quad\mand\quad \rho\int_{\calB_2-\bfx_2}d\bfzeta'=M,
\end{split}
\end{align*}
so plugging the expansion above into \eqref{eq: nabla psi temp 1} gives the desired identity

\begin{align*}
\begin{split}
 \nabla\bfpsi_2(\bfx)= \frac{GM\eta^2}{4R^2}\bfxi_1+\frac{GM\eta^3}{8R^3}\left((\bfx-\bfx_1)-3((\bfx-\bfx_1)\cdot\bfxi_1)\,\bfxi_1\right)+\bfE_1.
\end{split}
\end{align*}

\end{proof}


The following simple lemma provides an expression for $\bfx_1''$.


\begin{lemma}\label{lem: acce x1}
The acceleration of $\bfx_{1}$ is given by the average of $-\nabla\bfpsi_{2}$ over $\calB_{1}$. In other words, 

\begin{align*}
\bfx''_{1}(t)=-\frac{1}{|\calB_{1}|}\int_{\calB_{1}}\nabla\bfpsi_{2}(t,\bfx)d\bfx.
\end{align*}
\end{lemma}


\begin{proof}
We start by computing $\bfx''_{1}(t)$:

\begin{align*}
\bfx''_{1}=&\frac{1}{|\calB_{1}|}\frac{d^{2}}{dt^{2}}\left(\int_{\calB_{1}}\bfx d\bfx\right)=\frac{1}{|\calB_{1}|}
\int_{\calB_{1}}\left(\partial_{t}\bfv+\bfv\cdot\nabla\bfv\right)
d\bfx\\
=&-\frac{1}{|\calB_{1}|}\int_{\calB_{1}}\left(\nabla\bfp
+\nabla\bfpsi_{1}+\nabla\bfpsi_{2}\right)d\bfx.
\end{align*}
Using divergence theorem, we have

\begin{align*}
\int_{\calB_{1}}\nabla\bfp d\bfx=\int_{\partial\calB_{1}}\bfp\bfn dS=0.
\end{align*}
For the contribution from $\nabla\bfpsi_{1}$, we have

\begin{align*}
\int_{\calB_{1}}\nabla\bfpsi_{1}d\bfx=&G\rho\int_{\calB_{1}}\int_{\calB_{1}}
\frac{\bfx-\bfy}{|\bfx-\bfy|^{3}}d\bfy d\bfx
=G\rho\int_{\calB_{1}}\int_{\calB_{1}}
\frac{\bfx-\bfy}{|\bfx-\bfy|^{3}}d\bfx d\bfy=-G\rho\int_{\calB_{1}}\int_{\calB_{1}}
\frac{\bfx-\bfy}{|\bfx-\bfy|^{3}}d\bfy d\bfx=0.
\end{align*}

\end{proof}


We now turn to the proof of Proposition~\ref{prop: u eq}.


\begin{proof}[Proof of Proposition~\ref{prop: u eq}]

The statement will follow from differentiating \eqref{eq: zeta eq 1} in time.  We start by differentiating $\nabla\psi_1$. Using equations \eqref{eq: mu def}, \eqref{eq: mu nu}, and \eqref{eq: nabla psi1} we have

\begin{align}\label{eq:  u eq temp 1}
\begin{split}
-\partial_t\nabla\psi_1=&-\frac{GM}{R^3}u+\frac{GM}{2R^3}\partial_t(I+\Hone)(\mu\zeta)=-\frac{GM}{R^3}u+\frac{GM}{2R^3}[\partial_t,\Hone](\mu\zeta)+\frac{GM}{2R^3}(I+\Hone)\partial_t(\mu\zeta)\\
=&-\frac{GM}{R^3}u+\frac{3GM}{4R^5}(I+\Hone)(\tilh_t\zeta)\\
&+\frac{GM}{2R^3}(I+\Hone)(\mu u)+\frac{GM}{2R^3}(I+\Hone)((\tilh_t\nu+\tilh\nu_t)\zeta)+\frac{GM}{2R^3}[\partial_t,\Hone](\mu\zeta)\\
=&-\frac{GM}{R^3}u+\frac{3GM}{2R^5}(I+\Hone)((u\cdot\zeta)\zeta)\\
&+\frac{GM}{2R^3}(I+\Hone)(\mu u)+\frac{GM}{2R^3}(I+\Hone)((\tilh_t\nu+\tilh\nu_t)\zeta)+\frac{GM}{2R^3}[\partial_t,\Hone](\mu\zeta)\\
=&-\frac{GM}{R^3}u+\frac{3GM}{2R^3}(I+\Hone)((u\cdot n)n)+\frac{3GM}{2R^5}(I+\Hone)((u\cdot (\zeta-Rn))\zeta)\\
&+\frac{3GM}{2R^4}(I+\Hone)((u\cdot n)(\zeta-Rn))+\frac{GM}{2R^3}(I+\Hone)(\mu u)+\frac{GM}{2R^3}(I+\Hone)((\tilh_t\nu+\tilh\nu_t)\zeta)\\
&+\frac{GM}{2R^3}[\partial_t,\Hone](\mu\zeta)\\
=&-\frac{GM}{R^3}u+\frac{3GM}{2R^3}(I+\Hone)\left((u\cdot n)n\right)+E_2.
\end{split}
\end{align}
Next we differentiate $an$. Since $u$ is curl and divergence free, by \eqref{eq: nabla cross dot n}

\begin{align}\label{eq: Nt}
\begin{split}
 N_t= -(\xi_\beta \cross u_\alpha-\xi_\alpha\cross u_\beta)= -|N|n\cross\nabla u.
\end{split}
\end{align}
It follows that 

\begin{align}\label{eq: u eq temp 2}
\begin{split}
 \partial_t(an)=\partial_t\left(\frac{a}{|N|}\right)N -an\cross \nabla u.
\end{split}
\end{align}
Finally to compute $\partial_t(\nabla\psi_2+\bfx_1'')$, we note that by Lemma \ref{lem: acce x1} $\bfx_1''(t)=-\frac{1}{|\calB_{1}|}\int_{\calB_{1}}\nabla\bfpsi_2(t,\bfx)d\bfx$, and therefore by Lemma~\ref{lem: nabla psi}

\begin{align*}
\begin{split}
 \nabla\psi_2+\bfx_1''=F+E_1-\frac{1}{|\calB_{1}|}\int_{\calB_{1}}\bfE_1(t,\bfx)d\bfx,
\end{split}
\end{align*}
so

\begin{align}\label{eq: u eq temp 3}
\begin{split}
  \partial_t (\nabla\psi_2+\bfx_1'')=F_t+\partial_tE_1-\partial_t\left(\frac{1}{|\calB_{1}|}\int_{\calB_{1}}\bfE_1(t,\bfx)d\bfx\right).
\end{split}
\end{align}
Equation \eqref{eq: u} now follows by combining \eqref{eq: u eq temp 1}, \eqref{eq: u eq temp 2}, and \eqref{eq: u eq temp 3}.
\end{proof}


Before deriving the equations for higher derivatives of  $u$ we clarify the structure of the nonlinearity in equation \eqref{eq: u} a bit more. We start with deriving a formula for $\partial_t\left(\frac{a}{|N|}\right)N$ which is also of independent interest for the energy estimates.


\begin{proposition}\label{prop: at}

There holds

\begin{align}\label{eq: at}
\begin{split}
-(I+\Kone^\ast)\left(|N| \partial_t\left(\frac{a}{|N|}\right)\right)=&\Re n[\partial_t^2+an\cross\nabla,\Hone]u-\frac{GM}{2R^3}\Re n[\partial_t,\Hone](I+\Hone)\zeta\\
&+\Re n[\partial_t,\Hone](\nabla\psi_2+\bfx_1'').
\end{split}
\end{align}
\end{proposition}


\begin{proof}

We go back to equation \eqref{eq: zeta eq 1} which using \eqref{eq: nabla psi1} we rewrite as

\begin{align*}
\begin{split}
\partial_t^2\zeta-an= -\frac{GM}{R^3}\zeta+\frac{GM}{2R^3}(I+\Hone)\zeta-(\nabla\psi_2+\bfx_1'').
\end{split}
\end{align*}
Differentiating in time and using $N_t=-N\cross\nabla u$ we get

\begin{align*}
\begin{split}
 \partial_t\left(\frac{a}{|N|}\right)N =(\partial_t^2+an\cross\nabla)u+\frac{GM}{R^3}u-\frac{GM}{2R^3}\partial_t(I+\Hone)\zeta+\partial_t(\nabla\psi_2+\bfx_1'').
\end{split}
\end{align*}
Since $(I-\Hone)u=(I-\Hone)\nabla\psi_2=(I-\Hone)\bfx_1''=0$, applying $(I-\Hone)$ to this equation gives

\begin{align*}
\begin{split}
 (I-\Hone) \left( \partial_t\left(\frac{a}{|N|}\right)N\right)=[\partial_t^2+an\cross\nabla,\Hone]u-\frac{GM}{2R^3}[\partial_t,\Hone](I+\Hone)\zeta+[\partial_t,\Hone](\nabla\psi_2+\bfx_1'').
\end{split}
\end{align*}
Multiplying the equation by $n$ on both sides we get

\begin{align*}
\begin{split}
 -(I+\Hone^\ast)\left(|N|\partial_t\left(\frac{a}{|N|}\right)\right) =&n[\partial_t^2+an\cross\nabla,\Hone]u-\frac{GMn}{2R^3}[\partial_t,\Hone](I+\Hone)\zeta\\
 &+n[\partial_t,\Hone](\nabla\psi_2+\bfx_1'').
\end{split}
\end{align*}
The desired identity now follows from taking real parts.

\end{proof}


Finally we use equation \eqref{eq: zeta eq 1}  to derive expressions for $a-\frac{GM}{R^2}$ and $Rn-\zeta$ which will allow us to estimate these terms in the energy estimates.


\begin{lemma}\label{lem: a}
Let

\begin{align*}
\begin{split}
 w:=u_t-\frac{GM}{2R^3}(I+\Hone)\zeta+(\nabla\psi_2+\bfx_1''), 
\end{split}
\end{align*}
and

\begin{align*}
\begin{split}
 b:=|w|^2+\frac{2GM}{R^3}\zeta\cdot w+\frac{G^{2}M^{2}}{R^{6}}\left(|\zeta|^{2}-R^{2}\right).
\end{split}
\end{align*}
Then

\begin{align}\label{eq: a}
\begin{split}
  a-\frac{GM}{R^2}=\frac{b}{\frac{GM}{R^2}+\sqrt{\left(\frac{GM}{R^2}\right)^2+b}},
\end{split}
\end{align}
and

\begin{align}\label{eq: Rn zeta}
\begin{split}
  Rn-\zeta=\frac{R^3}{GM}u_t-\frac{R^3}{GM}\left(a-\frac{GM}{R^2}\right)n-\frac{1}{2}(I+\Hone)\zeta+\frac{R^3}{GM}(\nabla\psi_2+\bfx_1'').
\end{split}
\end{align}
\end{lemma}


\begin{proof}
First by the definition of $a=-\nabla\bfp\cdot\bfn$ we know $a\geq0$. Using \lem{lem: nabla psi} we rearrange \eqref{eq: zeta eq 1} to get 

\begin{align*}
\begin{split}
  a=\left| u_t+\frac{GM}{R^3}\zeta-\frac{GM}{2R^3}(I+\Hone)\zeta+(\nabla\psi_2+\bfx_1'')\right|=\left|\frac{GM}{R^3}\zeta+w\right|.
\end{split}
\end{align*}
A direct computation shows that $\left|\frac{GM}{R^3}\zeta+w\right|=\sqrt{b+\left(\frac{GM}{R^2}\right)^2}$ so

\begin{align*}
\begin{split}
 a-\frac{GM}{R^2}= \sqrt{b+\left(\frac{GM}{R^2}\right)^2}-\sqrt{\left(\frac{GM}{R^2}\right)^2}=\frac{b}{\frac{GM}{R^2}+\sqrt{\left(\frac{GM}{R^2}\right)^2+b}} .
\end{split}
\end{align*}
This proves \eqref{eq: a} and \eqref{eq: Rn zeta} follows directly from rearranging equation \eqref{eq: zeta eq 1} and using \lem{lem: nabla psi}.
\end{proof}


\begin{remark}
 Lemma \ref{lem: a} shows that under the bootstrap assumptions to be stated in Section~\ref{sec: energy}, $a\sim\frac{GM}{R^{2}}$, which is more precise than $a\geq0$.
 \end{remark}


\subsection{The Equation for Derivatives of u}\label{subsec: Dk u}

To obtain higher regularity, we need to commute spatial derivatives with equation \eqref{eq: u}. Let $f:\bbR\times S_R\to\bbR^3$ be a vectorfield (typically $f$ is of the form $f(t,p)=\bff(t,\xi(t,p))$ for some vectorfield $\bff:\bbR\times \partial\calB_1\to\bbR^3$, which is not necessarily tangent to $\partial\calB_1$ or $S_R$). Motivated by the discussion in the introduction we define

\begin{align}\label{eq: frakD vector}
\begin{split}
 \frakD_i  f:=\Omega_i f -\bfe_i\cross f.
\end{split}
\end{align}
Here $\Omega_1$ is the rotational vectorfield about the $\bfe_i$ axis in $\bbR^3$ and $\Omega_if=\sum_{j=1}^3(\Omega_if^j)\bfe_j$ is computed componentwisely. We also extend $\frakD_i$ to real-valued functions as

\begin{align}\label{eq: frakD scalar}
\begin{split}
  \frakD_i f:=\Omega_i f,
\end{split}
\end{align}
and if $f$ is a general Clifford algebra-valued function the we let

\begin{align*}
\begin{split}
  \frakD_i f:=\frakD_i\Re f +\frakD_i\vect f.
\end{split}
\end{align*}
Often, the choice of axis of symmetry $\bfe_i$ is irrelevant in our computations, so we simply write $\Omega$ instead of $\Omega_i$, $\frakD$ instead of $\frakD_i$, and $\bfe$ instead of $\bfe_i$ so for instance if $f$ is vector-valued, then

\begin{align*}
\begin{split}
  \frakD f =\Omega f-\bfe\cross f.
\end{split}
\end{align*}
 Finally, if $\alpha$ is any multi-index $\alpha=(i_1,\dots, i_k)$ we let 
 
 \begin{align*}
\begin{split}
 \frakD^\alpha=\frakD_{i_1}\dots\frakD_{i_k} 
\end{split}
\end{align*}
and if only the size $|\alpha|=k$ is important we simply write $\frakD^k$ instead of $\frakD^\alpha$.
 
 Before computing the equation satisfied by $\frakD u$ we record a simple product rule for $\frakD$ and an integration-by-parts formula for $\Omega$.

 
 \begin{lemma}\label{lem: product rule}
 
\begin{enumerate} 

\item If $f$ and $g$ are Clifford algebra-valued functions then
 
 \begin{align*}
\begin{split}
 \frakD (fg)=(\frakD f)g +f(\frakD  g).
\end{split}
\end{align*}
 Moreover, if $f$ and $g$ are vector-valued, then
 
 \begin{align*}
\begin{split}
 \frakD (f\cross g)=\frakD f\cross g+f\cross\frakD g, 
\end{split}
\end{align*}
and

\begin{align*}
\begin{split}
 \frakD(f\cdot g)=\frakD f\cdot g+f\cdot\frakD g. 
\end{split}
\end{align*}

\item For any differentiable $f$ and $g$ and with $|\sg|=\sqrt{\det\sg}$ denoting the volume element on $S_R$,

\begin{align*}
\begin{split}
 \int_{\partial\calB_1}f\Omega g dS =&-\int_{\partial\calB_1}(\Omega f)g dS-\int_{\partial\calB_1} fg \frac{1}{|N||\sg|^{-1}}\Omega(|N||\sg|^{-1})dS.
\end{split}
\end{align*}
Here $\sg$ is induced Euclidean metric on $S_R$.
\end{enumerate}

 \end{lemma}
 
 
 \begin{proof}
 
 \begin{enumerate}

\item The first statement follows from the usual product rule if either $f$ or $g$ are scalar-valued, so we assume that both $f$ and $g$ are vector-valued. Then
 
 \begin{align*}
\begin{split}
 \frakD(fg)=\frakD(-f\cdot g+f\cross g)=-\Omega f\cdot g-f\cdot\Omega g+\Omega f\cross g+f\cross\Omega g-\bfe\cross(f\cross g). 
\end{split}
\end{align*}
On the other hand

\begin{align*}
\begin{split}
 (\frakD f)g+f(\frakD g)=& -(\Omega f-\bfe\cross f)\cdot g-f\cdot(\Omega g-\bfe\cross g)+(\Omega f-\bfe\cross f)\cross g+f\cross(\Omega g-\bfe\cross g)\\
 =&-\Omega f\cdot g-f\cdot\Omega g+\Omega f\cross g+f\cross\Omega g-(\bfe\cross f)\cross g-f\cross(\bfe\cross g)\\
 =&-\Omega f\cdot g-f\cdot\Omega g+\Omega f\cross g+f\cross\Omega g-\bfe\cross(f\cross g)=\frakD(fg)
\end{split}
\end{align*}
according to the previous computation. This proves the first statement of the lemma, and comparing the real and vector parts proves the last two statements. 

\item This follows from the following computation, where $\snabla$ and $dS_R$ are the gradient and volume form on $S_R$ respectively:

\begin{align*}
\begin{split}
 \int_{\partial\calB_1}f\Omega g dS=&\int_{S_R} f(\Omega g) |N||\sg|^{-1}dS_R = \int_{S_R} f(\Omega\cdot\snabla g)|N||\sg|^{-1} dS_R\\
 =&-\int_{\partial\calB_1}(\Omega f)g dS-\int_{\partial\calB_1} fg \frac{1}{|N||\sg|^{-1}}\Omega(|N||\sg|^{-1})dS.
\end{split}
\end{align*}

\end{enumerate}

 \end{proof}


 We now derive the equation for the higher derivatives of $u$ in a slightly more abstract setting.

 
\begin{proposition}\label{prop: Du eq}
 Suppose $f=\bff\circ\xi$ where $\bff:\bbR\times\partial\calB_1\to\bbR^3$ is such that $\Hone f=f$. If $f$ satisfies
 
\begin{align*}
\begin{split}
\partial_t^2f+an\cross\nabla f +\frac{GM}{R^3}f-\frac{3GM}{2R^{3}}(I+\Hone)((n\cdot f)n)=g_0.
\end{split}
\end{align*}
Then for any positive integer $k$, $\frakD^k f$ satisfies
 
 \begin{align}\label{eq: Df eq}
\begin{split}
\partial_t^2\frakD^k f+an\cross\nabla \frakD^k f +\frac{GM}{R^3}\frakD^k f-\frac{3GM}{2R^{3}}(I+\Hone)((\frakD^k f\cdot n)n)=g_k&
\end{split}
\end{align}
where

\begin{align}\label{eq: gk}
\begin{split}
 g_k:=&\frakD^k g_0-\sum_{j=1}^k(\frakD^j a) n\cross\nabla \frakD^{k-j} f-\sum_{j=1}^k(\frakD^j a) [\frakD^{k-j},n\cross\nabla] f-a[\frakD^k,n\cross\nabla]f\\
&+\frac{3GM}{2R^3}[\frakD^k,\Hone]((n\cdot f)n)+\frac{3GM}{2R^3}\sum_{1\leq i+j\leq k}(I+\Hone)((\frakD^i n\cdot \frakD^{k-i-j} f)\frakD^j n).
\end{split}
\end{align}

 \end{proposition}


\begin{proof}
This follows from applying $\frakD^k$ to the equation satisfied by $f$ and using the product rules in Lemma~\ref{lem: product rule}.
\end{proof}

 
 In the following lemma we derive formulas for the commutators $[\frakD,n\cross\nabla]$ and $[\frakD,\Hone]$ appearing on the right-hand side of \eqref{eq: Df eq}.

 
 \begin{lemma}\label{lem: frakD com}

\begin{enumerate}

\item For any differentiable function $f$
 \begin{align}\label{comm D H final}
 [\frakD,\Hone]f=-\int_{\partial\calB_1}K(\frakD'\zeta'-\frakD\zeta)\cross(n'\cross\nabla f')dS.
\end{align}
 
 \item  For any $C^2$ function $f$ 
 
  \begin{align}\label{eq: D n cross nabla}
\begin{split}
   [\frakD,n\cross\nabla]f=& -\frac{\Omega(|N||\sg|^{-1})}{|N||\sg|^{-1}}n\cross\nabla f+\frac{1}{|N|}\left((\partial_{\beta}(\frakD\xi))f_{\alpha}-(\partial_{\alpha}(\frakD\xi))f_{\beta}\right).
 \end{split}
\end{align}

 \end{enumerate}
 
 \end{lemma}


 \begin{proof}
 
 \begin{enumerate}

  \item First note that $\Hone1=1$ which in particular implies $\frakD\left(\pv\int_{\partial\calB_1}Kn'dS'\right)=0$. Using this observation and two applications of Lemma~\ref{lem: frakD kernel} we get (where some of the integrals below need to be interpreted in the principal value sense)

 \begin{align}\label{eq: D H temp 3}
\begin{split}
 [\frakD,\Hone]f=&\frakD\int_{\partial\calB_1}Kn'f'dS'-\int_{\partial\calB_1}K\frakD'(n'f')dS'+\int_{\partial\calB_1}K(\frakD'n')f'dS'\\ 
 =&\int_{\partial\calB_1}((\Omega+\Omega')K-\bfe\cross K)n'f'dS'+\int_{\partial\calB_1}K(\frakD'n')f'dS' +\int_{\partial\calB_1}\frac{Kn'f'}{|N'||\sg'|^{-1}}\Omega'(|N'||\sg'|^{-1})dS'\\
 =&\int_{\partial\calB_1}((\Omega+\Omega')K-\bfe\cross K)n'(f'-f)dS'+\int_{\partial\calB_1}((\Omega+\Omega')K-\bfe\cross K)n'dS' f\\
 &+\int_{\partial\calB_1}K(\frakD'n')f'dS' +\int_{\partial\calB_1}\frac{Kn'f'}{|N'||\sg'|^{-1}}\Omega'(|N'||\sg'|^{-1})dS'\\
 =&\int_{\partial\calB_1}((\Omega+\Omega')K-\bfe\cross K)n'(f'-f)dS'\\
 &+\int_{\partial\calB_1}K(\frakD'n')(f'-f)dS'+\int_{\partial\calB_1}\frac{Kn'(f'-f)}{|N'||\sg'|^{-1}}\Omega'(|N'||\sg'|^{-1})dS'\\
  =&\int_{\partial\calB_1}((\frakD'\zeta'-\frakD\zeta)\cdot\nabla K)n'(f'-f)dS'\\
 & +\int_{\partial\calB_1}\frac{Kn'(f'-f)}{|N'||\sg'|^{-1}}\Omega'(|N'||\sg'|^{-1})dS'+\int_{\partial\calB_1}K(\frakD'n')(f'-f)dS'\\
  =&\int_{\partial\calB_1}((\frakD'\zeta'-\frakD\zeta)\cdot\nabla K)n'(f'-f)dS'\\
 & +\int_{\partial\calB_1}\frac{Kn'(f'-f)}{|\sg'|^{-1}}\Omega'(|\sg'|^{-1})dS'+\iint K(\frakD'(\zeta'_{\alpha'}\cross\zeta'_{\beta'}))(f'-f)d\alpha'd\beta'.
\end{split}
\end{align}
Now using \eqref{eq: Wu2 3.5} we write (the integration by parts here can be justified by choosing specific coordinates on $S_R$ or invariantly as in \lem{lem: int by parts} below)

\begin{align*}
\begin{split}
 \int_{\partial\calB_1}((\frakD'\zeta'-\frakD\zeta)\cdot\nabla K) n'(f'-f)dS' =&\iint\partial_{\alpha'}K((\frakD'\zeta'-\frakD\zeta)\cross\zeta'_{\beta'})(f'-f)d\alpha'd\beta'\\
 &+\iint\partial_{\beta'}K(\zeta'_{\alpha'}\cross(\frakD'\zeta'-\frakD\zeta))(f'-f)d\alpha'd\beta'\\
 =&-\int_{\partial\calB_1}K(\frakD'\zeta'-\frakD\zeta)\cross(n'\cross\nabla f')dS'\\
 &+\iint K(\zeta'_{\beta'}\cross\partial_{\alpha'}\frakD'\zeta'-\zeta'_{\alpha'}\cross\partial_{\beta'}\frakD'\zeta')(f'-f)d\alpha'd\beta'.
\end{split}
\end{align*}
Plugging this back into \eqref{eq: D H temp 3} we get

\begin{align}\label{eq: D H temp 4}
\begin{split}
 [\frakD,\Hone]f=&-\int_{\partial\calB_1}K(\frakD'\zeta'-\frakD\zeta)\cross(n'\cross\nabla f')dS'\\
 & +\int_{\partial\calB_1}\frac{Kn'(f'-f)}{|\sg'|^{-1}}\Omega'(|\sg'|^{-1})dS'\\
 &+\iint K([\Omega', \partial_{\alpha'}]\zeta'\cross\zeta'_{\beta'}+\zeta'_{\alpha'}\cross[\Omega',\partial_{\beta'}]\zeta')(f'-f)d\alpha'd\beta'.
\end{split}
\end{align}
We claim that the last two lines cancel. To see this we write $\Omega=\Omega^\alpha\partial_\alpha+\Omega^\beta\partial_\beta$ so that

\begin{align*}
\begin{split}
 [\Omega,\partial_\alpha]=-\partial_\alpha\Omega^\alpha\partial_\alpha -\partial_\alpha\Omega^\beta\partial_\beta,\quad\mand \quad [\Omega,\partial_\beta]=-\partial_\beta\Omega^\alpha\partial_\alpha-\partial_\beta\Omega^\beta\partial_\beta.
\end{split}
\end{align*}
It follows that

\begin{align*}
\begin{split}
[\Omega, \partial_{\alpha}]\zeta\cross\zeta_{\beta}+\zeta_{\alpha}\cross[\Omega,\partial_{\beta}]\zeta+|\sg|\Omega(|\sg|^{-1})N=&-\left(\sum_{\mu=1}^2|\sg|^{-1}\partial_\mu(|\sg|\Omega^\mu)\right)N=-(\snabla\cdot \Omega)N=0,
\end{split}
\end{align*}
because $\Omega$ is divergence free. Going back to \eqref{eq: D H temp 4} we conclude that

\begin{align*}
\begin{split}
 [\frakD,\Hone]f=-\int_{\partial\calB_1}K(\frakD'\zeta'-\frakD\zeta)\cross(n'\cross\nabla f')dS,
\end{split}
\end{align*}
as desired.
\item Using the product rule for $\frakD$ from \lem{lem: product rule} we write $\frakD \left(n \cross\nabla f\right)$ as

\begin{align*}
\frakD \frac{1}{|N|}(\xi_\beta f_\alpha-\xi_\alpha f_\beta)=-\frac{\Omega|N|}{|N|}n\cross\nabla f+\frac{1}{|N|}\frakD(\xi_\beta f_\alpha-\xi_\alpha f_\beta).
\end{align*}
We rewrite the second term as

\begin{align*}
\frakD(\xi_\beta f_\alpha-\xi_\alpha f_\beta)=&(\xi_\beta \partial_\alpha \frakD f-\xi_\alpha\partial_\beta\frakD f)+((\partial_\beta\frakD\xi) f_\alpha-(\partial_\alpha\frakD\xi) f_\beta)\\
&+([\Omega,\partial_\beta]\xi) f_\alpha-([\Omega,\partial_\alpha]\xi) f_\beta+\xi_\beta[\Omega,\partial_\alpha]f-\xi_\alpha[\Omega,\partial_\beta]f.
\end{align*}
For the first term we have

\begin{align*}
\frac{1}{|N|}\Omega|N|=\frac{1}{|N||\sg|^{-1}}\Omega(|\sg|^{-1}|N|)-|\sg|\Omega|\sg|^{-1}.
\end{align*}
Therefore

\begin{align*}
\frakD n\cross\nabla f=&n\cross\nabla\frakD f-\frac{\Omega(|N||\sg|^{-1})}{|N||\sg|^{-1}}n\cross\nabla f+\frac{1}{|N|}\left((\partial_{\beta}(\frakD\xi))f_{\alpha}-(\partial_{\alpha}(\frakD\xi))f_{\beta}\right)\\
&+\frac{1}{|N|}\left(([\Omega,\partial_\beta]\xi) f_\alpha-([\Omega,\partial_\alpha]\xi) f_\beta+\xi_\beta[\Omega,\partial_\alpha]f-\xi_\alpha[\Omega,\partial_\beta]f\right)+|\sg|\Omega|\sg|^{-1} n\cross\nabla f.
\end{align*}
Now using an argument similar to the one following \eqref{eq: D H temp 4} above we see that the last line is zero. 

 \end{enumerate}
 \end{proof}

Note that in view of \prop{prop: at} we will also need to commute $\frakD$ with $(I+\Kone^\ast)^{-1}$ in order to estimate the higher derivatives of the time derivative of $a$. Since $\Kone=\Re \Hone$ and in the case where $\partial\calB_1$ is a round sphere $\Hone^\ast=\Hone$, where $\Hone^\ast=n\Hone n$ is as defined in Appendix~\ref{app: Clifford}, it suffices to compute the commutator between $\frakD$ and $(I+\Kone)^{-1}$. This is an abstract computation which is presented in the next lemma.
 
 
 \begin{lemma}\label{lem: D com K}
Let $f$ be a real-valued function. Then

\begin{align}\label{eq: D com K}
\begin{split}
[\frakD,\Kone]f=\Re[\frakD,\Hone]f
\end{split}
\end{align} 
and

\begin{align}\label{eq: D com (I+K)-1}
\begin{split}
    [\frakD,(I+\Kone)^{-1}]f=-(I+\Kone)^{-1}[\frakD,\Kone](I+\Kone)^{-1}f.
\end{split}
\end{align}
\end{lemma}


\begin{proof}

To prove \eqref{eq: D com K} we first note that by definition $\frakD F= \frakD \mathring{F} +\frakD\vec{F}$ (in the notation of Appendix~\ref{app: Clifford}) for any Clifford algebra-valued function $F$, so $\frakD\Re F = \Re\frakD F$. It follows that

\begin{align*}
\begin{split}
 \frakD \Kone f=\frakD\Re \Hone  f=\Re \frakD \Hone f=\Re \Hone \frakD f+\Re[\frakD,\Hone]f=\Kone \frakD f+\Re[\frakD,\Hone]f,
\end{split}
\end{align*}
proving \eqref{eq: D com K}. For \eqref{eq: D com (I+K)-1} we let $g:=(I+\Kone)^{-1}f$. Then

\begin{align*}
\begin{split}
 \frakD g= \frakD(I+\Kone)^{-1}f=(I+\Kone)^{-1}\frakD f+[\frakD,(I+\Kone)^{-1}]f.
\end{split}
\end{align*}
So

\begin{align*}
\begin{split}
  [\frakD,(I+\Kone)^{-1}]f=& \frakD g-(I+\Kone)^{-1}\frakD (I+\Kone)g=\frakD g-(I+\Kone)^{-1}[\frakD,\Kone]g-\frakD g\\
  =&-(I+\Kone)^{-1}[\frakD,\Kone](I+\Kone)^{-1}f.
\end{split}
\end{align*}
\end{proof}


\prop{prop: Du eq} can be used to derive the equation for $\frakD^k u$. However, as we will see in Section~\ref{sec: energy}, for the purposes of the energy estimates it is more convenient to work with  the unknown $\vecu_k$ where

\begin{align*}
\begin{split}
 u_k:=\frac{1}{2}(I+\Hone)\frakD^k u .
\end{split}
\end{align*}
In view of \lem{lem: frakD com}, the difference between $\frakD^ku$ and $\vecu_k$ is small, but $\vecu_k$ has the advantage that it is the vector part of a Clifford analytic function. The following proposition allows us to derive the equation satisfied by $\vecu_k$.


\begin{proposition}\label{prop: vecuk}

 Suppose $f=\bff\circ\xi$ where $\bff:\bbR\times\partial\calB_1\to\bbR^3$ is such that $\Hone f=f$, and $f$ satisfies

\begin{align*}
\begin{split}
\partial_t^2f+an\cross\nabla f +\frac{GM}{R^3}f-\frac{3GM}{2R^{3}}(I+\Hone)((n\cdot f)n)=g_0.
\end{split}
\end{align*}
For any positive integer $k$ let $f_k:=\frac{1}{2} (I+\Hone)\frakD^k f$ and let $g_k$ be as defined in \eqref{eq: gk}. Then $\vecf_k$ satisfies

\begin{align*}
\begin{split}
\partial_t^2\vecf_k+an\cross\nabla \vecf_k +\frac{GM}{R^3}\vecf_k-\frac{3GM}{2R^{3}}(I+\Hone)((n\cdot \vecf_k)n)=\tilg_k,
\end{split}
\end{align*}
where

\begin{align}\label{eq: tilgk}
\begin{split}
 \tilg_k:=&g_k -\frac{1}{2}an\cross\nabla \vect[\frakD^k,\Hone]f -\frac{GM}{2R^3}\vect[\frakD^k,\Hone]f+\frac{3GM}{4R^{3}}(I+\Hone)((n\cdot \vect[\frakD^k,\Hone]f)n)\\
 &-\frac{\vect}{2}[\frakD^k,\Hone]\left(g_0-an\cross\nabla f -\frac{GM}{R^3}f+\frac{3GM}{2R^{3}}(I+\Hone)((n\cdot f)n)\right)-\frac{\vect}{2}[\partial_t^2,[\frakD^k,\Hone]]f.
\end{split}
\end{align}

\end{proposition}


\begin{proof}
Let $\calP$ denote the operator on the left-hand side of the equation for $f$, and let $\calP_0$ be the spatial part of this operator, that is

\begin{align*}
\begin{split}
 \calP f:=\partial_t^2 f+\calP_0 f:=\partial_t^2f+an\cross\nabla f +\frac{GM}{R^3}f-\frac{3GM}{2R^{3}}(I+\Hone)((n\cdot f)n),
\end{split}
\end{align*}
Then since $\Hone f =f$ and $\Re f=0$

\begin{align*}
\begin{split}
 \calP \vecf_k=&\calP\frakD^kf-\frac{1}{2}\calP\vect [\frakD^k,\Hone]f \\
 =&g_k-\frac{1}{2}\calP_0\vect[\frakD^k,\Hone]f-\frac{\vect}{2}[\frakD^k,\Hone]\partial_t^2f-\frac{\vect}{2}[\partial_t^2,[\frakD^k,\Hone]]f.
\end{split}
\end{align*}
The desired identity now follows if we use the equation for $f$ to solve for $\partial_t^2f$.
\end{proof}


We can now combine \props{prop: u eq},~\ref{prop: Du eq}, and \ref{prop: vecuk} to derive the equation for higher derivatives of $u$. We record this equation below for future reference. 


\begin{corollary}\label{cor: uk eq}

Let

\begin{align*}
\begin{split}
 g_0:=-F_t-\partial_tE_1+\partial_t\left(\frac{1}{|\calB_{1}|}\int_{\calB_{1}}\bfE_1(t,\bfx)d\bfx\right)+E_2+\partial_t\left(\frac{a}{|N|}\right)N, 
\end{split}
\end{align*}
 where $F$, $\bfE_1$ and $E_2$ are defined as in \eqref{eq: F},~\eqref{eq: E1}, and \eqref{eq: E2} respectively. With this choice of $g_0$, let $g_k$ and $\tilg_k$ be defined as in \eqref{eq: gk} and \eqref{eq: tilgk} respectively, with $f=u$. Let $u_k:=\frac{1}{2}(I+\Hone)\frakD^ku$. Then $\vecu_k$ satisfies
 
 \begin{align}\label{eq: uk eq}
\begin{split}
\partial_t^2\vecu_k+an\cross\nabla \vecu_k +\frac{GM}{R^3}\vecu_k-\frac{3GM}{2R^{3}}(I+\Hone)((n\cdot \vecu_k)n)=\tilg_k,
\end{split}
\end{align}

 \end{corollary}
 
 
 \begin{proof}
 
 This follows directly from \props{prop: u eq},~\ref{prop: Du eq}, and \ref{prop: vecuk}.
 
 \end{proof}
 

\section{Energy Estimates} \label{sec: energy}

\subsection{General Setup}\label{subsec: energy id}

Given two Clifford algebra-valued functions $f$ and $g$ we define their ``dot product" as

\begin{align*}
f\cdot g = \Re f \Re g+ \vect (f)\cdot \vect (g).
\end{align*}
Recall that

\begin{align*}
\Hone^\ast:=\bfn \Hone \bfn
\end{align*}
is the formal adjoint of $\Hone$ with respect to the pairing

\begin{align*}
\langle f,g\rangle =\int_{\partial\calB_1}f\cdot g \,dS,
\end{align*}
where $dS$ denotes the volume form on $\partial\calB_1$. Recall also that sometimes we use the notation

\begin{align*}
\begin{split}
 \vecf:=\vect{f} \quad\mand  \quad \circf=\Re f.
\end{split}
\end{align*}
We consider the model equation

\begin{equation}\label{eq: model eq}
\partial_t^2 f+ an\cross\nabla f+\frac{GM}{R^3}f-\frac{3GM}{2R^3}(I+\Hone)((n\cdot f)n)=g,
\end{equation}
where $f=\vect \tilf$ and $\tilf$ satisfies $\tilf=\Hone \tilf$, and define the following associated energy

\begin{equation}\label{eq: model energy}
\calE:=\calE_f(t):=\frac{1}{2}\int_{\partial\calB_1}\frac{|f_t|^2}{a} dS+\frac{1}{2}\int_{\partial\calB_1}(n\cross\nabla f)\cdot f dS+\frac{GM}{2R^3}\int_{\partial\calB_1}\frac{|f|^2}{a}dS-\frac{3GM}{2R^3}\int_{\partial\calB_1}\frac{(n\cdot f)^2}{a}dS.
\end{equation}
When $\tilf=u$, then of course $\Re\tilf=0$ and $\tilf$ is Clifford analytic, but after commuting derivatives with equation \eqref{eq: u} the new unknowns $\frakD^ku$ are not necessarily Clifford analytic, and we will instead work with $\vecu_k$ where $u_k=\frac{1}{2}(I+\Hone)\frakD^ku$. It is for this reason that we have defined the energies above for $\vect \tilf$ where $\tilf$ satisfies $\tilf=\Hone \tilf$. Also note that in our applications, the domain of the function $f$ is  $S_R$ rather than $\partial\calB_1$ (for instance when $f=u$). In this context we understand the notation $\calE_f$ as

\begin{align*}
\begin{split}
 \calE_f:=\calE_{f\circ\xi}. 
\end{split}
\end{align*}


\begin{remark}
Note that in the definition above $f$ was considered as a function on $\partial\calB_1$, whereas $a=\bfa\circ \xi$ is a function on $S_R$. So to be precise, we should replace $a$ by $\bfa$ in the definition of $\calE_f$, but by abuse of notation we use $a$ both as the function defined on $S_R$ and as $\bfa\vert_{\partial\calB_1}$.
\end{remark}

Before stating the main energy identity we record some integration-by-parts identities. For any two Clifford algebra-valued functions $f$ and $g$ define

\begin{equation}\label{eq: Q scalar}
Q(f,g):=\frac{1}{|N|}(f_\alpha g_\beta-f_\beta g_\alpha),
\end{equation}
Here $|N|=|\zeta_\alpha\cross\zeta_\beta|$ which makes $Q(f,g)$ coordinate-invariant. If $f$ and $g$ are vector-valued we also define 

\begin{equation}\label{eq: Q vector}
\vecQ(f,g):=\frac{1}{|N|}(f_\alpha\cross g_\beta-f_\beta\cross g_\alpha),
\end{equation}
 Note that when $f$ and $g$ are both scalars-valued  $Q(f,g)=-Q(g,f)$, and when they are both vector-valued $\vecQ(f,g)=\vecQ(g,f)$.


\begin{lemma}\label{lem: int by parts}
Let $Q$ and $\vecQ$ be defined as in \eqref{eq: Q scalar} and \eqref{eq: Q vector}.
\begin{enumerate}
\item If $f$, $g$, and $h$ are scalar-valued then

\begin{align*}
\begin{split}
 \int_{\partial\calB} Q(f,g) h dS=-\int_{\partial\calB}fQ(h,g)dS.
\end{split}
\end{align*}
\item If $f$, $g$, and $h$ are vector-valued then 

\begin{align*}
\begin{split}
 \int_{\partial\calB} \vecQ(f,g)\cdot h dS=\int_{\partial\calB}f\cdot \vecQ(h,g)dS.
\end{split}
\end{align*}
\end{enumerate}
\end{lemma}


\begin{proof}
In the scalar case the identity follows by writing 

\begin{align}\label{eq: Q d}
\begin{split}
 Q(f,g)dS= (f_\alpha g_\beta -f_\beta g_\alpha)d\alpha\wedge d\beta = df\wedge dg, 
\end{split}
\end{align}
where $d$ denotes the exterior differentiation operator on $\partial\calB_1$, and using Stokes' Theorem. In the vector case we write

\begin{align*}
\begin{split}
 \vecQ(f,g)\cdot h = \sum_{i,j,k=1}^3Q(f^i,g^j)h^k  (\bfe_i\cross \bfe_j)\cdot \bfe_k,
\end{split}
\end{align*}
and then apply the scalar identity.
\end{proof}


We are now ready to prove the main energy identity.


\begin{proposition}\label{prop: energy id}
Suppose $\tilf=\Hone \tilf$, and $f:=\vect \tilf$  satisfies \eqref{eq: model eq}. Then with $\calE$ as in \eqref{eq: model energy}, and $Q$ as in \eqref{eq: Q vector},

\begin{align}\label{eq: model energy id}
\begin{split}
\frac{d\calE}{dt}=&\left\langle g,\frac{f_t}{a}\right\rangle+\frac{1}{2}\left\langle \frac{1}{|N|}\partial_t\left(\frac{|N|}{a}\right)f_t,f_t\right\rangle-\frac{1}{2}\langle \vecQ(u,f),f\rangle+\frac{GM}{2R^3}\left\langle \frac{1}{|N|}\partial_t\left(\frac{|N|}{a}\right)f,f\right\rangle\\
&-\frac{3GM}{2R^3}\left\langle\frac{1}{|N|}\partial_t\left(\frac{|N|}{a}\right)f\cdot n,f\cdot n\right\rangle-\frac{3GM}{R^3}\left\langle a^{-1}f\cdot n,f\cdot n_t\right\rangle\\
 &+\frac{3GM}{2R^3}\langle ( n\cdot f) n,[\Hone,a^{-1}\partial_t]f\rangle+\frac{3GM}{2R^3}\langle (n\cdot f)n,(\Hone^\ast-\Hone)(a^{-1}f_t)\rangle\\
 &-\frac{3GM}{2R^{3}}\langle(n\cdot f)n,a^{-1}\partial_{t}(I-H_{\partial\calB_{1}})f\rangle.
\end{split}
\end{align}

\end{proposition}


\begin{proof}

We take the inner product of \eqref{eq: model eq} with $\frac{1}{a}f_t$ and study the terms on the left hand side one by one. The contributions of the first and third terms on the left hand side are clear, so we focus on the second and fourth terms. For the second term we have

\begin{align*}
\int_{\partial\calB_1}(n\cross\nabla f)\cdot f_t dS=&\iint (\zeta_\beta\cross f_\alpha-\zeta_\alpha\cross f_\beta)\cdot f_t\,d\alpha d\beta\\
=&\partial_t\int_{\partial\calB_1}(n\cross\nabla f)\cdot f dS-\iint (\zeta_\beta \cross f_{t\alpha}-\zeta_\alpha \cross f_{t\beta})\cdot f d\alpha d\beta\\
&-\iint (u_\beta\cross f_\alpha-u_\alpha\cross f_\beta)\cdot f d\alpha d\beta\\
=&\partial_t\int_{\partial\calB_1}(n\cross\nabla f)\cdot f dS+\int_{\partial\calB_1}\vecQ(u,f)\cdot f dS\\
&-\iint f_t \cdot (\zeta_\beta \cross f_\alpha - \zeta_\alpha \cross f_\beta ) d\alpha  d\beta,
\end{align*}
where in the last step we have used Lemma~\ref{lem: int by parts}. It follows that

\begin{align*}
\int_{\partial\calB_1}(n\cross\nabla f)\cdot f_t dS=\frac{1}{2}\partial_t\int_{\partial\calB_1}(n\cross\nabla f)\cdot f dS+\frac{1}{2}\int_{\partial\calB_1}\vecQ(u,f)\cdot f dS.
\end{align*}
Finally we consider the contribution from $(I+\Hone)((n\cdot f)n)$. We have

\begin{align*}
\begin{split}
 \langle(I+\Hone)((n\cdot f)n), a^{-1}f_t \rangle =&\langle (n\cdot f)n,(I+\Hone^\ast)(a^{-1}f_t)\rangle\\
 =&\langle (n\cdot f)n,(I+\Hone)(a^{-1}f_t)\rangle+\langle (n\cdot f)n,(\Hone^\ast-\Hone)(a^{-1}f_t)\rangle\\
 =&2\langle (n\cdot f)n,a^{-1}f_t\rangle+\langle (n\cdot f)n,[\Hone,a^{-1}\partial_t]f\rangle\\
 &-\langle (n\cdot f)n,a^{-1}\partial_{t}(I-H_{\partial\calB_{1}})f\rangle 
 +\langle (n\cdot f)n,(\Hone^\ast-\Hone)(a^{-1}f_t)\rangle\\
 =&\partial_t\langle a^{-1}f\cdot n,f\cdot n\rangle-\langle|N|^{-1}\partial_t(|N|a^{-1})f\cdot n,f\cdot n\rangle-2\langle a^{-1}f\cdot n,f\cdot n_t\rangle\\
 &+\langle (n\cdot f)n,[\Hone,a^{-1}\partial_t]f\rangle+\langle (n\cdot f)n,(\Hone^\ast-\Hone)(a^{-1}f_t)\rangle\\
 &-\langle (n\cdot f)n,a^{-1}\partial_{t}(I-H_{\partial\calB_{1}})f\rangle .
\end{split}
\end{align*}
The statement of the proposition now follows by combining the previous identities.

\end{proof}


\begin{remark}
Suppose $f=\sum_{i=1}^3 f^i\bfe_i$ is the vector part of a Clifford analytic function $\tilf$, $\tilf=\Hone \tilf$, and let $\circf=\Re\tilf$. Then from \eqref{eq: nabla cross dot n}

\begin{align*}
\begin{split}
  n\cross\nabla f=\sum_{i=1}^3(\nabla_{n}f^i)\bfe_i+\nabla_n\circf-n\cross\nabla \circf,
\end{split}
\end{align*}
where $\nabla_n$ denotes the Dirichlet-Neumann map of $\partial\calB_1$. It follows that in this case

\begin{align*}
\begin{split}
 \int_{\partial\calB_1}(n\cross\nabla f)\cdot f dS\geq -\int_{\partial\calB_1}(n\cross\nabla\circf)\cdot fdS.
\end{split}
\end{align*}

\end{remark}


The energy defined in \eqref{eq: model energy} is the natural energy associated to the time symmetry of the equation, but unfortunately it is not clear that it is in general positive definite. A similar problem was encountered in \cite{BMSW1}. As in \cite{BMSW1} here we will be able to show that $\calE_f$ is positive if $\calB_1$ is a small perturbation of $S_R$ (that is under our bootstrap assumptions) and for $f$ of interest. The approach we take here is more direct than the one in \cite{BMSW1}. The following  two general results are the first steps in this direction. The first result provides a lower bound on the constant for the Poincar\'e inequality, or equivalently the first nonzero Neumann eigenvalue of the positive Laplacian.


\begin{lemma}\label{lem: Poincare}\cite{Yang}
Let $D\subseteq\bbR^3$ be a simply connected bounded domain with diameter\footnote{The diameter $d$ of $D$ is by definition $d=\sup_{\bfx,\bfy\in D}|\bfx-\bfy|$.} $d$, such that the second fundamental form of $\partial D$ with respect to the exterior normal is non-negative. Then for any function $f$ satisfying $\int_D f(\bfx)d\bfx=0$ there holds

\begin{align*}
\|f\|_{L^2(D)}\leq \frac{d}{\pi}\|\nabla f\|_{L^2(D)}.
\end{align*}
\end{lemma}


\begin{remark}
We will later prove that the hypotheses of Lemma~\ref{lem: Poincare} are satisfied under the bootstrap assumptions to be introduced.
\end{remark}


The next result provides a lower bound for the constant of the trace embedding $H^1(D)\hookrightarrow L^2(\partial D)$.


\begin{lemma}\label{lem: trace}\cite{TraceTh}
Let $D\subseteq\bbR^3$ be a simply connected bounded domain with $C^1$ boundary $\partial D$, and let $\bfnu$ denote the exterior normal. Suppose $\bfmu$ is a $C^1$ vectorfield defined in a neighborhood of $D$ such that $\bfmu(\bfx)\cdot\bfnu(\bfx)\geq \frakb>0$ for all $\bfx\in\partial D$. Then for any function $\bff$ in $H^1(D)$, 

\begin{align*}
\frakb\int_{\partial D}|\bff(\bfx)|^2dS(\bfx)\leq \sup_{\barD}|\bfmu|^2\int_D|\nabla \bff(\bfx)|^2d\bfx+(1+\sup_{\barD}|\nabla\cdot\bfmu|)\int_D|\bff(\bfx)|^2d\bfx.
\end{align*}
\end{lemma}


\begin{proof}

Because the statement is slightly different from the one in \cite{TraceTh}, we provide the proof which is directly from  \cite{TraceTh}. By the divergence theorem

\begin{align*}
2\int_{D}\bff\nabla \bff\cdot\bfmu d\bfx=\int_{D}\nabla(\bff^2)\cdot\bfmu d\bfx=-\int_{D}\bff^2\nabla \cdot\bfmu d\bfx+\int_{\partial D} \bff^2 \bfmu\cdot\bfnu dS.
\end{align*}
The desired estimate follows by rearranging and applying Cauchy-Schwarz.

\end{proof}


\begin{remark}
In our applications $\partial D$ will be close to $S_R$, in which case $\bfnu(\bfx)=R^{-1}\bfx$. The choice $\bfmu(\bfx)=\bfx$ then gives $b=R^{-1}$.
\end{remark}


To use \lem{lem: Poincare} to prove lower bounds for the energies for $\vecu_k$ we also need to show that the average of $u_k$ is small for all $k$. We will use the following notation for a Clifford algebra-valued function $\bff$ defined on $\calB_1$:

\begin{align*}
\begin{split}
  \AV(\bff):=\frac{1}{M}\int_{\calB_1}\bff dx.
\end{split}
\end{align*}


\begin{lemma}\label{lem: averages}
Let $u_k:=\frac{1}{2}(I+\Hone)\frakD^ku$, and let $\bfu_k$ be the Clifford analytic extension of $u_k$ to $\calB_1$, with $u_0=u$ and $\bfu_0=\bfu$. Then

\begin{align}\label{eq: uk av 1}
\begin{split}
 \AV (\bfu_k)=\frac{R}{3M}\int_{\partial\calB_1} u_k dS-\frac{1}{3M}\int_{\partial\calB_1}(\zeta -Rn) n u_k dS+\frac{1}{6M}\int_{\partial\calB_1}((I+\Hone)\zeta)nu_kdS,\qquad k\geq0,
\end{split}
\end{align}
and

\begin{align}\label{eq: uk av 2}
\begin{split}
 \int_{\partial\calB_1} u_k dS=\int_{\partial\calB_1}u_{k-1}\frac{\Omega(|N||\sg|^{-1})}{|N||\sg|^{-1}}dS-\bfe\cross\vect\int_{\partial\calB_1}u_{k-1}dS
 -\frac{1}{2}\int_{\partial\calB_1}[\frakD,\Hone]\frakD^{k-1}u\,dS, \qquad k\geq1.
\end{split}
\end{align}
\end{lemma}


\begin{remark}\label{rmk: application of average}
This lemma can be used inductively as follows. First, since\footnote{By definition, $\bfu=\bfv-\bfx'_{1}$. On the other hand, $\bfx'_{1}=\partial_{t}\left(\frac{1}{|\calB_{1}|}\int_{\calB_{1}}\bfx d\bfx\right)=\frac{1}{|\calB_{1}|}\int_{\calB_{1}}\bfv d\bfx$.} $\AV(\bfu)=0$, we can use \eqref{eq: uk av 1} to control $\int_{\partial\calB_1}u_0dS$. Then we use \eqref{eq: uk av 2} to bound  $\int_{\partial\calB_{1}}u_{1}dS$. We then use \eqref{eq: uk av 1} to estimate $\AV(\bfu_1)$ in terms of $\int_{\partial\calB_1}u_1 dS$, which we in turn estimate in terms of $\int_{\partial\calB_1}u_{0}dS$ according to \eqref{eq: uk av 2}. This process can now be continued inductively to estimate $\AV(\bfu_k)$ for all $k$. We summarize this process in the following chart:

\begin{align*}
 &0=\AV(\bfu_{0})\quad\rightarrow\quad \int_{\partial\calB_{1}}u_{0}dS\quad \rightarrow\quad \int_{\partial\calB_{1}}u_{1}dS\quad \rightarrow\quad \int_{\partial\calB_{1}}u_{2}dS\quad \rightarrow...\rightarrow\quad \int_{\partial\calB_{1}}u_{k}dS\quad \rightarrow...,\\
 &\int_{\partial\calB_{1}}u_{k}dS\quad \rightarrow\quad \AV(\bfu_{k})\quad \textrm{for}\quad k\geq1.
\end{align*}
\end{remark}


\begin{proof}[Proof of Lemma~\ref{lem: averages}]
Since $u_k$ is Clifford analytic, by \thm{thm: Cuachy}, for any $\bfxi'\in\calB_1$

\begin{align*}
\begin{split}
 \bfu_k(\bfxi')=\frac{1}{4\pi}\int_{\partial\calB_1} \frac{\bfxi'-\bfxi}{|\bfxi'-\bfxi|^3}\bfn(\bfxi)\bfu_k(\bfxi)dS(\bfxi),
\end{split}
\end{align*}
so

\begin{align*}
\begin{split}
 \AV(\bfu_k) =\frac{1}{4\pi M}\int_{\partial\calB_1}\left(\int_{\calB_1}\frac{\bfxi'-\bfxi}{|\bfxi'-\bfxi|^3}d\bfxi'\right)\bfn(\bfxi)\bfu_k(\bfxi)dS(\bfxi).
\end{split}
\end{align*}
Notice that the inner integral is equal to $-\frac{1}{G\rho}\nabla\bfpsi_1(\bfxi)$. So using Lemma~\ref{lem: nabla psi},

\begin{align*}
\begin{split}
 \AV(\bfu_k)=&-\frac{1}{3M}\int_{\partial\calB_1}\zeta n u_k dS+\frac{1}{6M}\int_{\partial\calB_1}((I+\Hone)\zeta)nu_kdS\\
 =&\frac{R}{3M}\int_{\partial\calB_1} u_k dS-\frac{1}{3M}\int_{\partial\calB_1}(\zeta -Rn) n u_k dS+\frac{1}{6M}\int_{\partial\calB_1}((I+\Hone)\zeta)nu_kdS.
\end{split}
\end{align*}
This proves \eqref{eq: uk av 1}, and \eqref{eq: uk av 2} follows from the identity 

\begin{align*}
\begin{split}
  u_k :=\frac{1}{2}(I+\Hone)\frakD^ku=\Omega u_{k-1}-\bfe\cross \vect u_{k-1}-\frac{1}{2}[\frakD,\Hone]\frakD^{k-1}u, \qquad k\geq1,
\end{split}
\end{align*}
and an integration by parts according to Lemma~\ref{lem: product rule}.
\end{proof}


Using Lemmas~\ref{lem: Poincare} and~\ref{lem: trace} we can prove the positivity of the energy to leading order. In fact we will show that $\calE_f$ controls the $L^2$ norm of $f$.


\begin{lemma}\label{lem: pos energy}
Suppose the second fundamental form of $\partial\calB_1$ with respect to the exterior normal $\bfn$ is non-negative. Suppose further that there is $\delta\in[0,1)$ such that $\zeta$ satisfies $\zeta\cdot n\geq (1-\delta)R$ and such that for all $\bfx\in\calB_1$, $|\bfx-\bfx_1|\leq (1+\delta)R$. Let $\bff=\vect\tilde{\bff}$ and $\bff_0=\Re\tilde{\bff}$, where $\tilde{\bff}$ is a Clifford analytic function in $\calB_1$, and let $f$ and $f_0$ be the restrictions of $\bff$ and $\bff_0$ to $\partial\calB_1$. If $\delta$ is sufficiently small and $f$ satisfies the assumptions in Proposition~\ref{prop: energy id},
 then there exist absolute constants $c,C>0$ such that

\begin{align}\label{eq: pos energy 2}
\begin{split}
\frac{1}{2}\int_{\partial\calB_1}(\bfn\cross\nabla f)\cdot f dS+&\frac{1}{2R}\int_{\partial\calB_1}|f|^2dS-\frac{3}{2R}\int_{\partial\calB_1}(n\cdot f)^2dS\\
\geq& \,\frac{c}{4R}\int_{\partial\calB_1}|f|^2dS-\frac{3}{2R}\int_{\partial\calB_1}(f\cdot n)f\cdot(n-R^{-1}\zeta)dS-\frac{CM\rho}{R^2}(\AV(\bff))^2\\
&-\frac{3}{4R}\int_{\partial\calB_1}|f|^2(R^{-1}\zeta-n)\cdot n \,dS-\frac{3}{2R^2}\int_{\partial\calB_{1}}(f\times \zeta)\cdot(n-R^{-1}\zeta)f_{0}dS\\
&-\frac{9}{4R^2}\int_{\calB_{1}}f^{2}_{0}d\bfx+\frac{3}{4R^2}\int_{\partial\calB_{1}}(\zeta\cdot n)f^{2}_{0}dS-CR\int_{\partial\calB_1}|n\cross\nabla f_0|^2dS.
\end{split}
\end{align}

Moreover,
 
 \begin{align}\label{eq: f0 L2}
\begin{split}
 \int_{\calB_1}|\bff_0|^2d\bfx\leq CR\int_{\partial\calB_1}|f_0|^2dS+ CR^3 \int_{\partial\calB_1}|\nabla_n f_0|^2dS+CM\rho(\AV(\bff_0))^2. 
\end{split}
\end{align}

\end{lemma}


\begin{proof}
We start by proving \eqref{eq: pos energy 2}. Let $\bff=\vect\tilde{\bff}$, where  $\tilde{\bff}$ is the Clifford analytic extension of $\tilf$ to the interior of $\calB_1$, and let $\bfzeta=\bfx-\bfx_1$ be the harmonic extension of $\zeta$ to the interior. 
First note that since $n\cross\nabla f=\nabla_nf+\nabla_{n}f_{0}-n\times \nabla f_{0}$ and $\bff$ has harmonic components,

\begin{align*}
\begin{split}
 \frac{1}{2}\int_{\partial\calB_1} (n\cross\nabla f)\cdot f dS=\frac{1}{2}\int_{\calB_1}|\nabla \bff|^2 d\bfx-\frac{1}{2}\int_{\partial\calB_{1}}(n\times\nabla f_{0})\cdot fdS,
\end{split}
\end{align*}
where $|\nabla\bff|^2=\sum_{i=1}^3|\nabla\bff^i|^2$. Next,  by assumption $\partial_i\bff^j-\partial_j\bff^i=\epsilon_{ijk}\partial_k\bff_0$, where $\epsilon_{ijk}=1$ if $(ijk)$ is an even permutation of $(123)$, $\epsilon_{ijk}=-1$ if $(ijk)$ is an odd permutation of $(123)$, and $\epsilon_{ijk}=0$ otherwise. Combining this with the facts that $\nabla\cdot\bff=0$ and $\nabla\cdot\bfzeta=3$, we write

\begin{align}\label{eq: pos energy temp 1}
\begin{split}
 \int_{\partial\calB_1}(n\cdot f)^2dS=&\frac{1}{R}\int_{\partial\calB_1} (f\cdot\zeta)(f\cdot n) dS+\int_{\partial\calB_1}(f\cdot n)f\cdot(n-R^{-1}\zeta)dS\\
 =&\frac{1}{R}\int_{\calB_1}\nabla\cdot((\bff\cdot \bfzeta)\bff)d\bfx+\int_{\partial\calB_1}(f\cdot n)f\cdot(n-R^{-1}\zeta)dS\\
 =&\frac{1}{R}\int_{\calB_1}|\bff|^2d\bfx+\frac{1}{2R}\int_{\calB_1}\bfzeta\cdot\nabla|\bff|^2d\bfx+\int_{\partial\calB_1}(f\cdot n)f\cdot(n-R^{-1}\zeta)dS\\
 &+\frac{1}{R}\int_{\calB_{1}}\bff^i\bfzeta^j\epsilon_{ijk}\partial_k f_{0}d\bfx\\
 =&-\frac{1}{2R}\int_{\calB_1}|\bff|^2d\bfx+\frac{1}{2R}\int_{\calB_1}\nabla\cdot(|\bff|^2\bfzeta)d\bfx+\int_{\partial\calB_1}(f\cdot n)f\cdot(n-R^{-1}\zeta)dS\\
 &+\frac{1}{R}\int_{\calB_{1}}\partial_{k}(\bff^{i}\bfzeta^{j}\epsilon_{ijk}f_{0})d\bfx-\frac{1}{R}\int_{\calB_{1}}\partial_{k}\bff^{i}\bfzeta^{j}\epsilon_{ijk}f_{0}d\bfx\\
 =&-\frac{1}{2R}\int_{\calB_1}|\bff|^2d\bfx+\frac{1}{2R}\int_{\partial\calB_1}|f|^2\zeta\cdot n dS+\int_{\partial\calB_1}(f\cdot n)f\cdot(n-R^{-1}\zeta)dS\\
 &+\frac{1}{R}\int_{\partial\calB_{1}}(f\times \zeta)\cdot(n-R^{-1}\zeta)f_{0}dS-\frac{1}{2R}\int_{\calB_{1}}\epsilon_{ki\ell}\epsilon_{kij}\bfzeta^{j}f_{0}\partial_{\ell}f_{0}d\bfx\\
 =&-\frac{1}{2R}\int_{\calB_1}|\bff|^2d\bfx+\frac{1}{2}\int_{\partial\calB_1}|f|^2dS\\
 &+\frac{1}{2}\int_{\partial\calB_1}|f|^2(R^{-1}\zeta-n)\cdot n dS+\int_{\partial\calB_1}(f\cdot n)f\cdot(n-R^{-1}\zeta)dS\\
 &+\frac{1}{R}\int_{\partial\calB_{1}}(f\times \zeta)\cdot(n-R^{-1}\zeta)f_{0}dS+\frac{3}{2R}\int_{\calB_{1}}f_{0}^{2}d\bfx-\frac{1}{2R}\int_{\partial\calB_{1}}(\zeta\cdot n)f^{2}_{0}dS.
\end{split}
\end{align}
To pass to the last equality, we have used the face that

\begin{align*}
-\frac{1}{2}\epsilon_{ki\ell}\epsilon_{kij}\bfzeta^j\partial_\ell\bff_0^2=&-\frac{1}{2}\partial_\ell(\bfzeta^j\bff_0^2\epsilon_{ki\ell}\epsilon_{kij})+3\bff_0^2=-\nabla\cdot(\bff_0^2\bfzeta)+3\bff_0^2.
\end{align*}
Combining this with the previous identity, and for any constant $c\in(0,1)$, we get

\begin{align}\label{eq: pos energy temp 2}
\begin{split}
 \frac{1}{2}\int_{\partial\calB_1}(\bfn\cross\nabla f)\cdot f dS+&\frac{1}{2R}\int_{\partial\calB_1}|f|^2dS-\frac{3}{2R}\int_{\partial\calB_1}(n\cdot f)^2dS\\
 &= \frac{1}{2}\int_{\calB_1}|\nabla \bff|^2 d\bfx+\frac{3}{4R^2}\int_{\calB_1}|\bff|^2d\bfx-\frac{1+c}{4R}\int_{\partial\calB_1}|f|^2dS+\frac{c}{4R}\int_{\partial\calB_1}|f|^2dS\\
 &-\frac{3}{4R}\int_{\partial\calB_1}|f|^2(R^{-1}\zeta-n)\cdot n dS-\frac{3}{2R}\int_{\partial\calB_1}(f\cdot n)f\cdot(n-R^{-1}\zeta)dS\\
 &-\frac{1}{2}\int_{\partial\calB_{1}}(n\times\nabla f_{0})\cdot f dS-\frac{3}{2R^2}\int_{\partial\calB_{1}}(f\times \zeta)\cdot(n-R^{-1}\zeta)f_{0}dS\\
 &-\frac{9}{4R^2}\int_{\calB_{1}}f^{2}_{0}d\bfx+\frac{3}{4R^2}\int_{\partial\calB_{1}}(\zeta\cdot n)f^{2}_{0}dS.
\end{split}
\end{align}
Now applying Lemma~\ref{lem: trace} with $D=\calB_1$ and $\bfmu=\bfzeta$ we get

\begin{align*}
\begin{split}
 (1-\delta)R\int_{\partial\calB_1}|f|^2dS\leq \sup_{\calB_1}|\bfzeta|^2\int_{\calB_1}|\nabla\bff|^2d\bfx+4\int_{\calB_1}|\bff|^2d\bfx \leq (1+\delta)^2R^2\int_{\calB_1}|\nabla\bff|^2d\bfx+4\int_{\calB_1}|\bff|^2d\bfx.
\end{split}
\end{align*}
It follows that

\begin{align*}
\begin{split}
  \frac{1}{2}\int_{\calB_1}|\nabla \bff|^2 d\bfx+&\frac{3}{4R^2}\int_{\calB_1}|\bff|^2d\bfx-\frac{1+c}{4R}\int_{\partial\calB_1}|f|^2dS \\
  \geq& \frac{2-(1+\delta)^2(1+c)(1-\delta)^{-1}}{4}\int_{\calB_1}|\nabla \bff|^2 d\bfx\\
  &-\frac{(4(1+c)(1-\delta)^{-1}-3)(1+\varepsilon)}{4R^{2}}\int_{\calB_1}|\bff-\rho\AV(\bff)|^2d\bfx-\frac{CM\rho}{R^2}(\AV(\bff))^{2},
\end{split}
\end{align*}
where $\varepsilon>0$ is a small constant to be chosen, and $C>0$ is an absolute constant depending only on $c$, $\varepsilon$, and $\delta$.
Now since by assumption $|\bfx-\bfx_1|<(1+\delta)R$ for all $\bfx\in\calB_1$, the diameter $d$ of $\calB_1$ satisfies $d\leq 2(1+\delta)R$. Therefore applying the Poincar\'e estimate in Lemma~\ref{lem: Poincare} to the right-hand side of the estimate above, and choosing $\delta$, $\varepsilon$, and $c$ sufficiently small we get

\begin{align*}
\begin{split}
   \frac{1}{2}\int_{\calB_1}|\nabla \bff|^2 d\bfx+&\frac{3}{4R^2}\int_{\calB_1}|\bff|^2d\bfx-\frac{1+c}{4R}\int_{\partial\calB_1}|f|^2dS \geq -\frac{CM\rho}{R^{2}}(\AV(\bff))^{2}.
\end{split}
\end{align*}
Identity \eqref{eq: pos energy 2} follows from plugging this back into \eqref{eq: pos energy temp 2} and applying Cauchy-Schwarz to $\int_{\partial\calB_1}(n\cross\nabla f_0)\cdot f dS$.

We next prove \eqref{eq: f0 L2}.  Note that by the Poincar\'e estimate in Lemma~\ref{lem: Poincare}
  
 \begin{align*}
\begin{split}
  \int_{\calB_1}|\bff_0|^2d\bfx=&\int_{\calB_1}|\bff_0-\rho\AV(\bff_0)+\rho\AV(\bff_0)|^2d\bfx\leq 2\int_{\calB_1}|\bff_0-\rho\AV(\bff_0)|^2d\bfx+2\rho^2\int_{\calB_1}|\AV(\bff_0)|^2d\bfx\\
  \lesssim&R^2\int_{\calB_1}|\nabla\bff_0|^2d\bfx+M\rho(\AV(\bff_0))^2.
\end{split}
\end{align*}
But since $\bff_0$ is harmonic in $\calB_1$,

\begin{align*}
\begin{split}
  \int_{\calB_1}|\nabla\bff_0|^2d\bfx=\int_{\partial\calB_1}f_0\nabla_nf_0 dS\leq \frac{1}{2R}\int_{\partial\calB_1}|f_0|^2dS+\frac{R}{2}\int_{\partial\calB_1}|\nabla_nf_0|^2dS,
\end{split}
\end{align*}
proving \eqref{eq: f0 L2}.

\end{proof}


Since $a$ is equal to $\frac{GM}{R^2}$ to leading order (see Lemma \ref{lem: a}), and $\bfu$ has zero average in $\calB_1$, we can use Lemma~\ref{lem: pos energy} to prove positivity of the energy $\calE_f$ for $f=u$.


 We close this section by defining the main energies, motivated by the discussion above.


\begin{definition}\label{def: energies}
Let $u_0:=u$ and $u_\alpha:=\frac{1}{2}(I+\Hone)\frakD^\alpha u$ for any multi-index $\alpha$ with $|\alpha|\geq1$. We then define

\begin{align*}
\begin{split}
 \calE_j:=\sum_{|\alpha|=j}\calE_{\vecu_\alpha}, \quad \mand\quad \calE_{\leq k}=\sum_{j=1}^k\calE_{j},
\end{split}
\end{align*}
where $\calE_f$ is defined in \eqref{eq: model energy}.
\end{definition}


\subsection{Estimates on the Center of Mass $\bfx_1$ and Angular Momentum $\bfJ$} \label{subsec: x_1 J}

 In the next section we pose our bootstrap assumptions in terms of the distance $|\bfx_1|$ between the center of mass of each body to the origin, which is the center of mass of the entire system. We expect that as long as this distance is large, the size of $\bfx_1$ and its time derivatives can be approximated by the corresponding quantities in the motion of point masses. Our goal in this section is to make this claim rigorous by providing quantitative estimates. Recall once more the definitions

\begin{align*}
r_1:=|\bfx_1|,\quad \bfxi_1=\frac{\bfx_1}{|\bfx_1|},\quad \eta=\frac{R}{|\bfx_1|},\quad \beta=\frac{b}{R}.
\end{align*}
We also define the angular momentum vector $\bfJ$ as

\begin{align*}
\bfJ:=\bfx_1\cross\bfx_1',
\end{align*}
and the velocity $v_1$ as $v_1=\bfx_1'$ so that

\begin{align}\label{eq: v1}
|v_1|^2=|\bfx_1'|^2=|r_1'|^2+\frac{J^2}{r_1^2},
\end{align}
where $J=|\bfJ|$.
Recall from Lemma \ref{lem: acce x1} that $\bfx_1$ satisfies the ODE
 
 \begin{align}\label{eq: x1'' 2}
\bfx_1''=-\frac{1}{|\calB_{1}|}\int_{\calB_{1}}\nabla\bfpsi_2(t,\bfx)d\bfx=-\frac{GM\eta^2}{4R^2}\bfxi_1-\frac{1}{|\calB_{1}|}\int_{\calB_{1}}\bfE_1(t,\bfx)d\bfx=-\frac{GM\bfx_1}{4|\bfx_1|^3}-\frac{1}{|\calB_{1}|}\int_{\calB_{1}}\bfE_1(t,\bfx)d\bfx,
\end{align}
 where $\bfE_1$ is defined in \eqref{eq: bfE1}. In the remainder of this section we write 
 
 \begin{align*}
 \AVE(\bfE_1):=\frac{1}{|\calB_{1}|}\int_{\calB_{1}}\bfE_1(t,\bfx)d\bfx.
\end{align*}
It follows from a simple computation that
 
 \begin{align}\label{eq: r1''}
r_1''=-\frac{GM}{4|\bfx_1|^2}+\frac{|\bfx_1\cross \bfx_1'|^2}{|\bfx_1|^3}-\bfxi_1\cdot\AVE(\bfE_1)=-\frac{GM}{4|\bfx_1|^2}+\frac{J^2}{|\bfx_1|^3}-\bfxi_1\cdot \AVE(\bfE_1).
\end{align}
 The point mass energy $\scE_1$ of $\bfx_1$ is by definition
 
 \begin{align*}
\scE_1:=\frac{|v_1|^2}{2}-\frac{GM}{4r_1},
\end{align*}
so 

\begin{align}\label{eq: r1'}
|r_1'|^2=2\scE_1+\frac{GM}{2r_1}-\frac{J^2}{r_1^2}.
\end{align}
To be consistent with the assumptions of \thm{thm: tidal capture}, we write the initial velocity as

\begin{align*}
\begin{split}
 v_0=|\bfx_1'(T_0))|=\kappa\beta^{-\alpha}\sqrt{\frac{GM}{R}}, 
\end{split}
\end{align*}
where $\alpha\in[\frac{6}{7},1]$,  $\kappa^2\beta^{2-2\alpha}\gg1$, and $\kappa^{-14}\beta^{14\alpha-12}\gg1$ (see \rem{rem: initial data})\footnote{In practice we first choose $\kappa$ and then $\beta$ depending on $\kappa$ so that both conditions are satisfied. For instance, when $\alpha=1$ we choose $\kappa\gg1$ and $\beta$ such that $\beta^2\gg \kappa^{14}$. When $\alpha=\frac{6}{7}$ we choose $\kappa\ll1$ and $\beta$ such that $\beta^{\frac{2}{7}}\gg \kappa^{-2}$.}. To unify the notation, we introduce the new parameter

\begin{align*}
c_0:=\kappa \beta^{\frac{6}{7}-\alpha},
\end{align*}
so that $c_0\ll 1$ for all choices of $\alpha\in[\frac{6}{7},1]$ and

\begin{align*}
v_0=c_0\beta^{-\frac{6}{7}}\sqrt{\frac{GM}{R}}.
\end{align*}
This implies that

\begin{align}\label{eq: initial energies}
\begin{split}
 J^2(T_0)=c_0^2\beta^{\frac{2}{7}}GMR,\qquad 2\scE_1(T_0)= c_0^2\beta^{-\frac{12}{7}}\frac{GM}{R}-\frac{GM}{2R_1}.
\end{split}
\end{align}
Unlike the point-mass case, the energy $\scE_1$ and angular momentum $\bfJ$ are not conserved and instead satisfy the evolution laws

\begin{align}\label{eq: J'}
\bfJ'=\AVE(\bfE_1)\cross\bfx_1,\qquad \frac{d}{dt}J^2=2\left((\AVE(\bfE_1)\cdot\bfx_1)(\bfx_1\cdot\bfx_1')-(\AVE(\bfE_1)\cdot\bfx_1')|\bfx_1|^2\right),
\end{align}
and

\begin{align}\label{eq: scE'}
\scE_1'=-\AVE(\bfE_1)\cdot \bfx_1'.
\end{align}
In the following proposition we derive estimates on the radial velocity $|r_1'|$. Since $r_1'$ is initially negative, by continuity, it will remain negative for $t$ close to $T_0$.  It follows that there exists $\tilr$ such that for any $T>T_0$
\begin{align}\label{eq: rtil}
\begin{split}
  r(t)\geq \tilr\quad \forall t\leq T,\qquad \mathrm{implies} \qquad r_1'(t)<0 \quad\forall t\leq T.
\end{split}
\end{align}
We define $r_0$ to be the smallest such $\tilr$, that is, 

\begin{align*}
\begin{split}
 r_0:=\inf\{\tilr\quad\mathrm{s.t.~}\eqref{eq: rtil}\mathrm{~holds}\}, 
\end{split}
\end{align*}
and let $t_0$, if it exits, be the first time such that $r_1(t_0)=r_0$. Then by finding the roots of the quadratic polynomial in $r_1^{-1}$ in \eqref{eq: r1'} we see that $r_0$, if it exists, satisfies, 

\begin{align}\label{eq: r0}
\begin{split}
  r_0=\frac{4J^2(t_0)}{\sqrt{G^2M^2+32\scE_1(t_0)J^2(t_0)}+GM},
\end{split}
\end{align}
whenever $G^{2}M^{2}+32\scE_{1}(t_{0})J^{2}(t_{0})\geq 0$.
We prove the estimates on $r_1$ and its derivative under the mild assumption that the diameter of the second body $\calB_2$ is bound by a constant multiple of $R$. When we apply these estimates in the proof of energy estimates, the diameter of $\calB_2$ will in fact be close to $2R$ under our bootstrap assumptions.


\begin{proposition}\label{prop: r1' v1} The following statements hold if a solution to \eqref{eq: Euler} exists and the diameter of $\calB_2$ is no larger than $10R$. 

\begin{enumerate}

\item  If $\beta$ is sufficiently large, then for some universal constant $C$ and all $r\geq r_0$

\begin{align}\label{eq: v1'}
\begin{split}
\frac{GM}{C|r_1|^2}\leq |v_1'| \leq \frac{CGM}{|r_1|^2}.
\end{split}
\end{align} 

\item If $\beta$ is sufficiently large, then $\frac{3}{2}c_0^2R\beta^{\frac{2}{7}}\leq r_0\leq \frac{5}{2}c_0^2R\beta^{\frac{2}{7}}$ and $r_1'$ satisfies the following estimates for some universal constant $C$

\begin{align}\label{eq: r1' bounds}
\begin{split}
  \frac{1}{C}\sqrt{\frac{GM}{R}}c_0\beta^{-\frac{6}{7}}\leq |r_1'|\leq C\sqrt{\frac{GM}{R}}c_0\beta^{-\frac{6}{7}}\qquad &\mathrm{if}\qquad  r_1\geq Rc_0^{-2}\beta^{\frac{12}{7}},\\
  \frac{1}{C}\sqrt{\frac{GM}{R}}\eta^{\frac{1}{2}}\leq |r_1'|\leq C\sqrt{\frac{GM}{R}}\eta^{\frac{1}{2}}\qquad &\mathrm{if}\qquad 3Rc_0^2\beta^{\frac{2}{7}}\leq r_1\leq Rc_0^{-2}\beta^{\frac{12}{7}},\\
 \frac{1}{C}\frac{\sqrt{GM}}{c_0^2R}\beta^{-\frac{2}{7}}\sqrt{r_1-r_{0}}\leq |r_1'|\leq C\frac{\sqrt{GM}}{c_0^2R}\beta^{-\frac{2}{7}}\sqrt{r_1-r_{0}}\qquad &\mathrm{if}\qquad r_0<r_1\leq 4c_0^2R\beta^{\frac{2}{7}}.
\end{split}
\end{align}

\item  If $\beta$ is sufficiently large, $v_1$ satisfies the following estimates

\begin{align}\label{eq: v1 bounds}
\begin{split}
  \frac{1}{C}\sqrt{\frac{GM}{R}}c_0\beta^{-\frac{6}{7}}\leq |v_1|\leq C\sqrt{\frac{GM}{R}}c_0\beta^{-\frac{6}{7}}\qquad &\mathrm{if}\qquad  r_1\geq Rc_0^{-2}\beta^{\frac{12}{7}},\\
  \frac{1}{C}\sqrt{\frac{GM}{R}}\eta^{\frac{1}{2}}\leq |v_1|\leq C\sqrt{\frac{GM}{R}}\eta^{\frac{1}{2}}\qquad &\mathrm{if}\qquad 3c_0^2R\beta^{\frac{2}{7}}\leq r_1\leq Rc_0^{-2}\beta^{\frac{12}{7}},\\
 \frac{c_0^{-1}}{C}\sqrt{\frac{GM}{R}}\beta^{-\frac{1}{7}}\leq |v_1|\leq Cc_0^{-1}\sqrt{\frac{GM}{R}}\beta^{-\frac{1}{7}}\qquad &\mathrm{if}\qquad r_0<r_1\leq 4c_0^2R\beta^{\frac{2}{7}}.
\end{split}
\end{align}
\end{enumerate}
\end{proposition}


\begin{proof}

Note that under the assumption on the diameter of $\calB_2$ and by \eqref{eq: bfE1}, there exists a universal constant $c$ such that

\begin{align}\label{eq: r1' v1 temp 1}
\begin{split}
 |\bfE_1|\leq \frac{cGM\eta^4}{R^2},\quad \Rightarrow\quad\left|\frac{1}{|\calB_{1}|}\int_{\calB_{1}}\bfE_{1}(t,\bfx)d\bfx\right|\leq \frac{cGM\eta^4}{R^2}.
\end{split}
\end{align}
Combined with \eqref{eq: J'} and \eqref{eq: scE'} this implies that

\begin{align}\label{eq: r1' v1 temp 22}
\begin{split}
 \left|\frac{d\scE_1}{dt}\right|\leq \frac{cGM\eta^4|v_1|}{R^2}\quad\mand \quad \left|\frac{d}{dt}J^2\right|\leq \frac{cGM\eta^4r_1^2|v_1|}{R^2}.
\end{split}
\end{align}
The estimate on $|v_1'|$ is a direct consequence of \eqref{eq: x1'' 2} and \eqref{eq: r1' v1 temp 1}. Next we prove the estimates on $|r_1'|$ and $|v_1|$ in the first two stages, that is, when $r\geq 4Rc_0^2\beta^{ \frac{2}{7}  }$. For this we rewrite \eqref{eq: v1} and \eqref{eq: r1'} as

\begin{align}\label{eq: r1' new}
\begin{split}
|r_1'(t)|^2=2\scE_1(T_0)+\frac{GM}{2r_1(t)}-\frac{J^2(T_0)}{r_1^2(t)}+2\int_{T_0}^t\frac{d\scE_1}{ds}ds-\frac{1}{r_1^2(t)}\int_{T_0}^t\frac{d}{ds}J^2ds,  
\end{split}
\end{align}
and

\begin{align}\label{eq: v1 new}
\begin{split}
 |v_1(t)|^2=|r_1'(t)|^2+\frac{J^2(T_0)}{r_1^2(t)}+\frac{1}{r_1^2(t)}\int_{T_0}^t\frac{d}{ds}J^2ds. 
\end{split}
\end{align}
If $R_1$ is sufficiently large, it follows from \eqref{eq: initial energies} that the desired estimates hold at $t=T_0$ with $C=1+10^{-10}$. Therefore, it suffices to assume that the estimates on $|v_1|$ and $|r_1'|$ hold for $r_1\geq r_0\geq c_0^{-2}\beta^{12/7}R$ with $C=20$, and show that the same estimates hold with $C=10$. 
To prove this we
show that the contribution of the time integrals of $\left|\frac{d\scE_1}{dt}\right|$ and $\left|\frac{d}{dt}J^2\right|$ to \eqref{eq: r1' new} and \eqref{eq: v1 new} can be estimated. If $r_1\geq r_0\geq Rc_0^{-2}\beta^{\frac{12}{7}}$ it follows from \eqref{eq: r1' v1 temp 22} and the bootstrap assumption on $|v_1|$ and $|r_1'|$ that

\begin{align}\label{eq: r1' v1 temp 2}
\begin{split}
  \left|\frac{d\scE_1}{dt}\right|\lesssim \frac{GM\eta^4|r_1'|}{R^2}\quad\mand \quad \left|\frac{d}{dt}J^2\right|\lesssim \frac{GM\eta^4r_1^2|r_1'|}{R^2},
\end{split}
\end{align}
so

\begin{align*}
\begin{split}
 \int_{T_0}^t \left|\frac{d\scE_1}{dt}\right| dt=\int_{r_1(t)}^{r_1(T_0)}\left|\frac{d\scE_1}{dt}\right| \left|\frac{dt}{dr_1}\right|dr_1\lesssim \frac{GM\eta^3}{R}\lesssim \frac{c_0^6GM\beta^{-\frac{36}{7}}}{R},
\end{split}
\end{align*}
and similarly

\begin{align*}
\begin{split}
 \int_{T_0}^t  \left|\frac{d}{dt}J^2\right| dt\lesssim GMR\eta\lesssim c_0^2GMR \beta^{-\frac{12}{7}}.
\end{split}
\end{align*}
If $\beta$ is sufficiently large these estimates allow us to close the bootstrap assumption in the region $r\geq Rc_0^{-2}\beta^{\frac{12}{7}}$. The analysis in the second region $3Rc_0^2\beta^{\frac{2}{7}}\leq r_1\leq Rc_0^{-2}\beta^{\frac{12}{7}}$ is similar.  Indeed, the estimates at $r=Rc_0^{-2}\beta^{\frac{12}{7}}$ are satisfied, because we have already proved the estimates in the first stage with $C=10$. So as in the first stage, we assume that the estimates on $|v_1|$ and $|r_1'|$ hold with $C=20$ and improve this to $C=15$. Since \eqref{eq: r1' v1 temp 22} are still valid in the region region $3Rc_0^2\beta^{\frac{2}{7}}\leq r_1\leq Rc_0^{-2}\beta^{\frac{12}{7}}$,
\begin{align*}
\begin{split}
 &\int_{T_0}^t \left|\frac{d\scE_1}{dt}\right| dt\lesssim \frac{GM\eta^3}{R}\qquad\mand \qquad \int_{T_0}^t  \left|\frac{d}{dt}J^2\right| dt\lesssim GMR\eta. 
\end{split}
\end{align*}
It follows that if $3Rc_0^{2}\beta^{\frac{2}{7}}\leq r_1(t)\leq Rc_0^{-2}\beta^{\frac{12}{7}}$

\begin{align}\label{bounds on E and J}
\begin{split}
 \left| 2\scE_1(t)-\frac{GMc_0^{2}\beta^{-\frac{9}{5}}}{R}\right|\lesssim \frac{GM\eta^3}{R},\qquad c_0^2GMR\beta^{\frac{2}{7}}-CGMR\eta\leq J^2(t)\leq c_0^2GMR\beta^{\frac{2}{7}}+CGMR\eta.
\end{split}
\end{align}
Combining this with \eqref{eq: v1} and \eqref{eq: r1'} we can close the bootstrap assumption in the second region. Here we have used the fact that in the second stage $GM(2r_1)^{-1}-J^2(T_0)r_1^{-2}$ is positive. For the last part we argue a bit differently. Since we have already closed the bootstrap assumptions in the second region, the estimates on $|r_1'|$ and $|v_1|$ hold at the starting point of the last region where $r=4R c_0^2\beta^{\frac{2}{7}}$. Let $t_2$ be the time at which $r_1(t_2)=4c_0^{2}\beta^{\frac{2}{7}}R$. We still use a continuity argument by assuming the estimates for $|v_{1}|$ and $|r'_{1}|$ in the last region hold with a larger constant. Since the estimates \eqref{eq: r1' v1 temp 22} still hold, for $r(t)>r_{0}$,

\begin{align*}
\int_{t_{2}}^{t}\left|\frac{d\scE_{1}}{dt}\right|dt=\int_{r_{1}(t)}^{r_{1}(t_{2})}\left|\frac{d\scE_{1}}{dt}\right|\left|\frac{dt}{dr_{1}}\right|dr_{1}\lesssim\frac{GMc_0^{-7}}{R^{\frac{3}{2}}}\int_{r_{1}(t)}^{r_{1}(t_{2})}\frac{\beta^{-1}}{\sqrt{r'-r_{0}}}dr'\lesssim\frac{GMc_0^{-6}}{R}\beta^{-\frac{6}{7}}\lesssim \frac{GM\eta^{3}}{R}.
\end{align*}
It follows that the estimate on $\scE_{1}$ in \eqref{bounds on E and J} still holds in the last stage. A similar argument shows that the estimate on $J$ in \eqref{bounds on E and J} also holds in the last stage. Now we can argue that $r_0$ exists as follows.  From \eqref{bounds on E and J}  and \eqref{eq: r1''}, we see that $r_1''$ is positive for $r_1\in[c_0^2\beta^{2/7}R,3c_0^2\beta^{2/7}R]$. Assume for contradiction that $r_1'$ does not become zero in finite time. We first show that in this case $r_1$ must get as small as $\frac{3}{2}c_0^2\beta^{2/7}R$. If not, then since $r_1'$ is negative, it follows that $r_1$ will remain in the interval $(\frac{3}{2}c_0^2\beta^{2/7}R,3c_0^2\beta^{2/7}R]$. But then using  \eqref{bounds on E and J}  and \eqref{eq: r1''} we conclude that $r_1''$ has a non-trivial lower bound, contradicting the fact that $r_1'$ does not vanish in finite time. Let $t_{*}$ be the time at which $r_1(t_*)=\frac{3}{2}c_0^2\beta^{2/7}R$, and $t_3$ the time at which $r_1(t_3)=3c_0^2\beta^{2/7}R$. Integrating the identity

\begin{align*}
\begin{split}
\frac{d}{dt}|r_1'|^2=2r_1' r_1'',  
\end{split}
\end{align*}
from $t_3$ to $t_*$ and using \eqref{bounds on E and J}  and \eqref{eq: r1''} we conclude that $r_1'$ becomes positive, which is the desired contradiction.  Having proved that $r_1'$ vanishes in finite time, the formula \eqref{eq: r0} together with \eqref{bounds on E and J} give the desired range for $r_{0}$.

Next, we rewrite \eqref{eq: r1'} as

\begin{align*}
\begin{split}
|r_1'(t)|^2=2\scE_1(t_0)+\frac{GM}{2r_1(t)}-\frac{J^2(t_0)}{r_1^2(t)}+2(\scE_1(t)-\scE_1(t_0))-\frac{1}{r_1^2(t)}(J^2(t)-J^2(t_0))=I+II,
\end{split}
\end{align*}
where $I:=2\scE_1(t_0)+\frac{GM}{2r_1}-\frac{J^2(t_0)}{r_1^2(t)}$ and $II:=2(\scE_1(t)-\scE_1(t_0))-\frac{1}{r_1^2(t)}(J^2(t)-J^2(t_0))$. A simple calculation using \eqref{eq: r1'} and \eqref{eq: r0} shows that

\begin{align*}
\begin{split}
  I=\left(\frac{8\scE_1(t_0) r_1+GM+\sqrt{G^2M^2+32\scE_1(t_0)J^2(t_0)}}{4r_1^2}\right)|r_1-r_0|.
\end{split}
\end{align*}
It follows that

\begin{align*}
\begin{split}
\frac{2}{C c_0^4} \frac{GM}{R^2}\beta^{-\frac{4}{7}}|r-r_0|\leq  I\leq \frac{C}{2 c_0^4} \frac{GM}{R^2}\beta^{-\frac{4}{7}}|r-r_0|,
\end{split}
\end{align*}
so it suffices to show that the contribution of $II$ can be estimated using the bootstrap assumptions. For this we use \eqref{eq: scE'} and \eqref{eq: J'} to write

\begin{align}\label{eq: r1' v1 temp 3}
\begin{split}
 II=\frac{2}{r_1^2(t)}\int_{t_0}^t \AVE(\bfE_1)(s)\cdot v_1(s) (r_1^2(s)-r_1^2(t))ds-\frac{2}{r_1^2(t)}\int_{t_0}^t\bfx_1(s)\cdot \AVE(\bfE_1)(s) r_1(s)r_1'(s)ds. 
\end{split}
\end{align}
For the second term above we use \eqref{eq: r1' v1 temp 1} to estimate

\begin{align*}
\begin{split}
 \left|\frac{2}{r_1^2(t)}\int_{t_0}^t\bfx_1(s)\cdot \AVE(\bfE_1)(s) r_1(s)r_1'(s)ds\right| \lesssim \frac{GM\eta^4}{R^2}|r-r_0|\lesssim \frac{GM\beta^{-\frac{8}{7}}}{ c_0^8R^2}|r-r_0|.
\end{split}
\end{align*}
To estimate the first term on the right-hand side in \eqref{eq: r1' v1 temp 3} we note that since $r_1(t)$ is decreasing for $t\geq t_0$, we have 

\begin{align*}
\begin{split}
  |r_1^2(t)-r_1^2(s)|=|r_1(t)-r_1(s)| |r_1(t)+r_1(s) |\leq 2r_1(t)|r_1(t)-r_0|.
\end{split}
\end{align*}
It follows that under the bootstrap assumptions

\begin{align*}
\begin{split}
 \left| \frac{2}{r_1^2(t)}\int_{t_0}^t \AVE(\bfE_1)(s)\cdot v_1(s) (r_1^2(s)-r_1^2(t))ds \right| \lesssim& \frac{GM\eta^4\beta^{-\frac{1}{7}}}{c_0R^{\frac{5}{2}}}\left(\int_{r_0}^{r_1}\frac{d\tilr_1}{\sqrt{\tilr_1-r_0}}\right)|r_1-r_0|\lesssim \frac{GM\beta^{-\frac{8}{7}}}{c_0^8R^{2}}|r-r_0|,
\end{split}
\end{align*}
which can be controlled if $\beta$ is sufficiently large. This proves the desired estimates for $|r_1'|$. The estimates on $|v_1|$ are simpler and similar to the arguments in the first two stages. We omit the details. 
 
\end{proof}


Finally for future use we record the formulas for $v_1''$ and $r_1'''$ in the following lemma.


\begin{lemma}\label{lem: 3rd derivative}
$\bfx_1'''$ and $r_1'''$ satisfy

\begin{align}\label{eq: x1'''}
\begin{split}
\bfx'''_{1}
=&-\frac{GMv_{1}}{4|\bfx_{1}|^{3}}+\frac{3GM\bfx_{1}}{4|\bfx_{1}|^{4}}r'_{1}-\frac{d}{dt}\left(\frac{1}{|\calB_{1}|}\int_{\calB_{1}}\bfE_{1}(t,\bfx)d\bfx\right),
\end{split}
\end{align} 
and 

\begin{align}\label{eq: r1'''}
\begin{split}
r'''_{1}
=&\frac{GMr'_{1}}{2|\bfx_{1}|^{3}}-\frac{3J^{2}r'_{1}}{|\bfx_{1}|^{4}}+2\left(\left(\frac{1}{|\calB_{1}|}\int_{\calB_{1}}\bfE_{1}(t,\bfx)d\bfx\right)\cdot\bfxi_{1}\right)\frac{r_1'}{|\bfx_{1}|}-2\left(\frac{1}{|\calB_{1}|}\int_{\calB_{1}}\bfE_{1}(t,\bfx)d\bfx\right)\cdot\frac{\bfx'_{1}}{|\bfx_{1}|}\\
&-\bfxi_{1}\cdot\frac{d}{dt}\left(\frac{1}{|\calB_{1}|}\int_{\calB_{1}}\bfE_{1}(t,\bfx)d\bfx\right)-\bfxi'_{1}\cdot\left(\frac{1}{|\calB_{1}|}\int_{\calB_{1}}\bfE_{1}(t,\bfx)d\bfx\right).
\end{split}
\end{align}

\end{lemma}


\begin{proof}
These formulas follow by differentiating \eqref{eq: x1'' 2} and \eqref{eq: r1''} and using \eqref{eq: J'}.
\end{proof}


We end this subsection with the following simple consequence of \prop{prop: r1' v1} which will be used many times in the remainder of the paper.


\begin{lemma}\label{lem: eta integration}
Under the assumptions of \prop{prop: r1' v1} and for $r_1\geq r_0$, the following estimate holds for any $m\geq1$, and some universal constant $C=C(m)$:

\begin{align*}
\begin{split}
 \int_{T_0}^t\eta^{m+1}(s)|v_1(s)|ds\leq C R\eta^m(t).
\end{split}
\end{align*}
\end{lemma}


\begin{proof}
We write

\begin{align*}
\begin{split}
 \int_{{T_0}}^t\eta^{m+1}(s)|v_1(s)|ds=\int_{r_1(t)}^{R_1}\eta^{m+1}\left|\frac{ds}{dr_1}\right||v_1|dr_1.
\end{split}
\end{align*}
According to \prop{prop: r1' v1} in the first two stages of the evolution, that is, when $r_1\geq 3c_0^2\beta^{2/7}R$, $|v_1|$ and $|r_1'|$ are comparable and the statement of the lemma follows from the identity above. In the final stage of the evolution when $r_1$ is between $r_0$ and $3c_0^2\beta^{2/7}R$ we get 

\begin{align*}
\begin{split}
   \int_{{T_0}}^t\eta^{m+1}(s)|v_1(s)|ds \lesssim \int_{3c_0^2\beta^{2/7}R}^{R_1}\eta^{m+1}dr_1+ \eta^m(t)c_0^{-1}\beta^{-\frac{1}{7}}R^{\frac{1}{2}}  \int_{r_1(t)}^{3c_0^2\beta^{2/7}R}\frac{dr_1}{\sqrt{r_1-r_0}}\lesssim R\eta^m(t).
\end{split}
\end{align*}
\end{proof}


\subsection{Bootstrap Assumptions and Estimates on the Error Terms}\label{subsec: bootstrap}

In this subsection we estimate various error terms appearing in the equations for $u_j$, $j\geq0$, in terms of the energies in Definition~\ref{def: energies}. We will do this under smallness bootstrap assumptions on $u$. The precise bootstrap assumptions, which are motivated by the analysis in Section~\ref{subsec: x_1 J} and the ODE \eqref{eq: h intro}, are as follows. Suppose $T>0$ is such that $u(t,p)$ is a solution of \eqref{eq: u} for $t\leq T$ and $r_1(t)\geq r_0$ for $t\leq T$, and let $\ell\geq 5$ be a fixed integer. We assume that for some fixed constants $C_1$, all $p,q\in S_R$, and all $t\leq T$
  
\begin{align}\label{eq: bootstrap}
\begin{cases}
&R^{-1}\|\frakD^\alpha u(t)\|_{L^2(S_R)} +\sqrt{\frac{R}{GM}}\|\partial_t\frakD^\alpha u(t)\|_{L^2(S_R)}   \leq C_1\eta^4(t)|v_1(t)|,\quad|\alpha|\leq\ell \\\\
& \frac{1}{2}\leq \frac{|\xi(t,p)-\xi(t,q)|}{|p-q|}\leq 2,\quad |\xi(t,p)-\bfx_1(t)|\leq 5R,\quad \frac{1}{2}\leq\frac{|N|}{|\sg|}\leq 2\\\\
&R^{-1}\|\frakD^\alpha(|N||\sg|^{-1})\|_{L^2(S_R)} \leq C_1\eta^3,\quad 1\leq|\alpha|\leq \ell\\\\
&R^{-2}\|\frakD^{\alpha}(\zeta-Rn)\|_{L^2(S_R)}\leq C_1\eta^3,\quad |\alpha|\leq \ell\\\\
&\frac{R}{GM}\left\| \frakD^\alpha\left(a-\frac{GM}{R^2}\right)\right\|_{L^2(S_R)} \leq C_1\eta^3 ,\quad |\alpha|\leq \ell\\\\
&R^{-2}\|\frakD^{\alpha}\zeta\|_{L^2(S_R)}+R^{-1}\|\frakD^{\gamma}n\|_{L^2(S_R)}\leq C_1\eta^3, \quad 1\leq|\alpha|\leq\ell+1, \quad 1\leq|\gamma|\leq \ell\\\\
&R^{-2}\|\frakD^\alpha h\|_{L^2(S_R)}+R^{-1}\|\frakD^\alpha\mu\|_{L^2(S_R)}+R\|\frakD^\alpha \nu\|_{L^2(S_R)}\leq C_1\eta^3,\quad |\alpha|\leq\ell
\end{cases}.
\end{align} 
Here we recall from \eqref{eq: mu def} and \eqref{eq: nu def} that
 
 \begin{align*}
\mu=1-\frac{R^3}{|\zeta|^3}\quad \mand\quad \nu=\frac{|\zeta|^2+R|\zeta|+R^2}{|\zeta|^3(|\zeta|+R)}-\frac{3}{2R^2}.
\end{align*}
 
 Before proceeding to the estimates, we observe that by the Sobolev embedding, the bootstrap assumptions \eqref{eq: bootstrap} imply the simple pointwise bound
 
 \begin{align}\label{eq: Du pointwise}
\|\frakD^\alpha u(t)\|_{L^\infty(S_R)}\lesssim C_1\eta^4(t)|v_1(t)|,\quad |\alpha|\leq\ell-2.
\end{align}

We start with the following simple estimate on the expression $n\cross\nabla f$.


\begin{lemma}\label{lem: n cross nabla}
Suppose the bootstrap assumptions \eqref{eq: bootstrap} hold. Then 

\begin{align*}
\begin{split}
| n\cross\nabla f| \lesssim |\sd\xi||\sd f|\lesssim \frac{1}{R}\sum_{i=1}^3 |\Omega_if|.
\end{split}
\end{align*}
\end{lemma}

 
 \begin{proof}
 The first estimate is a direct consequence of \eqref{eq: Q d}, and the second a consequence of the first and the fact that,  by \eqref{eq: bootstrap},  
 $|\sd\xi|\lesssim 1$.
 \end{proof}
 

More generally, we have the following estimate on the quadratic form $Q$.
 
 
 \begin{lemma}\label{lem: Q}
 Suppose the bootstrap assumptions \eqref{eq: bootstrap} hold. Then

\begin{align*}
|Q(f,g)|\lesssim|\sd f||\sd g|\lesssim\frac{1}{R^{2}}\left(\sum_{i=1}^{3}|\Omega_{i}f|\right)\left(\sum_{i=1}^{3}|\Omega_{i}g|\right).
\end{align*}
 \end{lemma}


 \begin{proof}
 This follows from the same argument as in the proof of \lem{lem: n cross nabla}.
 \end{proof}


Using Lemma~\ref{lem: n cross nabla} we are able to compare $L^2$ norms on $S_R$ with $L^2$ norms on $\partial\calB_1$. For this, we first recall that by equation \eqref{eq: Nt}

\begin{align*}
\partial_t\frac{|N|}{|\sg|}=-\frac{|N|}{|\sg|} n\cdot(n\cross\nabla u).
\end{align*}
Since $\lim_{t\to T_0}\frac{|N|}{|\sg|}=1$, we conclude that

\begin{align}\label{eq: N sg}
\begin{split}
\frac{|N|}{|\sg|}-1=&-\int_{T_0}^t\frac{|N(s)|}{|\sg|} n(s)\cdot(n(s)\cross\nabla u(s))ds\\
=&-\int_{T_0}^t\left(\frac{|N(s)|}{|\sg|}-1\right) n(s)\cdot(n(s)\cross\nabla u(s))ds-\int_{T_{0}}^{t}n(s)\cdot(n(s)\cross\nabla u(s))ds.
\end{split}
\end{align}


\begin{lemma}\label{lem: L2 comp}
Suppose the bootstrap assumptions \eqref{eq: bootstrap} hold and $\beta$ is sufficiently large. Then

\begin{align*}
\begin{split}
\left| \frac{|N|}{|\sg|}-1\right| \lesssim \eta^3.
\end{split}
\end{align*}
Moreover, if $f$ is a function defined on $\partial\calB_1$ then

\begin{align*}
\frac{3}{4}\|f\circ\xi\|_{L^2(S_R)}\leq \|f\|_{L^2(\partial\calB_1)}\leq \frac{5}{4}\|f\circ\xi\|_{L^2(S_R)}.
\end{align*} 

\end{lemma}


\begin{proof}
The first estimate is a consequence of \eqref{eq: Du pointwise}, \lems{lem: eta integration} and \ref{lem: n cross nabla} , and \eqref{eq: N sg}. The second statement follows from writing

\begin{align*}
\|f\|_{L^2(\partial\calB_1)}^2=\int_{S_R}|f\circ\xi(p)|^2 \frac{|N(p)|}{|\sg(p)|}dS(p)=\|f\circ\xi\|_{L^2(S_R)}^2+\int_{S_R}|f\circ\xi(p)|^2 \left(\frac{|N(p)|}{|\sg(p)|}-1\right)dS(p),
\end{align*}
and using the bound on $\left| \frac{|N|}{|\sg|}-1\right|$.
\end{proof}


We now state the following important result which allows us to estimate error terms involving $\Hone$ and $\Kone$.


\begin{lemma}\label{lem: H estimates}

Suppose the bootstrap assumptions \eqref{eq: bootstrap} hold and that $\beta$ is sufficiently large. Then the following estimates hold for any function $f$ defined on $\partial\calB_1$ and any $k\leq\ell$:

\begin{align*}
&\|\frakD^{k}\Hone f\|_{L^2(S_R)}\lesssim \sum_{j\leq k}\|\frakD^jf\|_{L^2(S_R)},\\
&\|\frakD^{k}\Kone f\|_{L^2(S_R)}+\|\frakD^{k}(I+\Kone)^{-1} f\|_{L^2(S_R)}\lesssim \sum_{j\leq k}\|\frakD^jf\|_{L^2(S_R)},\\
&\|\frakD^{k-1}[\frakD,\Hone]f\|_{L^2(S_R)}+\|[\frakD^{k},\Hone]f\|_{L^2(S_R)}\lesssim \eta^3\sum_{j\leq k-1}\|\frakD^jf\|_{L^2(S_R)},\\
&\|\frakD^{k-1}[\frakD,n\cross\nabla]f\|_{L^2(S_R)}+\|[\frakD^{k},n\cross\nabla]f\|_{L^2(S_R)}\lesssim R^{-1}\eta^3\sum_{j\leq k}\|\frakD^jf\|_{L^2(S_R)},\\
&\|\frakD^{k-1}[n\cross\nabla,\Hone]f\|_{L^2(S_R)}\lesssim R^{-1}\eta^3\sum_{j\leq k-1}\|\frakD^jf\|_{L^2(S_R)},\\
&\|\frakD^k[\partial_t,\Hone]f\|_{L^2(S_R)}\lesssim R^{-1}\eta^4|v_1|\sum_{j\leq k}\|\frakD^jf\|_{L^2(S_R)},\\
&\|\frakD^k[\partial_t^2,\Hone]f\|_{L^2(S_R)}\lesssim \sqrt{\frac{GM}{R^5}}\eta^4|v_1|\sum_{j\leq k}\|\frakD^jf\|_{L^2(S_R)}+R^{-1}\eta^4|v_1|\sum_{j\leq k}\|\partial_t\frakD^jf\|_{L^2(S_R)},\\
&\|\nabla_nf\|_{L^2(S_R)}\lesssim \|n\cross\nabla f\|_{L^2(S_R)} \lesssim R^{-1}\sum_{j\leq 1}\|\frakD^jf\|_{L^2(S_R)}.
\end{align*}

\end{lemma}


\begin{proof}

This is a corollary of \lems{lem: frakD com}, \ref{lem: D com K},  \ref{lem: n cross nabla}, \ref{lem: L2 comp}, \ref{lem: DN K}, and \ref{lem: H commutators}, \cor{cor: Dk H com}, equation \eqref{eq:  n cross nabla H}, \props{prop: L2 C1} and \ref{prop: L2 C2}, and \thm{thm: K inverse}. Note that the estimate on $[\frakD^k,\Hone]f$ follows  from the estimate on $\frakD^{k-1}[\frakD,\Hone]f$ by writing 

\begin{align*}
\begin{split}
 [\frakD^{k},\Hone]f= [\frakD,\Hone]\frakD^{k-1}f+\frakD[\frakD,\Hone]\frakD^{k-2}f+\dots+\frakD^{k-1}[\frakD,\Hone]f,
\end{split}
\end{align*}
and a similar argument can be used to estimate $[\frakD^k,n\cross\nabla]f$.
\end{proof}


Recall the definition $ u_\alpha=\frac{1}{2}(I+\Hone)\frakD^\alpha u$. As a corollary of \lems{lem: L2 comp} and \ref{lem: H estimates} we can control the $L^2(S_R)$ norms of $\frakD^\alpha u$ in terms of the $L^2(\partial\calB_1)$ norms of $\vecu_\alpha$ which appear in the energies.


\begin{corollary}\label{cor: u L2 comp}
Suppose the bootstrap assumptions \eqref{eq: bootstrap} hold and that $\beta$ is sufficiently large. Then for all $|\alpha|\leq \ell$

\begin{align}\label{eq: u L2 comp}
\begin{split}
 \|\frakD^\alpha u\|_{L^2(S_R)}\lesssim  \sum_{|\gamma|\leq|\alpha|}\|\vecu_\gamma\|_{L^2(\partial\calB_1)},\quad \mand\quad  \|\partial_t\frakD^\alpha u\|_{L^2(S_R)}\lesssim \sum_{|\gamma|\leq|\alpha|} \|\partial_t\vecu_\gamma\|_{L^2(\partial\calB_1)}.
\end{split}
\end{align}
\end{corollary}


\begin{proof}

We prove \eqref{eq: u L2 comp} inductively on $|\alpha|$. If $|\alpha|=0$, this follows from the fact that $\Hone u=u$ and $u$ is a vector as well as Lemma \ref{lem: L2 comp}. Assume \eqref{eq: u L2 comp} holds for all $|\alpha|\leq k \leq \ell-1$. We write $\frakD^{k+1}u$

\begin{align*}
\begin{split}
  \frakD^{k+1}u=&\frac{1}{2}(I+\Hone)\frakD^{k+1}u+\frac{1}{2}[\frakD^{k+1},\Hone]u=\vecu_{k+1}+\frac{1}{2}\Re(I+H)\frakD^{k+1}u+\frac{1}{2}[\frakD^{k+1},\Hone]u\\
  =&\vecu_{k+1}-\frac{1}{2}\Re[\frakD^{k+1},\Hone]u+\frac{1}{2}[\frakD^{k+1},\Hone]u=\vecu_{k+1}+\frac{1}{2}\vect[\frakD^{k+1},\Hone]u.
\end{split}
\end{align*}
Similarly $\partial_t\frakD^{k+1}u=\partial_t\vecu_{k+1}+\frac{1}{2}\vect\,\partial_t[\frakD^{k+1},\Hone]u.$ Estimate \eqref{eq: u L2 comp} for $|\alpha|=k+1$ now follows from \lems{lem: L2 comp} and \ref{lem: H estimates} and the induction hypothesis. 
\end{proof}


 To be able to use Lemma \ref{lem: pos energy} 
 to replace the $L^2$ norms on the right-hand side of \eqref{eq: u L2 comp}  by the energies defined in \defn{def: energies}, we need to show that the second fundamental form of $\partial\calB_1$ is positive. This is an easy consequence of the bootstrap assumptions.


\begin{lemma}\label{lem: 2nd fund form}

Suppose the bootstrap assumptions \eqref{eq: bootstrap} hold and that $\beta$ is sufficiently large. Then the second fundamental form of $\partial\calB_1$ is positive.

\end{lemma}


\begin{proof}
In arbitrary orientation-preserving local coordinates, we need to show that the eigenvalues of the matrix

\begin{align*}
\begin{split}
M(t)= -\pmat{\frac{1}{|\zeta_\alpha|^2}\zeta_{\alpha\alpha}(t)\cdot n(t)&\frac{1}{|\zeta_\alpha||\zeta_\beta|}\zeta_{\alpha\beta}(t)\cdot n(t)\\ \frac{1}{|\zeta_\alpha||\zeta_\beta|}\zeta_{\beta\alpha}(t)\cdot n(t)&\frac{1}{|\zeta_\beta|^2}\zeta_{\beta\beta}(t)\cdot n(t)} 
\end{split}
\end{align*}
are positive. Note that since $\partial\calB_1$ at time $t=T_0$ is the round sphere $S_R$, the eigenvalues of $M(T_0)$ are both equal to $R^{-1}$. Now the positivity of the eigenvalues of $M(t)$ follow from writing

\begin{align*}
\begin{split}
 M(t)=M(T_0)+\int_{T_0}^t\frac{dM(s)}{ds} ds,
\end{split}
\end{align*} 
and using the bootstrap assumptions and equation \eqref{eq: Nt}.

\end{proof}


We are now in the position to estimate the $L^2$ norms of $\frakD^k u$ and $\partial_t\frakD^k u $ in terms of the energies in \defn{def: energies}.


\begin{proposition}\label{prop: Du Eu}

Suppose the bootstrap assumptions \eqref{eq: bootstrap} hold and that $\beta$ is sufficiently large. Let $\calE_{\leq k}$ be as defined in \defn{def: energies}. Then for for all $3\leq k \leq \ell$ 

\begin{align*}
\begin{split}
 \sum_{j\leq k}\left(R^{-1}\|\frakD^j u\|_{L^2(S_R)}^2 +(GM)^{-1}R^{2}\|\partial_t\frakD^ju\|_{L^2(S_R)}^2\right)\lesssim \calE_{\leq k}.
\end{split}
\end{align*}

\end{proposition}


\begin{proof}

Note that since under the bootstrap assumptions $a\lesssim\frac{GM}{R^2}$ from \lem{lem: L2 comp} we have

\begin{align*}
\sum_{j\leq k}\|\partial_t\frakD^ju\|_{L^2(S_R)}^2\lesssim \sum_{j\leq k} \int_{\partial\calB_1}|\partial_t\vecu_j|^2dS\lesssim \frac{GM}{R^2}\sum_{j\leq k} \int_{\partial\calB_1}\frac{|\partial_t\vecu_j|^2}{a}dS.
\end{align*}
Therefore to prove the proposition it suffices to show that

\begin{align}\label{eq: Du Eu temp 1}
\frac{1}{R}\sum_{j\leq k}\|\frakD^ju\|_{L^2(S_R)}^2\lesssim \sum_{j\leq k} \calE_j^0,
\end{align}
where

\begin{align*}
\calE_j^0:=\frac{1}{2}\int_{\partial\calB_1}(n\cross\nabla \vecu_j)\cdot \vecu_j dS+\frac{GM}{2R^3}\int_{\partial\calB_1}\frac{|\vecu_j|^2}{a}dS-\frac{3GM}{2R^3}\int_{\partial\calB_1}\frac{(n\cdot \vecu_j)^2}{a}dS.
\end{align*}
By the bootstrap assumptions \eqref{eq: bootstrap} it suffices to prove \eqref{eq: Du Eu temp 1} with $\calE_j^0$ replaced by $\calE_j^1$, where 

\begin{align*}
\calE_j^1:=\frac{1}{2}\int_{\partial\calB_1}(n\cross\nabla \vecu_j)\cdot \vecu_j dS+\frac{1}{2R}\int_{\partial\calB_1}|\vecu_j|^2dS-\frac{3}{2R}\int_{\partial\calB_1}(n\cdot \vecu_j)^2dS.
\end{align*}
Indeed, using the bootstrap assumptions on $a-\frac{GM}{R^2}$ the difference between $\calE_j^0$ and $\calE_j^1$ can be absorbed into the left-hand side of \eqref{eq: Du Eu temp 1}. By \lem{lem: 2nd fund form} the second fundamental form of $\partial\calB_1$ is positive. Therefore by \lem{lem: pos energy} and with the same notation,

\begin{align*}
\calE_j^1\geq&\, \frac{c}{4R}\int_{\partial\calB_1}|\vecu_j|^2dS-\frac{C}{R}\int_{\partial\calB_1}(\vecu_j\cdot n)\vecu_j\cdot(n-R^{-1}\zeta)dS\\
&-\frac{C}{R}\int_{\partial\calB_1}|\vecu_j|^2(R^{-1}\zeta-n)\cdot n \,dS-\frac{C}{R^{2}}\int_{\partial\calB_{1}}(\vecu_{j}\times\zeta)\cdot(n-R^{-1}\zeta)\mathring{u}_{j}dS\\
&-\frac{C}{R^2}\int_{\calB_1}|\mathring{\bfu}_j|^2d\bfx-\frac{CM\rho}{R^2}|\AV(\vec{\bfu}_j)|^2
+\frac{C}{R^{2}}\int_{\partial\calB_{1}}|\zeta\cdot n||\mathring{u}_{j}|^{2}dS-CR\int_{\partial\calB_1}|n\cross\nabla\circu_j|^2dS.
\end{align*}
Using the bootstrap assumptions we conclude that

\begin{align*}
\frac{1}{R}\sum_{j\leq k}\|\frakD^ju\|_{L^2(S_R)}^2\lesssim &\sum_{j\leq k} \calE_j^1+\sum_{j\leq k}\left(\frac{C}{R^2}\int_{\calB_1}|\mathring{\bfu}_j|^2d\bfx+\frac{CM\rho}{R^2}|\AV(\vec{\bfu}_j)|^2\right)+\sum_{j\leq k}\frac{C}{R}\int_{\partial\calB_{1}}|\mathring{u}_{j}|^{2}dS\\
&+\sum_{j\leq k}CR\int_{\partial\calB_{1}}|n\cross\nabla\mathring{u}_{j}|^{2}dS.
\end{align*}
Now by equation \eqref{eq: f0 L2}, we have

\begin{align*}
 \int_{\calB_1}|\mathring{\bfu_j}|^2d\bfx\leq CR\int_{\partial\calB_1}|\mathring{u}_j|^2dS+CR^3\int_{\partial\calB_1}|\nabla_n \mathring{u}_j|^2dS+CM\rho(\AV(\mathring{\bfu}_j))^2.
\end{align*}
For the average $\AV(\mathring{\bfu}_j))^2$, we use \lem{lem: averages}. Taking real parts on both sides of \eqref{eq: uk av 1} and using the bootstrap assumptions \eqref{eq: bootstrap} we have

\begin{align*}
&\left(\AV(\mathring{\bfu}_{j})\right)^{2}\lesssim\frac{R^{4}}{M^{2}}\int_{\partial\calB_{1}}|\mathring{u}_{j}|^{2}dS+\frac{\eta^6R^4}{M^2}\sum_{j\leq k}\|\frakD^ju\|_{L^2(S_R)}^2,\\ 
&\frac{CM\rho}{R^{2}}\left(\AV(\mathring{\bfu}_{j})\right)^{2}\lesssim \frac{1}{R}\int_{\partial\calB_{1}}|\mathring{u}_{j}|^{2}dS+\frac{\eta^6}{R}\sum_{j\leq k}\|\frakD^ju\|_{L^2(S_R)}^2.
\end{align*}
Taking vector parts on both sides of \eqref{eq: uk av 1} and using bootstrap assumptions \eqref{eq: bootstrap} as well as \eqref{eq: uk av 2} inductively, we have

\begin{align*}
\frac{CM\rho}{R^{2}}\left(\AV(\vec{\bfu}_{j})\right)^{2}\lesssim \frac{C\rho}{M}\left(\int_{\partial\calB_{1}}\vec{u}_{j}dS\right)^{2}
\lesssim\frac{\eta(t)^{6}}{R}\sum_{l\leq j}\int_{\partial\calB_{1}}|\vec{u}_{l}|^{2}dS+\frac{\eta^6}{R}\sum_{j\leq k}\|\frakD^ju\|_{L^2(S_R)}^2,
\end{align*}
which, by the bootstrap assumptions and the argument in \rem{rmk: application of average}, can be absorbed by the left hand side $\frac{1}{R}\sum_{j\leq k}\|\frakD^{j}u\|^{2}_{L^{2}(S_{R})}$. So we get

\begin{align*}
\frac{1}{R}\sum_{j\leq k}\|\frakD^ju\|_{L^2(S_R)}^2\lesssim \sum_{j\leq k} \calE_j^1+\sum_{j\leq k}\left(\frac{C}{R}\int_{\partial\calB_1}|\mathring{u}_j|^2dS+CR\int_{\partial\calB_1}|\nabla_n \mathring{u}_j|^2dS +CR\int_{\partial\calB_{1}}|n\cross\nabla\mathring{u}_{j}|^{2}dS\right).
\end{align*} 
Since $\frakD^ju$ is a vector, $\mathring{u}_j=-\frac{1}{2}\Re[\frakD^j,\Hone]u$. So using Lemma \ref{lem: H estimates} we conclude that

\begin{align*}
\frac{1}{R}\sum_{j\leq k}\|\frakD^ju\|_{L^2(S_R)}^2\lesssim \sum_{j\leq k} \calE_j^1.
\end{align*}
\end{proof}


Using \prop{prop: Du Eu}, we can now estimate the quantities $h$, $\mu$, $\nu$, $\frakD\zeta$, $\frakD n$, $\frakD (|N||\sg|^{-1}-1)$, $Rn-\zeta$, and $a-\frac{GM}{R^2}$ appearing in the bootstrap assumptions \eqref{eq: bootstrap} in terms of the energies $\calE_j$. We will also derive a simple estimate for $n_t$ and $\partial_t(I-\Hone)\vecu_k$, where $u_k:=\frac{1}{2}(I+\Hone)\frakD^ku$, which appear in the energy identity \eqref{eq: model energy id}. Note that since $n_t$ appears only in the energy identity \eqref{eq: model energy id} and not in the nonlinearity in the equation \eqref{eq: u} for $u$, we do not need to estimates the higher derivatives  of $n_t$. Also note that in \eqref{eq: D zeta E} and \eqref{eq: N E} below we do not estimate the top order derivatives yet. We will consider the top order derivatives later, after estimating $|N|\partial_t(a|N|^{-1})$ in \prop{prop: at E}.


\begin{proposition}\label{prop: bootstrap by E}
Suppose that the bootstrap assumptions \eqref{eq: bootstrap} hold and that $\beta$ is sufficiently large. Let $u_k:=\frac{1}{2}(I+\Hone)\frakD^ku$. Then for all $3\leq k\leq\ell$ and some universal constant $C$

\begin{align}
 &\sum_{j\leq k} \|\frakD^j(\zeta-Rn)\|_{L^2(S_R)} \leq C R^{\frac{1}{2}} \int_{T_0}^t \calE_{\leq k}^{\frac{1}{2}}(s)ds+C\frac{R^2}{\sqrt{GM}}\calE_{\leq k}^{\frac{1}{2}}+20R^2\eta^3, \label{eq: zeta Rn E}\\
 &\sum_{1\leq j \leq k}\|\frakD^j \zeta\|_{L^2(S_R)}\leq C R^{\frac{1}{2}} \int_{T_0}^t \calE_{\leq k}^{\frac{1}{2}}(s)ds,\label{eq: D zeta E}\\
 &\sum_{1\leq j\leq k}\|\frakD^jn\|_{L^2(S_R)}\leq C R^{-\frac{1}{2}} \int_{T_0}^t \calE_{\leq k}^{\frac{1}{2}}(s)ds+C\frac{R}{\sqrt{GM}}\calE_{\leq k}^{\frac{1}{2}}+30R\eta^3,\label{eq: D n E}\\
  &\sum_{j \leq k}\|\frakD^j h\|_{L^2(S_R)}+R\sum_{ j\leq k}\|\frakD^j\mu\|_{L^2(S_R)}+R^3\sum_{ j\leq k}\|\frakD^j\nu\|_{L^2(S_R)}\leq C R^{\frac{1}{2}} \int_{T_0}^t \calE_{\leq k}^{\frac{1}{2}}(s)ds,\label{eq: h E}\\
 &\sum_{j\leq k}\|\frakD^j\left(a-GMR^{-2}\right)\|_{L^2(S_R)}    \leq C \frac{GM}{R^{\frac{5}{2}}}\int_{T_0}^t \calE_{\leq k}^{\frac{1}{2}}(s)ds+C\frac{\sqrt{GM}}{R}\calE_{\leq k}^{\frac{1}{2}}+\frac{10\,GM\eta^3}{R},\label{eq: a E}\\
 & \sum_{j\leq k-1}\|\frakD^j(|N||\sg|^{-1}-1)\|_{L^2(S_R)}\leq C R^{-\frac{1}{2}} \int_{T_0}^t \calE_{\leq k}^{\frac{1}{2}}(s)ds,\label{eq: N E}\\
 &\|n_t\|_{L^2(S_R)}\leq C R^{-\frac{1}{2}}\calE_{\leq3}^{\frac{1}{2}},\label{eq: nt E}\\
  &\|\partial_t(I-\Hone)\vecu_k\|_{L^2(S_R)}\leq C\left(\frac{\sqrt{GM}}{R}\eta^3+R^{-\frac{1}{2}}\eta^4|v_1|\right)\calE_{\leq k}^{\frac{1}{2}}.\label{eq: I-H uk}
\end{align}

\end{proposition}


 \begin{proof}
 We start with the estimates for $h$. Note that $\partial_th =\frac{u\cdot\zeta}{|\zeta|}$ so $|\partial_th|=| u|$. Since $h(T_0)=0$, we have 
 
 \begin{align*}
\begin{split}
| h(t)|\leq\int_{T_0}^t |u(s)|ds.
\end{split}
\end{align*}
It follows that

\begin{align*}
\begin{split}
 \|h\|_{L^2(S_R)}\leq \int_{T_0}^t\|u(s)\|_{L^2(S_R)}ds .
\end{split}
\end{align*}
The desired estimate for $\|h\|_{L^2(S_R)}$ now follows from \prop{prop: Du Eu} and \lem{lem: eta integration}. The estimates for the derivatives of $h$ follow similarly by differentiating $\partial_th =\frac{u\cdot\zeta}{|\zeta|}$ and using the bootstrap assumptions. For $\mu$ and $\nu$ we argue exactly the same way using the fact that $\mu(T_0)=\nu(T_0)=0$. To estimate the higher derivatives of $\zeta$, first note that $\frakD^\alpha\zeta(T_0)=0$ for all $|\alpha|\geq1$. It then follows that

\begin{align*}
\begin{split}
  \frakD^\alpha\zeta(t)=\int_{T_0}^t\frakD^\alpha u(s) ds,
\end{split}
\end{align*}
and the estimates on $\|\frakD^\alpha\zeta\|_{L^2(S_R)}$ follow from the assumptions \eqref{eq: bootstrap} and \prop{prop: Du Eu}. We next turn to $a-\frac{GM}{R^2}$ for which we use equation \eqref{eq: a} from  \lem{lem: a}. First note that in view of \eqref{eq: nabla bfpsi2} and \lem{lem: acce x1}, under the bootstrap assumptions \eqref{eq: bootstrap} and if $\beta$ is sufficiently large, we have

\begin{align}\label{eq: nabla psi2 eta}
\begin{split}
\sum_{j\leq k} \|\frakD^j(\nabla\psi_2+\bfx_1'')\|_{L^2(S_R)}\leq \frac{GM\eta^3}{R}. 
\end{split}
\end{align}
Since we have already proved \eqref{eq: D zeta E} and \eqref{eq: h E}, we can use \eqref{eq: a} together with \prop{prop: Du Eu}, the bootstrap assumptions \eqref{eq: bootstrap}, and \lems{lem: nabla psi} and \ref{lem: acce x1} to establish \eqref{eq: a E}. Here  $(I+\Hone)\zeta$ can be estimated by observing that $(I+\Hone)\zeta=(I+\Hone)(\zeta\mu)$ and using equation \eqref{eq: h E} and \lem{lem: H estimates}. We can now use \eqref{eq: a E} to prove \eqref{eq: zeta Rn E}. The argument is similar to the one used to prove \eqref{eq: a E}, but we now use \eqref{eq: Rn zeta} instead of \eqref{eq: a}, and use \eqref{eq: a E} to estimate the term involving $a-\frac{GM}{R^2}$ in \eqref{eq: Rn zeta}.  Estimate \eqref{eq: D n E} now follows from \eqref{eq: zeta Rn E} and \eqref{eq: D zeta E}. Estimate \eqref{eq: N E} follows from differentiating equation \eqref{eq: N sg} and using \lem{lem: n cross nabla} and \prop{prop: Du Eu}. For \eqref{eq: nt E} we use \eqref{eq: Nt} to get

\begin{align*}
\begin{split}
 n_t=-n\cross\nabla u+(n\cdot(n\cross\nabla u))n, 
\end{split}
\end{align*}
and use \prop{prop: Du Eu} and the bootstrap assumptions \eqref{eq: bootstrap}. Finally for \eqref{eq: I-H uk} we note that since $\Re\frakD^ku=0$

\begin{align*}
\begin{split}
 \partial_t(I-\Hone)\vecu_k=&-\frac{1}{2}\partial_t(I-\Hone)\Re(I+\Hone)\frakD^ku=\frac{1}{2}\partial_t(I-\Hone)\Re[\frakD^k,\Hone]u\\
  =&\frac{1}{2}(I-\Hone)\Re\partial_t[\frakD^k,\Hone]u-\frac{1}{2}[\partial_t,\Hone]\Re[\frakD^k,\Hone]u\\
  =&-\frac{1}{2}(I-\Hone)\Re\left\{ [\partial_t,\Hone]\frakD^ku+[\Hone,\frakD^k]\partial_tu+\frakD^k[\Hone,\partial_t]u \right\}\\
  &-\frac{1}{2}[\partial_t,\Hone]\Re[\frakD^k,\Hone]u.
  \end{split}
\end{align*}
The estimate \eqref{eq: I-H uk} now follows from the bootstrap assumptions \eqref{eq: bootstrap}, \lem{lem: H estimates}, and \prop{prop: Du Eu}.
\end{proof}


To estimate the nonlinearity in equation \eqref{eq: u} we still need to estimate $|N|\partial_t(a|N|^{-1})$ in terms of the energies. For this we will use the expression \eqref{eq: at} derived for $|N|\partial_t(a|N|^{-1})$ in \prop{prop: at}. However, to estimate the derivatives of this expression we need to compute the commutator of $(I+\Kone^\ast)^{-1}$ and $\frakD$. The same argument used in \lem{lem: D com K} to derive \eqref{eq: D com (I+K)-1} gives

\begin{align*}
\begin{split}
    [\frakD,(I+\Kone^\ast)^{-1}]f=-(I+\Kone^\ast)^{-1}[\frakD,\Kone^\ast](I+\Kone^\ast)^{-1}f.
\end{split}
\end{align*}
The commutator $[\frakD,\Kone^\ast]$ can now be computed using the fact that $\Kone-\Kone^\ast$ is small. That is, we rewrite the identity above as 

\begin{align}\label{eq: D com (I+K*)-1}
\begin{split}
    [\frakD,(I+\Kone^\ast)^{-1}]f=&-(I+\Kone^\ast)^{-1}[\frakD,\Kone](I+\Kone^\ast)^{-1}f\\
    &-(I+\Kone^\ast)^{-1}[\frakD,\Kone^\ast-\Kone](I+\Kone^\ast)^{-1}f.
\end{split}
\end{align}
A precise expression for the difference $\Kone^\ast-\Kone$ can be obtained by taking the  real part of identity \eqref{eq: H H*} in the following lemma, where we compute $\Hone-\Hone^\ast$ which also appears in the energy identity \eqref{eq: model energy id}.


\begin{lemma}\label{lem: H H*}
For any Clifford-algebra valued function $f$ we have

\begin{align}\label{eq: H H*}
\begin{split}
 (\Hone^\ast-\Hone)f=&\frac{\pv}{2\pi R}\int_{\partial\calB_1}\frac{(|\zeta'|^2-R^2)-(|\zeta|^2-R^2)}{|\xi'-\xi|^3}f'dS'\\
 &-(n-R^{-1}\zeta)\frac{\pv}{2\pi}\int_{\partial\calB_1}K(\xi'-\xi)f'dS'-\frac{\pv}{2\pi}\int_{\partial\calB_1}K(\xi'-\xi)(n'-R^{-1}\zeta')f'dS'.
\end{split}
\end{align}
Moreover, if the bootstrap assumptions \eqref{eq: bootstrap} hold and $\beta$ is sufficiently large then for any $k \leq \ell$

\begin{align}\label{eq: H H* L2}
\begin{split}
 \|\frakD^k(\Hone-\Hone^\ast)f\|_{L^2(S_R)}\lesssim \eta^3\sum_{j\leq k}\|\frakD^jf\|_{L^2(S_R)}
\end{split}
\end{align}
\end{lemma}


\begin{proof}

By definition $\Hone^\ast=n\Hone n$ so with $K=K(\xi'-\xi)$

\begin{align*}
\begin{split}
 \Hone^\ast f=&n\,\pv\int_{\partial\calB_1}Kn'n'f'dS'=-\pv\int_{\partial\calB_1}nKf'dS'= \Hone f-\pv\int_{\partial\calB_1}(nK+Kn')f'dS'.
\end{split}
\end{align*}
Identity \eqref{eq: H H*} follows by observing that

\begin{align*}
\begin{split}
 nK+Kn'=&(n-R^{-1}\zeta)K+K(n'-R^{-1}\zeta')+R^{-1}(\zeta K+K\zeta')\\
 =&(n-R^{-1}\zeta)K+K(n'-R^{-1}\zeta')+\frac{1}{2\pi R}\frac{|\zeta|^2-|\zeta'|^2}{|\zeta'-\zeta|^3}.
\end{split}
\end{align*}
The proof of estimate \eqref{eq: H H* L2} using \eqref{eq: H H*} uses \props{prop: L2 C1} and \ref{prop: L2 C2} and \cor{cor: Dk H com}  as in the proof of \lem{lem: H estimates}.

\end{proof}


We can now estimate $|N|\partial_t(a|N|^{-1})$.


\begin{proposition}\label{prop: at E}

Suppose the bootstrap assumptions \eqref{eq: bootstrap} hold and that $\beta$ is sufficiently large. Then for any $3\leq k\leq \ell$ (the implicit constant below depends on $C_1$ in \eqref{eq: bootstrap})

\begin{align*}
\begin{split}
  \sum_{j\leq k}\|\frakD^j(|N|\partial_t(a|N|^{-1}))\|_{L^2(S_R)}\lesssim \frac{GM}{R^\frac{5}{2}}\eta^3\calE_{\leq k}^{\frac{1}{2}}.
\end{split}
\end{align*}

\end{proposition}


\begin{proof}

This proposition follows from differentiating equation \eqref{eq: at} and applying the estimates in \lem{lem: H estimates}, \props{prop: Du Eu}, \props{prop: L2 C1}, and \ref{prop: L2 C2}, \cor{cor: Dk H com}, and the bootstrap assumptions \eqref{eq: bootstrap}. To estimate the contribution of $\nabla\psi_2+\bfx_1''$ we use \eqref{eq: nabla psi2 eta}, and to commute derivatives with $(I+\Kone^\ast)^{-1}$ we also use \eqref{eq: D com (I+K*)-1} and \lem{lem: H H*}. 

\end{proof}


Finally we use \prop{prop: at E} to estimate one more derivative in \eqref{eq: D zeta E} and \eqref{eq: N E}.


\begin{corollary}\label{cor: D ell zeta}

Under the assumptions of \props{prop: bootstrap by E} and \ref{prop: at E},

\begin{align*}
\begin{split}
 \|\frakD^{\ell+1}\zeta\|_{L^2(S_R)} + R \|\frakD^\ell(|N|^{-1}|\sg|-1)\|_{L^2(S_R)} \leq C R^{\frac{1}{2}} \int_{T_0}^t \calE_{\leq \ell}^{\frac{1}{2}}(s)ds.
\end{split}
\end{align*}

\end{corollary}


\begin{proof}
We first estimate $\frakD^\ell(|\sg||N|^{-1})$. Note that in view of \eqref{eq: bootstrap} and \eqref{eq: a E}  it suffices to show that

\begin{align*}
\begin{split}
 \|\frakD^{\ell}(a|\sg||N|^{-1})\|_{L^2(S_R)}\leq C\frac{GM}{R^{\frac{5}{2}}}\int_{T_0}^t\calE_{\leq\ell}^{\frac{1}{2}}(s)ds. 
\end{split}
\end{align*}
But this follows from the identity

\begin{align*}
\begin{split}
 \partial_t\frakD^\ell(a|\sg||N|^{-1})=\frakD^\ell\left(|N|^{-1}|\sg|\,|N|\partial_t(a|N|^{-1})\right)=\sum_{k\leq\ell} c_{k,\ell}\,\frakD^k(|\sg||N|^{-1})\frakD^{\ell-k}(|N|\partial_t(a|N|^{-1})),
\end{split}
\end{align*}
where $c_{k,\ell}$ are some constants, combined with \eqref{eq: bootstrap} and \props{prop: bootstrap by E} and \ref{prop: at E}. Next, to estimate $\frakD^{\ell+1}\zeta$ we first note that

\begin{align*}
\begin{split}
 n\cross\nabla\zeta=\frac{1}{|N|}(\zeta_\beta\zeta_\alpha-\zeta_\alpha\zeta_\beta) =-2n.
\end{split}
\end{align*}
It follows that

\begin{align}\label{eq: high der zeta temp 1}
\begin{split}
n\cross\nabla \frakD^{\ell}  \zeta= -[\frakD^{\ell},n\cross\nabla]\zeta-2\,\frakD^{\ell}n.
\end{split}
\end{align}
Now using \eqref{eq: D n cross nabla} we write

\begin{align*}
\begin{split}
 [\frakD^{\ell},n\cross\nabla]\zeta=&[\frakD^{\ell-1},n\cross\nabla]\frakD\zeta+\frakD^{\ell-1} [\frakD,n\cross\nabla]\zeta\\
 =&[\frakD^{\ell-1},n\cross\nabla]\frakD^\zeta+\frakD^{\ell-1}\left(\frac{1}{|N|}\left((\partial_{\beta}(\frakD\zeta))\zeta_{\alpha}-(\partial_{\alpha}(\frakD\zeta))\zeta_{\beta}\right)-\frac{\Omega(|N||\sg|^{-1})}{|N||\sg|^{-1}}n\cross\nabla \zeta\right)\\
 =:&\frac{1}{|N|}\left((\partial_{\beta}(\frakD^\ell\zeta))\zeta_{\alpha}-(\partial_{\alpha}(\frakD^\ell\eta))\zeta_{\beta}\right)+I,
\end{split}
\end{align*}
where $I$ consists of terms which are lower order in regularity, or which can be absorbed using the bootstrap assumptions \eqref{eq: bootstrap}, and the fact that we have already estimated $\frakD^\ell(|N||\sg|^{-1})$. Combining with \eqref{eq: high der zeta temp 1} and using the coordinate expression for $n\cross\nabla\frakD^\ell\zeta$ we get

\begin{align}\label{eq: high der zeta temp 2}
\begin{split}
\frac{2}{|N|}\left(\zeta_\beta\cross\partial_\alpha(\frakD^\ell\zeta)-\zeta_\alpha\cross\partial_\beta(\frakD^\ell\zeta)\right)=-I-2\,\frakD^\ell n,
\end{split}
\end{align}
Note that $\frakD^\ell\zeta=\frac{1}{2}\vect(I-\Hone)\frakD^\ell\zeta+\frac{1}{2}\vect(I+\Hone)\frakD^\ell\zeta$. Since

\begin{align*}
\begin{split}
 \zeta_\alpha\cross\partial_\beta\vect(I-\Hone)\frakD^\ell\zeta&-\zeta_\beta\cross\partial_\alpha\vect(I-\Hone)\frakD^\ell\zeta\\
 = & \zeta_\alpha\partial_\beta\vect(I-\Hone)\frakD^\ell\zeta-\zeta_\beta\partial_\alpha\vect(I-\Hone)\frakD^\ell\zeta\\
 &+ \zeta_\alpha\cdot\partial_\beta\vect(I-\Hone)\frakD^\ell\zeta-\zeta_\beta\cdot\partial_\alpha\vect(I-\Hone)\frakD^\ell\zeta,
\end{split}
\end{align*}
and using a similar computation for $(I+\Hone)\frakD^\ell\zeta$, we get

\begin{align*}
\begin{split}
 \frac{1}{|N|}(\zeta_\beta\cross\partial_\alpha\frakD^\ell\zeta-\zeta_\alpha\cross\partial_\beta\frakD^\ell\zeta)= &n\cross\nabla\frakD^\ell\zeta+\frac{1}{2|N|}\left(\zeta_\alpha\cdot\partial_\beta\vect(I-\Hone)\frakD^\ell\zeta-\zeta_\beta\cdot\partial_\alpha\vect(I-\Hone)\frakD^\ell\zeta\right)\\
 &+\frac{1}{2|N|}\left(\zeta_\alpha\cdot\partial_\beta\vect(I+\Hone)\frakD^\ell\zeta-\zeta_\beta\cdot\partial_\alpha\vect(I+\Hone)\frakD^\ell\zeta\right).
\end{split}
\end{align*}
Now by \eqref{eq: n cross nabla clifford vec} applied to $\calB_1$ and $\calB_1^c$,

\begin{align*}
\begin{split}
  &\frac{1}{2|N|}\left(\zeta_\alpha\cdot\partial_\beta\vect(I-\Hone)\frakD^\ell\zeta-\zeta_\beta\cdot\partial_\alpha\vect(I-\Hone)\frakD^\ell\zeta\right)=\nabla_n^{\mathrm{out}}\frac{\Re}{2}(I-\Hone)\frakD^\ell\zeta,\\
  &\frac{1}{2|N|}\left(\zeta_\alpha\cdot\partial_\beta\vect(I+\Hone)\frakD^\ell\zeta-\zeta_\beta\cdot\partial_\alpha\vect(I+\Hone)\frakD^\ell\zeta\right)=\nabla_n^{\mathrm{in}}\frac{\Re}{2}(I+\Hone)\frakD^\ell\zeta,
\end{split}
\end{align*}
where $\nabla^{\mathrm{out}}_n$ and $\nabla_n^{\mathrm{in}}$ are the Dirichlet-Neumann maps for $\calB_1^c$ and $\calB_1$, respectively, both with respect to the exterior normal $n$ to $\calB_1$. Moreover, since $\Hone\zeta$ is a pure vector (see \eqref{eq: nabla psi1}),

\begin{align*}
\begin{split}
 \Re(I-\Hone) \frakD^\ell\zeta=\Re[\frakD^\ell,\Hone]\zeta\quad\mand\quad\Re(I+\Hone)\frakD^\ell\zeta=-\Re[\frakD^\ell,\Hone]\zeta,
\end{split}
\end{align*}
which shows that the real parts $ \Re(I-\Hone) \frakD^\ell\zeta$ and $ \Re(I+\Hone) \frakD^\ell\zeta$ are lower order with respect to regularity compared to $\frakD^\ell\zeta$, so

\begin{align}\label{eq: high der zeta temp 3}
\begin{split}
   \frac{1}{|N|}(\zeta_\beta\cross\partial_\alpha\frakD^\ell\zeta-\zeta_\alpha\cross\partial_\beta\frakD^\ell\zeta)= n\cross\nabla\frakD^\ell\zeta+II.
\end{split}
\end{align}
where $II$ is lower order with respect to regularity compared to $\frakD^{\ell+1}\zeta$. Next we write

\begin{align*}
\begin{split}
 n\cross\nabla\frakD^\ell\zeta=\frac{1}{|\sg|}\left(\zeta_\beta(0)\partial_\alpha\frakD^\ell\zeta-\zeta_\alpha(0)\partial_\beta\frakD^\ell\zeta\right)+III,
\end{split}
\end{align*}
where $III$ denotes the error which can be absorbed using the bootstrap assumptions \eqref{eq: bootstrap}. Since $\zeta(0):S_R\to S_R$ is the identity map, combined with \eqref{eq: high der zeta temp 3} this gives us control of $\snabla\frakD^\ell\zeta$ where $\snabla$ is the intrinsic covariant differentiation operator on $S_R$, which in turn gives us control of $\frakD^{\ell+1}\zeta$.

\end{proof}


\subsection{Closing the Energy Estimates} \label{subsec: apriori}

 In this section we use the results in Sections~\ref{subsec: energy id}, \ref{subsec: x_1 J}, and \ref{subsec: bootstrap} to close the energy estimates for equation \eqref{eq: u}. More precisely, our goal is to prove the following result, where we use the notation introduced in Sectiona~\ref{subsec: energy id}, \ref{subsec: x_1 J}, and \ref{subsec: bootstrap}.

 
 \begin{proposition}\label{prop: a priori}
 
 Let $\ell\geq 5$ be a fixed integer, and suppose $\beta$ is sufficiently large. There exists a constant $C_0$ such that if $u$ is a solution to $\eqref{eq: u}$ on $[T_0,T)$ and $r_1(t)\geq r_0$ for all $t\in[T_0,T)$, then
 
 \begin{align}\label{eq: a priori}
\begin{split}
 \calE_{\leq \ell}(t)\leq C_0R\eta^8(t)|v_1(t)|^2,\qquad t\in[T_0,T).
\end{split}
\end{align}
 Moreover, there exists a constant $C_1>0$ such that the estimates \eqref{eq: bootstrap} are satisfied for all $t\in[T_0,T)$.
 \end{proposition}
 

The existence of a solution up to the point of closest approach $r_0$ is an immediate corollary of \prop{prop: a priori} and the local well-posendess result in \cite{Nor1}.


\begin{corollary}\label{cor: existence}
If $t_0$ is the time at which $r(t_0)=r_0$ then the solution to \eqref{eq: Euler} remains regular for $t\leq t_0$.
\end{corollary}

\begin{proof}
By the local well-posedness result of \cite{LinNor1, Nor1}\footnote{See also \cite{BMSW1} where the local well-posedness in dimension two and in the irrotational case was derived in a set up more similar to the one used here. The local well-posedness proof of \cite{BMSW1} can also be extended to three dimensions as in \cite{Wu99}.} the solution exists locally in time starting at $t=T_0$. By \prop{prop: a priori} the higher order derivatives of the solution and the Lipschitz constant of the boundary remain bounded as long as $r< r_0$, so referring back to \cite{LinNor1, Nor1} the solution can be extended to $t=t_0$. 

\end{proof}


 The proof of Proposition \ref{prop: a priori} uses the energy identity \eqref{eq: model energy id}, where we will treat the contributions from the source term $\partial_tF$ and the rest of the terms, to which we refer as ``error terms", separately. To understand the contribution of the source term we use the analysis from Section~\ref{subsec: x_1 J}, and for the error terms we use the estimates in Section~\ref{subsec: bootstrap}.
 

\begin{proof}[Proof of \prop{prop: a priori}]
We divide the proof into several steps.

\emph{Step 1.} Since $\calE_j(T_0)=0$ for all $j$, the estimate \eqref{eq: a priori} is trivially satisfied at $t=T_0$ for any choice of $C_0$. Given $C_0$ (to be chosen later) we let $T_1$ be the largest time in $[T_0,T]$ such that 

 \begin{align}\label{eq: a priori bootstrap}
\begin{split}
 \calE_{\leq \ell}(t)\leq \frac{1}{2}C_0R\eta^8(t)|v_1(t)|^2,\qquad t\in[T_0,T_1).
\end{split}
\end{align}
For another constant $C_1$ to be fixed later, we let $T_2$ be the largest time in $[T_0,T]$ such that the bootstrap assumptions \eqref{eq: bootstrap} are satisfied for all $t\in[T_0,T_2)$. Note that since the assumptions \eqref{eq: bootstrap} are trivially satisfied at $t=T_0$, we have $T_2>T_0$.

\emph{Step 2.} We claim that if $C_1$ and $\beta$ are chosen sufficiently large, then $T_1\leq T_2$. Assume for contradiction that $T_2<T_1$.
  Let $\beta$ be so large that the results in Section~\ref{subsec: bootstrap} hold. By \props{prop: Du Eu}, if $C_1$ is chosen sufficiently large relative to $C_0$, then the estimates on $\|\frakD^ku\|_{L^2(S_R)}$ and $\|\partial_t\frakD^ku\|_{L^2(S_R)}$ in \eqref{eq: bootstrap} hold with $C_1$ replaced by $C_1/2$ for $t< T_2$. We will now use \prop{prop: bootstrap by  E} and \cor{cor: D ell zeta} to show that  the remaining conditions in \eqref{eq: bootstrap} also hold with $C_1$ replaced by $C_1/2$ for $t< T_2$, which is a contradiction. Let $\beta$ so large that the conclusions of \prop{prop: r1' v1} hold. Then by \eqref{eq: a priori bootstrap} and \lem{lem: eta integration},

\begin{align*}
\begin{split}
 \int_{T_0}^t\calE_{\leq \ell}^{\frac{1}{2}}(s)ds\lesssim C_0^{\frac{1}{2}}R^{\frac{3}{2}}\eta^3(t).
\end{split}
\end{align*}
It follows from \prop{prop: bootstrap by E} and \cor{cor: D ell zeta} that if $C_1$ is chosen sufficiently large relative to $C_0$ then the bootstrap assumption on $h$, $\mu$, $\nu$, $\frakD\zeta$, $\frakD n$, $\zeta-Rn$, $a-\frac{GM}{R^2}$, and $|N||\sg|^{-1}-1$ are satisfied with $C_1$ replaced by $C_1/2$ for $t< T_2$. Note that the estimate on $|\zeta|=|\xi-\bfx_1|$ with $5R$ replaced by $3R$ follows from the estimate on $h$ if $\beta$ is sufficiently large. Similarly if $\beta$ is sufficiently large the estimates on $|N||\sg|^{-1}-1$ imply the estimate $\frac{2}{3}\leq \frac{|N|}{|\sg|^{-1}}\leq\frac{3}{2}$. Finally to show that $\frac{2}{3}\leq \frac{|\xi(p)-\xi(q)|}{|p-q|}\leq\frac{3}{2}$, we consider the map $\zeta: S_{R}\rightarrow \partial\calB_{1}$ and derive a pointwise estimate on the differential of $\zeta$ with respect to $p\in S_{R}$, which we denote by $d\zeta$. To estimate $d\zeta$ we view $\zeta$ as a map to all of $\bbR^3$ and identify the tangent space of $\bbR^3$ with $\bbR^3$.  Note that

\begin{align*}
\zeta(t,p)=\zeta(0,p)+\int_{T_{0}}^{t}\zeta_{t}(t',p)dt'=p+\int_{T_{0}}^{t}\zeta_{t}(t',p)dt',\quad \Rightarrow\quad d\zeta(t,p)=\iota+\int_{T_{0}}^{t}d\zeta_{t}(t',p)dt',
\end{align*}
where $\iota$ is the inclusion of the tangent space $T_pS_R$ into $\bbR^3$. Using the fact that $|d\zeta_{t}|\lesssim R^{-1}\left(\sum_{i=1}^3|\Omega_i\zeta_{t}|^2\right)^{\frac{1}{2}}$, Proposition \ref{prop: Du Eu}, and the Sobolev inequality, we have $\left|\int_{T_{0}}^{t}d\zeta_{t}(t',p)dt'\right|\lesssim\eta(t)^{3}$. So the differential $d\zeta$ is a small perturbation of the inclusion map. It follows that, viewed as a map from $T_pS_T$ to $T_{\zeta(p)}\partial\calB_1$, $d\zeta(p)$ has norm close to one. By the Inverse Function Theorem, as a map from $S_{R}$ to $\partial\calB_{1}$, $\zeta$ has an inverse and the differential of its inverse also has norm close to one. Using the mean value theorem for both $\zeta$ and its inverse, the desired estimate for $\frac{|\xi(p)-\xi(q)|}{|p-q|}=\frac{|\zeta(p)-\zeta(q)|}{|p-q|}$ follows.

\emph{Step 3.} In this and the next step we show that if $\beta$ and $C_0$ are sufficiently large, then $T_1=T$. We will do this by showing that of $\beta$ and $C_0$ are sufficiently large, then  

\begin{align}\label{eq: a priori closing bootstrap}
\begin{split}
  \calE_{\leq \ell}(t)\leq \frac{1}{4}C_0R\eta^8(t)|v_1(t)|^2,\qquad t\in [T_0,T_1). 
\end{split}
\end{align}
Note that since we have already shown that $T_1\leq T_2$, we may assume that the bootstrap assumptions \eqref{eq: bootstrap} hold for all $t<T_1$. Moreover, by taking $\beta$ sufficiently large we may assume that the conclusions of \props{prop: Du Eu} and \ref{prop: bootstrap by E} hold for all $t\leq T_1$. We treat the contribution of the source term $F_t$ to the energy identity \eqref{eq: model energy id} in this step, and leave the error terms to the following step. More precisely, we consider the contribution of  

\begin{align}\label{eq: a priori I def}
\begin{split}
 I:=\frac{R^2}{GM}\langle \partial_t\frakD^kF,\partial_t\vecu_k\rangle 
\end{split}
\end{align}
to the energy identity \eqref{eq: model energy id}, where we have replaced $a$ by its leading order contribution $\frac{GM}{R^2}$, and where the source term $F$ is defined in \eqref{eq: F}. We will show that if $\beta$ is large enough, then

\begin{align}\label{eq: a priori I bound}
\begin{split}
\left|\int_{T_0}^tI(s)ds\right|\lesssim C_0^{\frac{1}{2}}R\eta^8(t)|v_1(t)|^2.
\end{split}
\end{align}
Note that if $C_0$ us sufficiently large, then the right-hand side above is bounded by $\frac{1}{100}C_0\eta^8(t)|v_1(t)|^2$. Writing

\begin{align*}
\begin{split}
 I=\frac{R^2}{GM}\partial_t\langle \partial_t \frakD^kF,\vecu_k\rangle-\frac{R^2}{GM}\langle \partial_t^2 \frakD^kF,\vecu_k\rangle, 
\end{split}
\end{align*}
for any $t<T$ gives

\begin{align}\label{eq: a priori temp 1}
\begin{split}
 \int_{T_0}^t I(s)ds=&\frac{R^2}{GM}\langle\partial_t\frakD^kF(t),\vecu_k(t)\rangle-\frac{R^2}{GM}\int_{T_0}^t\langle\partial_s^2\frakD^kF(s),\vecu_k(s)\rangle ds=:I_1+I_2.
\end{split}
\end{align}
Using the results from Sections~\ref{subsec: x_1 J} and \ref{subsec: bootstrap}, we will obtain estimates on $\partial_tF$ and $\partial_t^2F$. Direct differentiation of \eqref{eq: F} gives
\begin{align*}
\begin{split}
 &\partial_tF=-\frac{3GM\eta^4r_1'}{8R^4}(\zeta-3(\zeta\cdot\bfxi_1)\bfxi_1)+\frac{GM\eta^3}{8R^3}(u-3(u\cdot\bfxi_1)\bfxi_1-3(\zeta\cdot \bfxi_1')\bfxi_1-3(\zeta\cdot\bfxi_1) \bfxi_1'),
\end{split}
\end{align*}
and

\begin{align*}
\begin{split}
 \partial_t^2F= &-\frac{3GM\eta^4r_1'}{4R^4}(u-3(u\cdot\bfxi_1)\bfxi_1-3(\zeta\cdot \bfxi_1')\bfxi_1-3(\zeta\cdot\bfxi_1) \bfxi_1')\\
 &+\frac{GM\eta^3}{8R^3}(u_t-3(u_t\cdot\bfxi_1)\bfxi_1-3(\zeta\cdot\bfxi_1'')\bfxi_1-3(\zeta\cdot\bfxi_1)\bfxi_1''-6(u\cdot\bfxi_1')\bfxi_1-6(u\cdot\bfxi_1)\bfxi_1'-6(\zeta\cdot\bfxi_1')\bfxi_1')\\
 &+\frac{3GM\eta^5(r_1')^2}{2R^5}(\zeta-3(\zeta\cdot\bfxi_1)\bfxi_1)-\frac{3GM\eta^4r_1''}{8R^4}(\zeta-3(\zeta\cdot\bfxi_1)\bfxi_1).
\end{split}
\end{align*}
Also

\begin{align*}
\begin{split}
  & \bfxi_1'=\frac{v_1}{r_1}-\frac{r_1'}{r_1}\bfxi_1,\\
  &\bfxi_1''=-\frac{2r_1'v_1}{r_1^2}+\frac{v_1'}{r_1}-\frac{r_1''\bfxi_1}{r_1}+\frac{2(r_1')^2\bfxi_1}{r_1^2},\\
  &r_1''=v_1'\cdot\bfxi_1+\frac{|v_1|^2}{r_1}-\frac{(v_1\cdot\bfxi_1)^2}{r_1}.
\end{split}
\end{align*}
Differentiating these relations using $\frakD$ and combining with \prop{prop: r1' v1} and the bootstrap assumptions \eqref{eq: bootstrap} we conclude that

\begin{align}\label{eq: Ft bound}
\begin{split}
 \sum_{j\leq k}\|\frakD^{j}\partial_tF(t)\|_{L^2(S_R)}\lesssim& \frac{GM\eta^3(t)}{R^3}\sum_{j\leq k}\|\frakD^ju(t)\|_{L^2(S_R)}+\frac{GM\eta^4(t)|r_1'(t)|}{R^2}+\frac{GM\eta^4(t)|v_1(t)|}{R^2}\\
 \lesssim &\frac{GM\eta^3(t)}{R^3}\sum_{j\leq k}\|\frakD^ju(t)\|_{L^2(S_R)}+\frac{GM\eta^4(t)|v_1(t)|}{R^2},
\end{split}
\end{align}
and

\begin{align}\label{eq: Ftt bound}
\begin{split}
   \sum_{j\leq k}\|\frakD^{j}\partial_t^2F(t)\|_{L^2(S_R)}\lesssim& \frac{GM\eta^3(t)}{R^3}\sum_{j\leq k}\|\partial_t\frakD^ju(t)\|_{L^2(S_R)}+\frac{GM\eta^4(t)(|r_1'(t)|+|v_1(t)|)}{R^4}\sum_{j\leq k}\|\frakD^ju(t)\|_{L^2(S_R)}\\
   &+\frac{GM\eta^5(|v_1|^2+|r_1'|^2)}{R^3}+\frac{GM\eta^4|v_1'|}{R^2}\\
   \lesssim&  \frac{GM\eta^3(t)}{R^3}\sum_{j\leq k}\|\partial_t\frakD^ju(t)\|_{L^2(S_R)}+\frac{GM\eta^4(t)|v_1(t)|}{R^4}\sum_{j\leq k}\|\frakD^ju(t)\|_{L^2(S_R)}\\
   &+\frac{GM\eta^5|v_1|^2}{R^3}+\frac{G^2M^2\eta^6}{R^4}.
\end{split}
\end{align}
Plugging \eqref{eq: Ft bound} back into \eqref{eq: a priori temp 1} and using \eqref{eq: a priori bootstrap} and \prop{prop: Du Eu} we get

\begin{align*}
\begin{split}
| I_1 |\lesssim& \eta^3(t)R^{-1}\sum_{j\leq k}\|\frakD^ju(t)\|_{L^2(S_R)}^2 +\eta^4(t)|v_1(t)|\sum_{j\leq k}\|\frakD^ju(t)\|_{L^2(S_R)}\\
 \lesssim& C_0R\eta^{11}(t)|v_1(t)|^2+C_0^{\frac{1}{2}}R\eta^8(t)|v_1(t)|^2\lesssim C_0^{\frac{1}{2}}R\eta^8(t)|v_1(t)|^2,
\end{split}
\end{align*}
where the last step follows if $\beta$ is sufficiently large. We next turn to $I_2$ in \eqref{eq: a priori temp 1}. By repeated applications of \prop{prop: Du Eu}, \eqref{eq: Ftt bound}, and \eqref{eq: a priori bootstrap} we get

\begin{align}\label{eq: a priori temp 2}
\begin{split}
\frac{R^2}{GM}\langle \partial^2_t\frakD^kF(s),\vecu_k(s)\rangle\lesssim &C_0\sqrt{\frac{GM}{R}}\eta^{11}(s)|v_1(s)|^2+C_0\eta^{12}(s)|v_1(s)|^3\\
&+C_0^{\frac{1}{2}}\eta^9(s)|v_1(s)|^3+\frac{GMC_0^{\frac{1}{2}}}{R}\eta^{10}(s)|v_1(s)|.
\end{split}
\end{align}
Now using \prop{prop: r1' v1} and \lem{lem: eta integration} we conclude that if $\beta$ is sufficiently large, then

\begin{align}\label{eq: I2 bound}
| I_2(t) |\lesssim C_0^{\frac{1}{2}}R\eta^8(t)|v_1(t)|^2.
\end{align}
Since $|v_1|$ has higher powers in \eqref{eq: a priori temp 2} compared to \lem{lem: eta integration}, we carry out this computation for the last two terms on the right had side of \eqref{eq: a priori temp 2}, $C_0^{\frac{1}{2}}\eta^{9}(s)|v_1(s)|^3$ and $\frac{GMC_{0}^{\frac{1}{2}}}{R}\eta(s)^{10}|v_{1}(s)|$. Note that  if $\beta>0$ is large enough, the first term on the right had side of \eqref{eq: a priori temp 2} is smaller than the last term and second term is smaller than the third term. Let

\begin{align*}
I_{2,3}:=C_0^{\frac{1}{2}}\int_{r_1(t)}^{R_1}\eta^{9}(s)|v_1(s)|^3\left|\frac{ds}{dr_1}\right| dr_1.
\end{align*}
We consider three different cases according to \prop{prop: r1' v1} depending on the value of $r_1(t)$. If $r_1\geq c_0^{-2}R\beta^{\frac{12}{7}}$ then

\begin{align*}
I_{2,3}\lesssim \frac{GMc_0^2C_0^{\frac{1}{2}}\beta^{-\frac{12}{7}} }{R}\int_{r_1(t)}^{\infty}\eta^9(r)dr\lesssim C_0^{\frac{1}{2}}R\eta^8(r_1)|v_1(r_1)|^2.
\end{align*}
Next, using this estimate, if $r_1\in[3c_0^2R\beta^{\frac{2}{7}},c_0^{-2}R\beta^{\frac{12}{7}}]$, then

\begin{align*}
I_{2,3}\lesssim &C_0^{\frac{1}{2}}R\eta^8(c_0^{-2}R\beta^{\frac{12}{7}})|v_1(c_0^{-2}R\beta^{\frac{12}{7}})|^2+\frac{GMC_0^{\frac{1}{2}}}{R}\int_{r_1}^{c_0^{-2}R\beta^{\frac{12}{7}}}\eta^{10}(r)dr\\
\lesssim &C_0^{\frac{1}{2}}R\eta^8(r_1)|v_1(r_1)|^2+C_0^{\frac{1}{2}}R\eta^8(r_1)\frac{GM}{R}\eta(r_1)\lesssim C_0^{\frac{1}{2}}R\eta^8(r_1)|v_1(r_1)|^2.
\end{align*}
Finally, using this estimate, if $r_1\in(r_0,4c_0^2R\beta^{\frac{2}{7}}]$, then

\begin{align*}
I_{2,3}\lesssim & C_0^{\frac{1}{2}}R\eta^8(4c_0^2R\beta^{\frac{2}{7}})|v_1(4c_0^2R\beta^{\frac{2}{7}})|^2+\frac{GMc_0^{-19}C_0^{\frac{1}{2}}}{R^{\frac{1}{2}}}\beta^{-\frac{19}{7}}\int_{r_1}^{4c_0^2R\beta^{\frac{2}{7}}}\frac{dr}{\sqrt{r-r_0}}\\
\lesssim & C_0^{\frac{1}{2}}R\eta^8(r_1)|v_1(r_1)|^2 +C_0^{\frac{1}{2}}Rc_0^{-16}\beta^{-\frac{16}{7}}\frac{GMc_0^{-2}}{R}\beta^{-\frac{2}{7}}\lesssim C_0^{\frac{1}{2}}R\eta^8(r_1)|v_1(r_1)|^2,
\end{align*}
completing the estimate for $I_{2,3}$. The estimate for

\begin{align*}
I_{2,4}:=\frac{GMC^{\frac{1}{2}}_{0}}{R}\int_{r_{1}(t)}^{r(T_{0})}\eta^{10}(s)|v_{1}(s)|\left|\frac{ds}{dr_{1}}\right|dr_{1}.
\end{align*}
follows from the one for $I_{2,3}$ and the observation that, in view of \eqref{eq: v1 bounds}, $|\eta|\lesssim(GM)^{-1}R|v_1|^2$. This completes the proof of \eqref{eq: a priori I bound}.

\emph{Step 4.} Let $I$ be as defined in \eqref{eq: a priori I def}, and let $II$ be defined according to the relations \eqref{eq: model energy id} and

\begin{align*}
\frac{d\calE_k}{dt}=:II-I.
\end{align*}
We will show that if $\beta$ is sufficiently large then

\begin{align}\label{eq: a priori II bound}
\left|\int_{T_0}^tII(s)ds\right|\leq \frac{1}{100} C_0R\eta^8(t)|v_1(t)|^2.
\end{align}
Combined with \eqref{eq: a priori I bound} this will complete the proof of \eqref{eq: a priori closing bootstrap} and hence of the proposition. We divide the terms in $II$ into several groups. Let $\tilg_k$ be as defined in \cor{cor: uk eq}, and set

\begin{align*}
\begin{split}
  &II_1:=\langle \tilg_k,\frac{\partial_t\vecu_k}{a}\rangle + I,\\
 & II_2:=\frac{1}{2}\left\langle \frac{1}{|N|}\partial_t\left(\frac{|N|}{a}\right)\partial_t\vecu_k,\partial_t\vecu_k\right\rangle+\frac{GM}{2R^3}\left\langle \frac{1}{|N|}\partial_t\left(\frac{|N|}{a}\right)\vecu_k,\vecu_k\right\rangle+\frac{3GM}{2R^3}\left\langle\frac{1}{|N|}\partial_t\left(\frac{|N|}{a}\right)\vecu_k\cdot n,\vecu_k\cdot n\right\rangle,\\
&II_3:=\frac{3GM}{R^3}\left\langle a^{-1}\vecu_k\cdot n,\vecu_k\cdot n_t\right\rangle,\\
&II_4:=-\frac{3GM}{2R^3}\langle (n\cdot \vecu_k)n,(\Hone^\ast-\Hone)(a^{-1}\partial_t\vecu_k)\rangle,\\
&II_5:=\frac{3GM}{2R^{3}}\langle(n\cdot \vecu_k)n,a^{-1}\partial_{t}(I-H_{\partial\calB_{1}})\vecu_k\rangle,\\
 &II_6:=-\frac{3GM}{2R^3}\langle ( n\cdot \vecu_k) n,[\Hone,a^{-1}\partial_t]\vecu_k\rangle-\frac{1}{2}\langle \vecQ(u,\vecu_k),\vecu_k\rangle.
\end{split}
\end{align*}
The first term $II_1$ contains the error terms from the nonlinearity in the equation \eqref{eq: uk eq} for $\vecu_k$ and the remaining terms $II_2,\dots, II_6$ contain the other error terms arising in the energy identity \eqref{eq: model energy id}. Using the bootstrap assumptions \eqref{eq: bootstrap} and \eqref{eq: a priori bootstrap}, \props{prop: Du Eu}, \ref{prop: bootstrap by E}, and \ref{prop: at E}, \cor{cor: D ell zeta}, and \lems{lem: H H*} and Lemma \ref{lem: Sobolev} we get (with constants possibly depending on  $C_1$)

\begin{align*}
\begin{split}
& |II_2|\lesssim  R^{-1}\eta^7|v_1|\calE_{\leq k}\lesssim C_0 \eta^{15}|v_1|^3,\\
 &|II_3|\lesssim R^{-\frac{3}{2}}\calE_{\leq 3}^{\frac{1}{2}}\calE_{\leq k}\lesssim C_0^{\frac{3}{2}}\eta^{12}|v_1|^3,\\
 &|II_4|\lesssim \sqrt{\frac{GM}{R}}R^{-1}\eta^3\calE_{\leq k}\lesssim C_0\sqrt{\frac{GM}{R}}\eta^{11}|v_1|^2,\\
 &|II_5|\lesssim \sqrt{\frac{GM}{R}}R^{-1}\eta^3\calE_{\leq k}+R^{-1}\eta^4|v_1|\calE_{\leq k}\lesssim  C_0\sqrt{\frac{GM}{R}}\eta^{11}|v_1|^2+ C_0\eta^{12}|v_1|^3,\\
 & |II_6|\lesssim R^{-1}\eta^4|v_1|\calE_{\leq k}\lesssim C_0\eta^{12}|v_1|^3.
\end{split}
\end{align*}
Arguing as in the proof of \eqref{eq: I2 bound} we conclude that if $\beta$ is sufficiently large (depending on $C_0$ and $C_1$), then

\begin{align}\label{eq: a priori II bound temp 1}
\begin{split}
\sum_{j=2}^6\int_{T_0}^t|II_j(s)|ds\leq \frac{1}{300} C_0R\eta^8(t)|v_1(t)|^2.
\end{split}
\end{align}
We now turn to the error terms form the nonlinearity in $II_1$. Let $g_0$ be as defined in \cor{cor: uk eq}, that is, in the notation of \prop{prop: u eq},

\begin{align*}
\begin{split}
g_0=-\partial_tF-\partial_t  E_1+\partial_t\left(\frac{1}{|\calB_{1}|}\int_{\calB_{1}}\bfE_1(t,\bfx)d\bfx\right)+E_2 +\partial_t\left(\frac{a}{|N|}\right)N.
\end{split}
\end{align*} 
Repeated applications of \lem{lem: H estimates}, \props{prop: Du Eu} and \ref{prop: bootstrap by E} , and \cor{cor: D ell zeta} give

\begin{align*}
\begin{split}
 \sum_{0\leq j\leq k}\|\tilg_j-\frakD^j g_0\|_{L^2(S_R)}\lesssim & \frac{GM}{R^{\frac{5}{2}}}\eta^3\calE_{\leq k}^{\frac{1}{2}}+\frac{\sqrt{GM}}{R^2}\eta^4|v_1|\calE_{\leq k}^{\frac{1}{2}}+\frac{GM}{R^{\frac{7}{2}}}\eta^4|v_1| \int_{T_0}^t\calE_{\leq k}^{\frac{1}{2}}(s)ds +\frac{GM}{R^2}\eta^7|v_1|\\
 \lesssim & \frac{C_0^{\frac{1}{2}}GM}{R^{2}}\eta^7|v_1|+\frac{\sqrt{GMC_0}}{R^\frac{3}{2}}\eta^8|v_1|^2.
\end{split}
\end{align*}
Here to estimate the term $[\partial_t^2,[\frakD^k,\Hone]]u$ we have used the fact that

\begin{align*}
\begin{split}
[\partial_t^2,[\frakD^k,\Hone]]u=\frakD^k[\partial_t^2,\Hone]u-[\partial_t^2,\Hone]\frakD^ku,
\end{split}
\end{align*}
and applied \lem{lem: H estimates}.  Use \lem{lem: eta integration} as in the proof of \eqref{eq: I2 bound} now shows that if $\beta$ is sufficiently large, then

\begin{align*}
\begin{split}
 \int_{T_0}^t \left| \langle\tilg_k(s)-\frakD^kg_0(s),\frac{\partial_s\vecu_k(s)}{a(s)}\rangle \right| ds\leq \frac{1}{400}C_0\eta^8(t)|v_1(t)|^2.
\end{split}
\end{align*}
For the contribution of $\frakD^kg_0$ note that

\begin{align*}
\begin{split}
 \langle\frakD^kg_0,\frac{\partial_t\vecu_k}{a}\rangle+I=&\langle\partial_t\frakD^kF,\left(\frac{R^2}{GM}-\frac{1}{a}\right)\partial_t\vecu_k\rangle -\langle \partial_t \frakD^k E_1, \frac{\partial_t\vecu_k}{a}\rangle+\langle \partial_t \frakD^k \left(\frac{1}{|\calB_{1}|}\int_{\calB_{1}}\bfE_1(t,\bfx)d\bfx\right), \frac{\partial_t\vecu_k}{a}\rangle\\
 &+\langle \frakD^k E_2, \frac{\partial_t\vecu_k}{a}\rangle +\langle \frakD^k(\partial_t(a|N|^{-1})N),\frac{\partial_t\vecu_k}{a}\rangle\\
 =:&II_{1,1}+II_{1,2}+II_{1,3}+II_{1,4}+II_{1,5}.
\end{split}
\end{align*}
The contribution of $II_{1,1}$ is smaller that the contribution of $I$ computed in Step 3. The contribution of $II_{1,2}$ and $II_{1,3}$ can also be estimated similarly. Indeed $E_1$ was obtained by subtracting the leading order term $F$ from $\nabla \psi_2+\bfx_1''$ so the contributions of $E_1$ and $\frac{1}{|\calB_{1}|}\int_{\calB_{1}}\bfE_1(t,\bfx)d\bfx$ are also smaller than that of $F$ which we estimated in Step 3. Note that since the fluid is incompressible and irrotational, and $\bfE_1$ is a function of $\bfx-\bfx_1$, the time derivative of $\frac{1}{|\calB_{1}|}\int_{\calB_{1}}\bfE_1(t,\bfx)d\bfx$ can be computed as

\begin{align}\label{eq: bfE1 time der bound}
\begin{split}
\partial_t\left(\frac{1}{|\calB_{1}|}\int_{\calB_{1}}\bfE_1(t,\bfx)d\bfx\right)=&\frac{1}{|\calB_1|}\int_{\calB_1}(\nabla\bfE_1)(t,\bfx)\cdot(\bfv(t,\bfx)-\bfx_1') d\bfx+\frac{1}{|\calB_1|}\int_{\calB_1}(\partial_t\bfE_1)(t,\bfx)d\bfx\\
&=\frac{1}{|\calB_1|}\int_{\partial\calB_1}E_1n\cdot\zeta_tdS+\frac{1}{|\calB_1|}\int_{\calB_1}(\partial_t\bfE_1)(t,\bfx)d\bfx.
\end{split}
\end{align}
The time derivative $(\partial_t\bfE_1)(t,\bfx)$ can be computed similarly using the fact that the fluid is incompressible and irrotational in $\calB_2$, and the fact that in the integral expression for $\bfE_1$ in \eqref{eq: bfE1} the integrand depends only on $\bfy-\bfx_2$, where $\bfy\in\calB_2$ is the variable of integration.
Next, recalling the definition of $E_2$ from \eqref{eq: E2}, from repeated applications of the bootstrap assumptions \eqref{eq: bootstrap} and \eqref{eq: a priori bootstrap}, \prop{prop: Du Eu}, and \lem{lem: H estimates}, we get

\begin{align*}
\begin{split}
 \|\frakD^kE_2\|_{L^2(S_R)} \lesssim \frac{GM}{R^{\frac{5}{2}}}\eta^3\calE_{\leq k}^{\frac{1}{2}}+\frac{GM}{R^2}\eta^7|v_1|\lesssim \frac{C_0^{\frac{1}{2}}GM}{R^2}\eta^7|v_1|.
\end{split}
\end{align*}
Using \lem{lem: eta integration} as in the proof of \eqref{eq: I2 bound} we conclude that if $\beta$ is sufficiently large, then

\begin{align*}
\begin{split}
 \int_{T_0}^t \left|\langle II_{2,4}(s),\frac{\partial_s\vecu_k(s)}{a(s)}\rangle \right|ds\leq \frac{1}{800}C_0\eta^8(t)|v_1(t)|^2.
 \end{split}
\end{align*}
Finally direct application of \prop{prop: at E} gives

\begin{align*}
\|II_{2,5}\|_{L^2(S_R)}\lesssim \sqrt{\frac{GM}{R^3}}\eta^3\calE_{\leq k}\lesssim \sqrt{\frac{GM}{R^3}}C_0\eta^{11}|v_1|^2.
\end{align*}
It follows that if $\beta$ is sufficiently large then

\begin{align*}
 \int_{T_0}^t \left|\langle II_{2,5}(s),\frac{\partial_s\vecu_k(s)}{a(s)}\rangle \right|ds\leq \frac{1}{1000}C_0\eta^8(t)|v_1(t)|^2,
\end{align*}
completing the proof of the proposition.
\end{proof}



  \section{Tidal Energy}\label{sec: tidal capture}


In this final section we prove \thm{thm: tidal capture}. For this we will express the tidal energy in terms of the height function $h=|\zeta|-R$ and its time derivative $\partial_th$. The first subsection below is devoted to the analysis of the dynamics of $h$ and $\partial_th$. We then use this analysis in the second subsection to prove \thm{thm: tidal capture}.


\subsection{The height function}\label{subsec: h}


A crucial step in the proof of \thm{thm: tidal capture} is to analyze the behavior of the tidal energy in terms of the height function  

\begin{align}\label{def h}
h(t,\omega):=h(\zeta(t,\omega)):=|\zeta(t,\omega)|-R,\quad \omega\in\bbS^{2}.
\end{align}
Recall also the definition of the \emph{modified height function} 

\begin{align}\label{def tilh}
\tilh(t,\omega):=|\zeta(t,\omega)|^{2}-R^{2},
\end{align}
and the relation $\tilh=h(h+2R)=2Rh+h^{2}$. Our goal in this section is to derive the equation satisfied by $h$. Let $\bfphi$ and $\circbfphi$ denote the velocity potentials for $\bfv$ and $\bfx_1'$, respectively, that is,

\begin{align*}
&\bfv(t,\bfx)=-\nabla\bfphi(t,\bfx),\\
&\bfx_1'(t)=-\nabla\circbfphi(t,\bfx),\qquad \circbfphi(t,\bfx):=-\bfx_1'\cdot(\bfx-\bfx_1).
\end{align*}
We will denote $\bfx-\bfx_1$ by $\bfzeta$. In Lagrangian coordinates, let $\phi(t,p)=\bfphi(t,\xi(t,p))$ and $\rphi(t,p):=\circbfphi(t,\zeta(t,p))$. The modified height function satisfies,

\begin{align*}
\begin{split}
 \tilh_t=  2\zeta_{t}\cdot\zeta= 2R\zeta_{t}\cdot n+2\zeta_{t}\cdot(\zeta-Rn)
 =-2R\nabla_n(\phi-\rphi)+2\zeta_{t}\cdot(\zeta-Rn).
\end{split}
\end{align*}
Differentiating this equation in time we get

\begin{align}\label{eq tilh tt 1}
\begin{split}
\tilh_{tt}=&-2R\nabla_n(\phi_t-\rphi_t)-2R[\partial_t,\nabla_n](\phi-\rphi)+2\zeta_{tt}\cdot(\zeta-Rn) +2\zeta_{t}\cdot(\zeta_{t}-Rn')\\
=&:-2R\nabla_n (\phi_t-\rphi_t)+\frakR_1,
\end{split}
\end{align}
with 

\begin{align}\label{eq: frakR1}
\frakR_{1}:=-2R[\partial_t,\nabla_n](\phi-\rphi)+2u_{t}\cdot(\zeta-Rn) +2u\cdot(u-Rn').
\end{align}
To express $\phi-\rphi$ in terms of $h$, we use the Bernoulli equation

\begin{align*}
\partial_t\bfphi=\bfpsi_{1}+\bfpsi_{2}+\bfp+\frac{1}{2}|\bfv|^2,\quad \Rightarrow\quad \phi_t=\psi_{1}+\psi_{2}-\frac{1}{2}|\zeta_t+\bfx'_{1}|^2.
\end{align*}
On the other hand, by the definition of $\rphi$,

\begin{align*}
\rphi_{t}=-\zeta_{t}\cdot\bfx'_{1}-\zeta\cdot\bfx''_{1},
\end{align*}
so

\begin{align*}
\phi_{t}-\rphi_{t}=\psi_{1}+\psi_{2}-\frac{1}{2}|\zeta_{t}|^{2}-\frac{1}{2}|\bfx'_{1}|^{2}+\zeta\cdot\bfx''_{1}.
\end{align*}
Equation \eqref{eq tilh tt 1} is therefore equivalent to

\begin{align}\label{eq tilh tt 2}
\tilh_{tt}=-2R\nabla_n\psi_1-2R(\bfx''_1\cdot n+\nabla_n\psi_2)
+\frakR_1+\frakR_2,
\end{align}
where

\begin{align}\label{eq: frakR2}
\begin{split}
 \frakR_2= R\nabla_n|u|^2.
\end{split}
\end{align}
Next we derive an expression for $\psi_{1}$. Recall from \lem{lem: nabla psi} that

\begin{align*}
\nabla\psi_1=\frac{GM}{R^3}\zeta-\frac{GM}{2R^3}(I+\Hone)\zeta.
\end{align*}
This formula will be useful for analyzing $\frakR_1$, but since $\bfpsi_{1}$ is not harmonic inside $\calB_1$, $\nabla_n\psi_1$ cannot be computed by taking the inner product of this identity with $n$. Instead, to compute $\nabla_n\psi_1$ we argue as follows. First, since

\begin{align*}
\calD\bfpsi_1=\nabla\bfpsi_1=\calD((\nabla\bfpsi_1)\bfzeta)
\end{align*}
outside of $\calB_1$ (here $(\nabla\bfpsi_1)\zeta$ denotes the Clifford product),

\begin{align*}
(I+\Hone)\psi_1=(I+\Hone)((\nabla\bfpsi_1)\zeta).
\end{align*}
Using the formula $\nabla\psi_1=\frac{GM}{R^3}\zeta-\frac{GM}{2R^3}(I+\Hone)\zeta$ above we get

\begin{align}\label{eq: psi1 temp 1}
(I+\Hone)\psi_1=-\frac{GM}{R^3}(I+\Hone)(|\zeta|^2)-\frac{GM}{2R^3}(I+\Hone)(((I+\Hone)\zeta)\zeta).
\end{align}
Now recall from \eqref{eq: mu def}--\eqref{eq: nu def} that with $\mu=1-\frac{R^3}{|\zeta|^3}=\frac{3}{2R^2}\tilh+\tilh\nu$,

\begin{align*}
(I+\Hone)\zeta=(I+\Hone)(\zeta\mu).
\end{align*}
Going back to equation \eqref{eq: psi1 temp 1} we get

\begin{align}\label{eq: psi1 temp 2}
(I+\Hone)\psi_1=-\frac{GM}{R^3}(I+\Hone)(\tilh+R^2)+B+A_1,
\end{align}
where

\begin{align*}
B:=-\frac{3GM}{4R^5}(I+\Hone)(((I+\Hone)(\tilh\zeta))\zeta)
\end{align*}
and

\begin{align*}
A_1:=-\frac{GM}{2R^3}(I+\Hone)(((I+\Hone)(\tilh \nu \zeta))\zeta).
\end{align*}
To compute $B$ we keep in mind that $\Hone n f+n \Hone f$ and $\zeta-Rn$ are small quantities for any $f$, and write

\begin{align*}
B=&-\frac{3GM}{4R^5}(I+\Hone)[(\zeta)(I-\Hone)(\tilh)(\zeta)]\overbrace{-\frac{3GM}{4R^5}(I+\Hone)\left[\left(\zeta\Hone \tilh+\Hone(\zeta\tilh)\right)\zeta\right]}^{A_2}\\
=&A_2-\frac{3GM}{4R^5}(I+\Hone)[(\zeta)^2(I+\Hone)\tilh-2(\zeta)^2\Kone \tilh] \overbrace{-\frac{3GM}{2R^5}(I+\Hone)[(\zeta)((\zeta)\cdot(\Hone \tilh - \Kone \tilh))]}^{A_3}\\
=&A_2+A_3+\frac{3GM}{4R^3}(I+\Hone)[(I+\Hone)\tilh-2\Kone \tilh]\overbrace{+\frac{3GM}{4R^5}(I+\Hone)[\tilh(I+\Hone)\tilh-2\tilh\Kone \tilh]}^{A_4}\\
=&A_2+A_3+A_4+\frac{3GM}{2R^3}(I+\Hone)(I-\Kone)\tilh.
\end{align*}
Plugging back into \eqref{eq: psi1 temp 2} we get

\begin{align*}
(I+\Hone)\psi_1=-\frac{GM}{R^3}(I+\Hone)(\tilh+R^2)+\frac{3GM}{2R^3}(I+\Hone)(I-\Kone)\tilh+A_1+A_2+A_3+A_4.
\end{align*}
Taking real parts and applying $(I+\Kone)^{-1}$ gives

\begin{align}\label{eq: psi1}
\begin{split}
 \psi_1= -\frac{GM}{R}+\frac{\txtg}{2R}(I-3\Kone)\tilh+\tilde{\frakR_{3}}=-\frac{GM}{R}+\txtg(I-3\Kone)h+\frakR_{3},
\end{split}
\end{align}
where $\txtg=\frac{GM}{R^2}$,

\begin{align*}
\begin{split}
 \tilde{\frakR_{3}} = (I+\Kone)^{-1}\left(\Re(A_1+A_2+A_3+A_4)\right),
\end{split}
\end{align*}
and

\begin{align}\label{eq: frakR3}
\begin{split}
 \frakR_{3}= \tilde{\frakR_{3}}+\frac{\txtg}{2R}(I-3\Kone)h^2.
\end{split}
\end{align}
It follows that

\begin{align}\label{DN psi1 on boundary}
\nabla_n\psi_{1}=\txtg\nabla_n (I-3\Kone)h
+\nabla_n \frakR_{3}.
\end{align}

The formula above gives us an expression for the first term on the right-hand side of \eqref{eq tilh tt 2}. Since $\bfpsi_2$ is harmonic in $\calB_1$, the second term on the righthand side of \eqref{eq tilh tt 2} can simply be computed by using the relation $\nabla_{n}\psi_{2}+n\cdot\bfx''_{1}=n\cdot(\nabla\psi_2
+\bfx_1'')$. Indeed, this identity combined with \lem{lem: nabla psi} and the fact that $\bfx_1''=-\frac{\rho}{M}\int_{\calB_1}\nabla\bfpsi_2 d\bfx$, gives

\begin{align}\label{eq: force 1}
\begin{split}
\nabla_n\psi_{2}+n\cdot\bfx''_{1}=&\frac{GM\eta(t)^{3}}{8R^{3}}(\zeta\cdot n-3(\zeta\cdot\bfxi_{1})(\bfxi_{1}\cdot n))+n\cdot\left(E_1-\frac{\rho}{M}\int_{\calB_1}\bfE_{1}(t,\bfx)d\bfx\right)\\
=&\frac{GM\eta(t)^{3}}{8R^{2}}\left(1-\frac{3(\bfxi_{1}\cdot\zeta)^{2}}{R^{2}}\right)+\frakR_4,
\end{split}
\end{align}
where

\begin{align}\label{eq: frakR4}
\frakR_4:=\frac{GM\eta^3}{8R^3}\left((\zeta-Rn)\cdot n-3(\zeta\cdot\bfxi_1)(\bfxi_1\cdot(n-R^{-1}\zeta))\right)+n\cdot\left(E_1-\frac{\rho}{M}\int_{\calB_1}\bfE_{1}(t,\bfx)d\bfx\right).
\end{align}
Inserting \eqref{DN psi1 on boundary} and \eqref{eq: force 1} back into \eqref{eq tilh tt 2}, we get the desired equation for $h$ which we summarize in the following proposition.


\begin{proposition}\label{prop: h eq}
Let $\frakR_1,\dots,\frakR_4$ be as defined in \eqref{eq: frakR1}, \eqref{eq: frakR2}, \eqref{eq: frakR3}, and \eqref{eq: frakR4}, respsectively. Then $h$ satisfies 

\begin{align}\label{eq h tt 1}
 \begin{split}
  h_{tt}+ \emph{\txtg} \nabla_n(I-3\Kone) h=-\frac{GM\eta^3}{8R^{2}}\left(1-3(\bfxi_{1}\cdot R^{-1}\zeta)^{2}\right)+\frac{\frakR_1}{2R}+\frac{\frakR_2}{2R}-\nabla_n\frakR_3-\frakR_4-\frac{1}{R}h_{t}^{2}-\frac{1}{R}hh_{tt}.
 \end{split}
\end{align}
\end{proposition}


\begin{proof}
This follows by combining equations \eqref{eq tilh tt 1}--\eqref{eq: frakR4}.
\end{proof}


Except for the term $[\partial_t,\nabla_n](\phi-\rphi)$  in $\frakR_1$, all the error terms in \eqref{eq h tt 1} can be estimated using the results in Section~\ref{sec: energy}. The commutator $[\partial_t,\nabla_n]$ is also calculated in \lem{lem: H commutators}. We next derive an expression for $\phi-\rphi$ in terms of $h$ which is of independent interest for future calculations. Note that $\bfphi-\circbfphi$ is the solution to the Neumann problem

\begin{align}\label{neumann for potential}
\begin{split}
&\Delta(\bfphi-\circbfphi)=0,~\quad\qquad\qquad\qquad\quad \textrm{in}\quad \calB_{1},\\
&\bfn\cdot\nabla(\bfphi-\circbfphi)=-\bfn\cdot(\bfv-\bfx'_{1}(t))\quad \textrm{on}\quad \partial\calB_{1}.
\end{split}
\end{align}
Therefore by equation \eqref{eq: K Neumann},

\begin{align}\label{potential formula}
\begin{split}
\phi-\rphi=&\,2\Sone(I-\Kone^{*})^{-1}
\left(n\cdot u\right)=\frac{2}{R}\Sone(I-\Kone^{*})^{-1}(\zeta\cdot u+(Rn-\zeta)\cdot u)\\
=&\frac{1}{R}\Sone(I-\Kone^{*})^{-1}(\tilh_{t}+2(Rn-\zeta)\cdot u)\\
=&\frac{2}{R}\Sone(I-\Kone^{*})^{-1}(Rh_{t}+hh_t+(Rn-\zeta)\cdot u).
\end{split}
\end{align}
To estimate $[\partial_t,\nabla_n](\phi-\rphi)$,  we only need the expression $\phi-\rphi=2\Sone(I-\Kone^\ast)^{-1}(n\cdot u)$ together with \lems{lem: H commutators} and \ref{lem: frakD kernel}. However, the entire expression will be needed in the next subsection.

To analyze equation \eqref{eq h tt 1} we will decompose $h$ into spherical harmonics, but in order to do this effectively we need to replace the non-local terms $\nabla_n$ and $\Kone$ on the left-hand side of the equation by the corresponding operators on $S_R$. In the following lemma we calculate the resulting errors.


\begin{lemma}\label{lem: operators}
Let $\bfK$, $\bfS$, and $\bfD$ denote the double-layered potential, single-layered potential, and  Dirichlet-Neumann operators on $S_R$, respectively. Let $f:\bbR\times S_R\to\bbR$ and $F:\bbR\times\partial\calB_1\to\bbR$ be related as $f(t,p)=F(t,\xi(t,p))$. Then

\begin{align*}
&\Kone F- \bfK f=-\int_{T_0}^t\left(\pv\int_{S_R}\left(((u'-u)\cdot\nabla) K(\xi'-\xi)\cdot n'+K(\xi'-\xi)\cdot n_t'  \right)(s)\frac{|N'|}{|\sg'|}F'(t)dS(p')\right)ds\\
&\quad\qquad\qquad\qquad -\int_{T_0}^t\left(\pv\int_{S_R}\left(K(\xi'-\xi)\cdot n'\right)(s) \frac{\partial_t|N'|}{|\sg'|}F'(t)dS(p')\right)ds,\\
&\Sone F-\bfS f=R(\bfK f-\Kone F)+\pv\int_{\partial\calB_1}\left(\frac{1}{4\pi}\frac{h'-h}{|\xi'-\xi|^3}-K(\xi'-\xi)\cdot(Rn'-\zeta')\right)F'dS',\\
&\Kone^\ast F-\bfK f=\Kone F-\bfK f+\pv\int_{\partial\calB_1}\left(\frac{1}{2\pi R}\frac{h-h'}{|\xi'-\xi|^3}+K(\xi'-\xi)\cdot((n'-R^{-1}\zeta')+(n-R^{-1}\zeta))\right)F'dS',\\
&\nabla_n F-\bfD f=\left((I+\Kone^\ast)^{-1}-(I+\bfK)^{-1}\right)\pv\int_{\partial\calB_1}(n\cross K)\cdot (n'\cross\nabla)F'dS'\\
&\quad\qquad\qquad\qquad+(I+\bfK)^{-1}\pv\int_{S_R}(n\cross K)\cdot (n'\cross\nabla)F' \left(\frac{|N'|}{|\sg'|}-1\right)dS'\\
&\quad\qquad\qquad\qquad+(I+\bfK)^{-1}\left(\int_{T_0}^t\left(\pv\int_{S_R}
\left(n_s(s)\cross K(s)+n(s)\cross((u'(s)-u(s))\cdot\nabla)K(s)\right)\cdot n'(s)\cross\nabla F'(t) dS'\right)ds\right)\\
&\quad\qquad\qquad\qquad+(I+\bfK)^{-1}
\left(\int_{T_0}^t\left(\pv\int_{S_R}
n(s)\cross K(s)\cdot\left(|N(s)|^{-1}(u_\beta(s)F_\alpha(t)-u_\alpha(s)F_\beta(t))\right)dS'\right)ds\right)\\
&\quad\qquad\qquad\qquad-(I+\bfK)^{-1}\left(\int_{T_0}^t
\left(\pv\int_{S_R}n(s)\cross K(s)\cdot\left(\frac{\partial_s|N(s)|}{|N(s)|^{2}}(\xi_\beta(s)F_\alpha(t)-\xi_\alpha(s)F_\beta(t))\right)dS'\right)ds
\right).
\end{align*}
\end{lemma}


\begin{proof}
Recall that by assumption $\partial\calB_1(T_0)=\bfx_1(T_0)+S_R$ and by definition $\zeta(T_0,p)=p$ for all $p\in S_R$. It follows that

\begin{align*}
\Kone F (\xi)=&-\pv \int_{S_R}K(\xi-\xi')\cdot n(p') \frac{|N(p')|}{|\sg(p')|}F(\xi')dS(p')\\
=&\bfK f(p)-\int_{T_0}^t\left(\pv\int_{S_R}\frac{\partial}{\partial s}\left(K(\xi-\xi')\cdot n' \frac{|N'|}{|\sg'|}\right)(s)\,F'(t)dS(p')\right)ds\\
=&\bfK f(p)-\int_{T_0}^t\left(\pv\int_{S_R}\left(((u'-u)\cdot\nabla) K(\xi'-\xi)\cdot n'+K(\xi'-\xi)\cdot n_t'  \right)(s)\frac{|N'|}{|\sg'|}F'(t)dS(p')\right)ds\\
&-\int_{T_0}^t\left(\pv\int_{S_R}\left(K(\xi'-\xi)\cdot n'\right)(s) \frac{\partial_t|N'|}{|\sg'|}F'(t)dS(p')\right)ds.
\end{align*}
The second and third identities follow by inspection of \eqref{eq: K def} and \eqref{eq: K* def}, and noting that

\begin{align*}
-K(\xi'-\xi)\cdot n'=&\frac{1}{2\pi R}\frac{(\zeta'-\zeta)\cdot\zeta'}{|\zeta'-\zeta|^3}-K(\xi'-\xi)\cdot(n'-R^{-1}\zeta')\\
=&\frac{1}{4\pi R}\frac{1}{|\xi'-\xi|}+\frac{1}{4\pi R}\frac{\tilh'-\tilh}{|\xi'-\xi|^3}-K(\xi'-\xi)\cdot (n'-R^{-1}\zeta'),
\end{align*}
and similarly

\begin{align*}
(n+n')\cdot K(\xi'-\xi)=\frac{1}{2\pi R}\frac{\tilh-\tilh'}{|\xi'-\xi|^3}+K(\xi'-\xi)\cdot((n'-R^{-1}\zeta')+(n-R^{-1}\zeta)).
\end{align*}
For the last identity, we use \lem{lem: DN K} to write

\begin{align*}
\nabla_n F=&(I+\bfK)^{-1}\pv\int_{S_R}(n\cross K)\cdot (n'\cross\nabla)F'dS'\\
&+\left((I+\Kone^\ast)^{-1}-(I+\bfK)^{-1}\right)\pv\int_{\partial\calB_1}(n\cross K)\cdot (n'\cross\nabla)F'dS'\\
&+(I+\bfK)^{-1}\pv\int_{S_R}(n\cross K)\cdot (n'\cross\nabla)F' \left(\frac{|N'|}{|\sg'|}-1\right)dS'\\
=&\bfD f +\left((I+\Kone^\ast)^{-1}-(I+\bfK)^{-1}\right)\pv\int_{\partial\calB_1}(n\cross K)\cdot (n'\cross\nabla)F'dS'\\
&+(I+\bfK)^{-1}\pv\int_{S_R}(n\cross K)\cdot (n'\cross\nabla)F' \left(\frac{|N'|}{|\sg'|}-1\right)dS'\\
&+(I+\bfK)^{-1}\left(\int_{T_0}^t\left(\pv\int_{S_R}\left(n_s(s)\cross K(s)+n(s)\cross((u'(s)-u(s))\cdot\nabla)K(s)\right)\cdot \left(n'(s)\cross\nabla\right) F'(t) dS'\right)ds\right)\\
&+(I+\bfK)^{-1}\left(\int_{T_0}^t\left(\pv\int_{S_R}\left(n(s)\cross K(s)\right)\cdot\left(|N(s)|^{-1}(u_\beta(s)F_\alpha(t)-u_\alpha(s)F_\beta(t))\right)dS'\right)ds\right)\\
&-(I+\bfK)^{-1}\left(\int_{T_0}^t\left(\pv\int_{S_R}\left(n(s)\cross K(s)\right)\cdot(\partial_s|N(s)|)|N(s)|^{-2}(\xi_\beta(s)F_\alpha(t)-\xi_\alpha(s)F_\beta(t))dS\right)ds\right).
\end{align*}
\end{proof}


Combining the previous lemma with \prop{prop: h eq} we arrive at our final equation for $h$ which we record in the following corollary.


\begin{corollary}\label{cor: h eq}
Let 

\begin{align}
&f(t,\omega):=-\frac{\emph{\txtg}\eta^3}{8}\left(1-3(\bfxi_{1}\cdot \omega)^{2}\right),\label{eq: f definition}\\
&\frakR_5:=-\frac{1}{R}h_t^2-\frac{1}{R}hh_{tt}+\emph{\txtg} \left(\bfD(I-3\bfK)-\nabla_n(I-3\Kone)\right)h+\frac{3\emph{\txtg}\eta^3}{8R^2}\left(\bfxi_1\cdot\int_{T_0}^tu(s)ds\right)^2\nonumber\\
&\qquad\qquad+\frac{3\emph{\txtg}\eta^3}{4R^2}(\bfxi_1\cdot\omega)\left(\bfxi_1\cdot\int_{T_0}^tu(s)ds\right),\label{eq: frakR5}
\end{align}
 and $\frakR:= \frac{\frakR_1}{2R}+\frac{\frakR_2}{2R}-\nabla_n\frakR_3-\frakR_4+\frakR_5$, where $\frakR_1,\dots,\frakR_4$ are as in \prop{prop: h eq}. Then $h:\bbR\times S_R\to\bbR$ satisfies 

\begin{align}\label{eq: h eq}
(\partial_t^2+\emph{\txtg}\bfD(I-3\bfK))h(t,\omega)=f(t,\omega)+\frakR,
\end{align}
and the remainder $\frakR$ satisfies the estimates

\begin{align*}
&\|\frakR\|_{L^2(S_R)}\lesssim GM R^{-1}\eta^4,\\
&\|\partial_t\frakR\|_{L^2(S_R)}\lesssim GM R^{-2}\eta^5|v_1|,\\
&\|\partial_t^2\frakR\|_{L^2(S_R)}\lesssim (GM)^{\frac{3}{2}}R^{-\frac{7}{2}}\eta^6|v_1|,\\
&\|\partial_t^3\frakR\|_{L^2(S_R)}\lesssim (GM)^{2}R^{-5}\eta^7|v_1|.
\end{align*}
\end{corollary}

\begin{proof}
Equation \eqref{eq: h eq} is just a rewriting of \eqref{eq h tt 1}. The estimates on the remainder $\frakR$ can be proved much in the same way as the estimates in Section~\ref{sec: energy}. We will use \prop{prop: a priori} to obtain the estimates. Using \lems{lem: DN K}, \ref{lem: H commutators}, and \ref{lem: frakD kernel} and equations \eqref{potential formula} and \eqref{eq: frakR1} we get
\begin{align*}
\|\frakR_1\|_{L^2(S_R)}\lesssim \sqrt{GMR}\, \eta^7|v_1|.
\end{align*}
For $\frakR_2$, by \lems{lem: DN K} and \ref{lem: frakD kernel} and \eqref{eq: frakR2} we have 

\begin{align*}
\|\frakR_2\|_{L^2(S_R)}\lesssim R\eta^8|v_1|^2.
\end{align*}
Next, using \eqref{eq: frakR3} and \lems{lem: DN K} and \ref{lem: frakD kernel}, we get

\begin{align*}
\|\nabla_n\frakR_3\|_{L^2(S_R)}\lesssim \frac{GM}{R}\eta^6.
\end{align*}  
Here to estimate $A_2$ we have written

\begin{align*}
 \zeta \Hone\tilh+\Hone(\zeta\tilh)=Rn(\Hone-\Hone^\ast)\tilh+(\zeta-Rn)\Hone\tilh+\Hone((\zeta-Rn)\tilh),
\end{align*}
and applied \lem{lem: H H*}. Similarly for $A_3$ we have used the expression $K=K(\xi'-\xi)=-\frac{1}{2\pi}\frac{\zeta'-\zeta}{|\zeta'-\zeta|^3}$ for the kernel of $\Hone$ to write

\begin{align*}
\zeta\cdot(\Hone-\Kone)\tilh=&\zeta\cdot\pv\int_{\partial\calB_1}K\cross(n'-R^{-1}\zeta')\tilh'dS'+R^{-1}\zeta\cdot\pv\int_{\partial\calB_1}(\zeta\cross K+K\cross\zeta')\tilh'dS'\\
=&\zeta\cdot\pv\int_{\partial\calB_1}K\cross(n'-R^{-1}\zeta')\tilh'dS'.
\end{align*}
For $\frakR_4$ by \eqref{eq: frakR4} we have

\begin{align*}
\|\frakR_4\|_{L^2(S_R)}\lesssim \frac{GM}{R}\eta^6+\frac{GM}{R}\eta^4,
\end{align*}
where the last term on the right is the contribution of $E_1$. Finally by \eqref{eq: frakR5} and \lem{lem: operators}

\begin{align*}
\|\frakR_5\|_{L^2(S_R)}\lesssim \sqrt{\frac{GM}{R}}\eta^7|v_1|+\frac{GM}{R}\eta^6,
\end{align*}
finishing the proof of the estimate on $\|\frakR\|_{L^2(S_R)}$.  The estimates for the time derivatives are similar where we additionally use \prop{prop: r1' v1}, \lem{lem: 3rd derivative}, and \lem{lem: frakD kernel} and to estimate the time derivative of $\bfE_1$ we use the same argument as in \eqref{eq: bfE1 time der bound}.

\end{proof}


Equation \eqref{eq: h eq} is our working equation for the remainder of this section. Our goal now is to use this equation to obtain a lower bound for $\|\partial_th\|_{L^2(S_R)}$. For this, we decompose $h$ and $\frakR$ into spherical harmonics (see Appendix~\ref{app: Clifford})

\begin{align*}
&h=\sum_{\ell=0}^\infty h_\ell,\qquad \slashed{\Delta} h_\ell=\frac{\ell(\ell+1)}{R^2}h_\ell,\\
&\frakR=\sum_{\ell=0}^\infty\calR_\ell,\qquad \slashed{\Delta}\frakR_\ell=\frac{\ell(\ell+1)}{R^{2}}\calR_\ell.
\end{align*}
Using \prop{spherical harmonics for SR} and the facts that $\bfD h_\ell =\frac{\ell}{R}h_\ell$ and that $f$ belongs to the second eigenspace of the Laplacian, $\calY_2$, we can write equation \eqref{eq: h eq} as

\begin{align}\label{eq: h ell}
\begin{split}
&\partial_t^2 h_\ell+a_\ell h_\ell=f_\ell,\\
&f_2=f+{\calR_2},\qquad f_\ell:=\calR_\ell,\quad \ell\neq 2,\\
&\lim_{t\to T_0}h_\ell(t)=\lim_{t\to T_0}\partial_th_\ell(t)=0,\qquad \forall \ell\geq0,
\end{split}
\end{align}
where 

\begin{align*}
a_\ell:= \frac{\txtg}{R}\frac{2\ell(\ell-1)}{2\ell+1}.
\end{align*}
The following is the main result of this section.


\begin{proposition}\label{prop: h ell}
The following estimates hold for $r\in[r_0,10r_0]$:

\begin{align}\label{eq: h ell estimate 1}
\begin{split}
&R^2\eta^3\lesssim \|h_2\|_{L^2(S_R)}\lesssim 2\eta^3,\\
& \|h-h_2\|_{L^2(S_R)}\lesssim R^2\eta^4,
\end{split}
\end{align}
and

\begin{align}\label{eq: h ell estimate 2}
\begin{split}
&R\eta^4|v_1|\lesssim \|\partial_th_2\|_{L^2(S_R)}\lesssim R\eta^4|v_1|,\\
& \|\partial_t(h-h_2)\|_{L^2(S_R)}\lesssim R\eta^5|v_1|.
\end{split}
\end{align}
In particular $\| h(t)\|_{L^2(S_R)}\gtrsim R^2\eta^3(t)$ and $\|\partial_t h(t)\|_{L^2(S_R)}\gtrsim R\eta^4(t)|v_1(t)|$.

\end{proposition}


\begin{proof}
The proofs of \eqref{eq: h ell estimate 1} and \eqref{eq: h ell estimate 2} are almost identical, so we provide the details only for the latter. Solving the ODE \eqref{eq: h ell} we get

\begin{align*}
h_\ell(t)=\frac{1}{\sqrt{a_\ell}}\int_{T_0}^t\sin(\sqrt{a_\ell}(t-s))f_\ell(s)ds.
\end{align*}
Differentiating in time and using several integration-by-parts, for $\ell\geq2$ we arrive at 
\begin{align}\label{ht ell 1}
\begin{split}
\partial_{t}h_{\ell}(t)=&\int_{T_{0}}^{t}\cos(\sqrt{a_{\ell}}(t-s))f_{\ell}(s)ds\\
=&-\frac{1}{\sqrt{a_{\ell}}}\int_{T_{0}}^{t}\frac{d}{ds}(\sin(\sqrt{a_{\ell}}(t-s)))f_{\ell}(s)ds\\
=&\frac{1}{\sqrt{a_{\ell}}}\sin\left(\sqrt{a_{\ell}}(t-T_{0})\right)f_{\ell}(T_{0})+\frac{1}{\sqrt{a_{\ell}}}\int_{T_{0}}^{t}\sin(\sqrt{a_{\ell}}(t-s))f'_{\ell}(s)ds\\
=&\frac{1}{a_{\ell}}\int_{T_{0}}^{t}\frac{d}{ds}(\cos(\sqrt{a_{\ell}}(t-s)))f_{\ell}'(s)ds+\frac{1}{\sqrt{a_{\ell}}}\sin\left(\sqrt{a_{\ell}}(t-T_{0})\right)f_{\ell}(T_{0})\\
=&\frac{f_{\ell}'(t)}{a_{\ell}}-\frac{1}{a_{\ell}}\int_{T_{0}}^{t}\cos(\sqrt{a_{\ell}}(t-s))f_{\ell}''(s)ds\\
&+\frac{1}{\sqrt{a_{\ell}}}\sin\left(\sqrt{a_{\ell}}(t-T_{0})\right)f_{\ell}(T_{0})-\frac{1}{a_{\ell}}\cos\left(\sqrt{a_{\ell}}(t-T_{0})\right)f'_{\ell}(T_{0})\\
=&\frac{f_{\ell}'(t)}{a_{\ell}}+\frac{1}{a_{\ell}^{3/2}}\int_{T_{0}}^{t}\frac{d}{ds}(\sin(\sqrt{a_{\ell}}(t-s)))f_{\ell}''(s)ds\\
&+\frac{1}{\sqrt{a_{\ell}}}\sin\left(\sqrt{a_{\ell}}(t-T_{0})\right)f_{\ell}(T_{0})-\frac{1}{a_{\ell}}\cos\left(\sqrt{a_{\ell}}(t-T_{0})\right)f'_{\ell}(T_{0})\\
=&\frac{f_{\ell}'(t)}{a_{\ell}}-\frac{1}{a_{\ell}^{3/2}}\int_{T_{0}}^{t}\sin(\sqrt{a_{\ell}}(t-s))f_{\ell}'''(s)ds-\frac{1}{a_{\ell}^{3/2}}\sin\left(\sqrt{a_{\ell}}(t-T_{0})\right)f''_{\ell}(T_{0})\\
&+\frac{1}{\sqrt{a_{\ell}}}\sin\left(\sqrt{a_{\ell}}(t-T_{0})\right)f_{\ell}(T_{0})-\frac{1}{a_{\ell}}\cos\left(\sqrt{a_{\ell}}(t-T_{0})\right)f'_{\ell}(T_{0}).
\end{split}
\end{align} 
It follows from the estimates in \cor{cor: h eq}, \prop{prop: r1' v1}, and \lem{lem: 3rd derivative},  that

\begin{align*}
\int_{S_{R}}|\partial_{t}h_{2}|^{2}dS=\frac{1}{a_{2}^{2}}\int_{S_{R}}|f'_{2}|^{2}dS+O\left(|v_{1}|^{2}\eta^{10}\right).
\end{align*}
Here we want to show that the first term on the right hand side above has a lower bound of order $O\left(|v_{1}|^{2}\eta^{8}\right)$. Since $\partial_{t}f_{2}=\partial_{t}f+\partial_{t}\calR_{2}$ and $\|\partial_{t}\calR_{2}\|_{L^{2}(S_{R})}^{2}\lesssim |v_{1}|^{2}\eta^{10}$, the problem reduces to deriving a lower bound of order $O\left(|v_{1}|^{2}\eta^{8}\right)$ for $\int_{S_{R}}|\partial_{t}f|^{2}dS$. According to \eqref{eq: f definition} we can write $f$ and $\partial_{t}f$ in the schematic forms

\begin{align*}
 f(t,\omega)=c_{1}\eta^{3}(t)(1-c_{2}(\bfxi_{1}(t)\cdot\omega)^{2}),\quad \Rightarrow\quad \partial_{t}f=c_{1}\left(\eta^{3}(t)\right)'(1-c_{2}(\bfxi_{1}(t)\cdot\omega)^{2})-c_{1}c_{2}\eta^{3}(t)\left((\bfxi_{1}\cdot\omega)^{2}\right)',
\end{align*}
where $c_{1}, c_{2}$ are constants depending on $G,M,R$. This gives the following formula for $|\partial_{t}f|^{2}$:

\begin{align}\label{pt f2}
\begin{split}
 |\partial_{t}f|^{2}=&c_{1}^{2}\left(3\eta^{2}(t)\eta'(t)\right)^{2}(1-c_{2}(\bfxi_{1}\cdot\omega)^{2})^{2}+4c_{1}^{2}c_{2}^{2}\eta^{6}(t)(\bfxi_{1}\cdot\omega)^{2}(\xi'_{1}\cdot\omega)^{2}\\
 &-2c_{1}^{2}c_{2}
 \left(\eta^{3}(t)\right)'\eta^{3}(t)(1-c_{2}(\bfxi_{1}\cdot\omega)^{2})\left((\bfxi_{1}\cdot\omega)^{2}\right)'
 \end{split}
\end{align}
The integral of the second line above on $S_R$ is zero. Indeed,

\begin{align*}
-2c_{1}^{2}\left(\eta^{3}(t)\right)'\eta^{3}(t)\int_{S_{R}}(1-c_{2}(\bfxi_{1}\cdot\omega)^{2})c_{2}\left((\bfxi_{1}\cdot\omega)^{2}\right)'
dS=&c_{1}^{2}\left(\eta^{3}(t)\right)'\eta^{3}(t)\int_{S_{R}}\left((1-c_{2}(\bfxi_{1}\cdot\omega)^{2})^{2}\right)'dS\\
=&c_{1}^{2}\left(\eta^{3}(t)\right)'\eta^{3}(t)\frac{d}{dt}\left(\int_{S_{R}}(1-c_{2}(\bfxi_{1}\cdot\omega)^{2})^{2}dS\right).
\end{align*}
But due to the symmetry of $S_{R}$, the integral $\int_{S_{R}}(1-c_{2}(\bfxi_{1}\cdot\omega)^{2})^{2}dS$ does not depend on $\bfxi_{1}$ and hence on time, so the last term above vanishes. Going back to \eqref{pt f2} and using the symmetry of $S_R$ again, and using the values of $c_1$ and $c_2$ from \eqref{eq: f definition}, we see conclude that

\begin{align*}
\int_{S_R}|\partial_tf|^2dS\gtrsim& \left(\frac{GM}{R}\right)^2\eta^{4}(t)(\eta'(t))^{2}+\left(\frac{GM}{R}\right)^2\eta^{6}(t)|\bfxi_1'(t)|^2\\
\gtrsim& \left(\frac{GM}{R}\right)^2R^{-2}\eta^8(t)\left(|r_1'(t)|^2+|v_1-r_1'\bfxi_1|^2\right)\gtrsim \left(\frac{GM}{R}\right)^2R^{-2}\eta^8(t)|v_1|^2.
\end{align*}
For the last inequality above we simply observe that the estimate is trivial if $|r_1'|\geq \frac{1}{4}|v_1|$, and if $|r_1'|\leq \frac{1}{4}|v_1|$, then $|v_1-r_1'\bfxi_1|\geq \frac{3}{4}|v_1|$.  Returning to \eqref{ht ell 1} we have proved that

\begin{align}\label{eq: ht final temp 1}
&R\eta^4|v_1|\lesssim\|\partial_th_2\|_{L^2(S_R)}\lesssim R\eta^4|v_1|,\\
&\|\partial_th_\ell\|_{L^2(S_R)}\lesssim \frac{R^3}{GM}\|\partial_t\calR_\ell\|_{L^2(S_R)},\quad \ell\geq3.
\end{align}
Since $a_0=a_1=0$ this argument does not apply to $\ell=0,1$. Instead, to estimate $h_0$ and $h_1$ we argue as follows. First, since the fluid is incompressible
 
 \begin{align*}
 \begin{split}
4\pi R^{3}=&\int_{\partial\calB_{1}}n\cdot\zeta dS=\int_{S_{R}}n\cdot\zeta\frac{|N|}{|\sg|}dS=\int_{S_{R}}n\cdot\zeta dS+\int_{S_{R}}n\cdot\zeta\left(\frac{|N|}{|\sg|}-1\right)dS\\
 =&\frac{1}{R}\int_{S_{R}}\zeta\cdot\zeta dS+\int_{S_{R}}\left(n-\frac{1}{R}\zeta\right)\cdot\zeta dS+\int_{S_{R}}n\cdot\zeta\left(\frac{|N|}{|\sg|}-1\right)dS\\
 =&4\pi R^{3}+\frac{1}{R}\int_{S_{R}}\tilh dS+\int_{S_{R}}\left(n-\frac{1}{R}\zeta\right)\cdot\zeta dS+\int_{S_{R}}n\cdot\zeta\left(\frac{|N|}{|\sg|}-1\right)dS\\
 =&4\pi R^3+2\int_{S_R}hdS+\frac{1}{R}\int_{S_R}h^2dS+\int_{S_{R}}\left(n-\frac{1}{R}\zeta\right)\cdot\zeta dS+\int_{S_{R}}n\cdot\zeta\left(\frac{|N|}{|\sg|}-1\right)dS.
\end{split}
 \end{align*}
 We can now solve for $h_0=\frac{1}{4\pi}\int_{S_R}h dS$ from this equation to estimate $h_0$, and differentiating in time we obtain the desired estimate for $\partial_th_0$.  Similarly, for $h_1$, since $\int_{\calB_1}(\bfx-\bfx_1)d\bfx=0$, 

\begin{align*}
0=&\int_{\partial\calB_{1}}|\zeta|^{2}n dS=\int_{\partial\calB_{1}}\tilh n dS=2\int_{\partial\calB_{1}}h \zeta dS+2\int_{\partial\calB_1}h(Rn-\zeta)dS+\int_{\partial\calB_{1}}h^{2}n dS\\
=&2\int_{\partial\calB_{1}}h\omega dS+2\int_{\partial\calB_{1}}h(\zeta(t)-\zeta(0))dS
+2\int_{\partial\calB_1}h(Rn-\zeta)dS+\int_{\partial\calB_{1}}h^{2}n dS\\
=&2\int_{S_{R}}h\omega dS+2\int_{S_{R}}h\omega\left(\frac{|N|}{|\sg|}-1\right)dS+2\int_{\partial\calB_{1}}h(\zeta(t)-\zeta(0))dS
+2\int_{\partial\calB_1}h(Rn-\zeta)dS+\int_{\partial\calB_{1}}h^{2}n dS.
\end{align*}
Since $\{\omega^j\}_{j=1,2,3}$ form a basis for $\calY_1$, we can solve for $h_1$ from this relation, and differentiating in time we can estiamte $\partial_th_1$. Combining with \eqref{eq: ht final temp 1} we get the desired bound

\begin{align*}
\|\partial_t(h-h_2)\|_{L^2(S_R)} \lesssim R\eta^5|v_1|.
\end{align*}
\end{proof}



\subsection{Proof of \thm{thm: tidal capture}}\label{subsec: tidal capture}


Recall that the total energy 

\begin{align*}
\begin{split}
 \scE :=\frac{1}{2}\int_{\calB_{1}}|\bfv(t,\bfx)|^{2}d\bfx+\frac{1}{2}\int_{\calB_{1}}\bfpsi_{1}(t,\bfx)d\bfx+\frac{1}{2}\int_{\calB_{1}}\bfpsi_{2}(t,\bfx)d\bfx,
\end{split}
\end{align*}
is conserved during the evolution. To see this we compute the time derivative of each of the terms in the definition of $\scE$. First,

\begin{align}\label{time derivative ET 1}
 \begin{split}
  \frac{1}{2}\frac{d}{dt}\left(\int_{\calB_{1}}|\bfv(t,\bfx)|^{2}d\bfx\right)=&\int_{\calB_{1}}\bfv\cdot\left(\partial_{t}\bfv+\bfv\cdot\nabla\bfv\right)d\bfx=-\int_{\calB_{1}}\bfv\cdot\left(\nabla P+\nabla\bfpsi_{1}+\nabla\bfpsi_{2}\right)d\bfx\\
  =&-\int_{\calB_{1}}\bfv\cdot\left(\nabla\bfpsi_{1}+\nabla\bfpsi_{2}\right)d\bfx.
  \end{split}
 \end{align}
For the contribution of $\bfpsi_1$ we have

\begin{align}\label{time derivative ET 4}
\begin{split}
\frac{d}{dt}\left(\int_{\calB_{1}}\bfpsi_{1}(t,\bfx)d\bfx\right)=&G\rho\int_{\calB_{1}}\int_{\calB_{1}}\frac{\bfx-\bfy}{|\bfx-\bfy|^{3}}\cdot\left(\bfv(\bfx)-\bfv(\bfy)\right)d\bfy d\bfx\\
=&\int_{\calB_{1}}\bfv(\bfx)\cdot\nabla\bfpsi_{1}(\bfx)d\bfx
-G\rho\int_{\calB_{1}}\bfv(\bfy)\cdot\left(\int_{\calB_{1}}\frac{\bfx-\bfy}{|\bfx-\bfy|^{3}}d\bfx\right)d\bfy\\
=&2\int_{\calB_{1}}\bfv(\bfx)\cdot\nabla\bfpsi_{1}(\bfx)d\bfx.
\end{split}
\end{align}
For $\bfpsi_2$

\begin{align}\label{time derivative ET 2}
 \begin{split}
\frac{d}{dt}\left(\int_{\calB_{1}}\bfpsi_{2}(t,\bfx)d\bfx\right)=&G\rho\int_{\calB_{1}}\int_{\calB_{2}}\frac{\bfx-\bfy}{|\bfx-\bfy|^{3}}\cdot\left(\bfv(\bfx)-\bfv(\bfy)\right)d\bfy d\bfx \\
=&\int_{\calB_{1}}\bfv(\bfx)\cdot\nabla\bfpsi_{2}(\bfx)d\bfx+G\rho\int_{\calB_{2}}\bfv(\bfy)\cdot\left(\int_{\calB_{1}}\frac{\bfy-\bfx}{|\bfy-\bfx|^{3}}d\bfx\right)d\bfy\\
=&\int_{\calB_{1}}\bfv(\bfx)\cdot\nabla\bfpsi_{2}(\bfx)d\bfx+\int_{\calB_{2}}\bfv(\bfy)\cdot\nabla\bfpsi_{1}(\bfy)d\bfy\\
=&\int_{\calB_{1}}\bfv(\bfx)\cdot\nabla\bfpsi_{2}(\bfx)d\bfx+\int_{\calB_{1}}\bfv(-\bfy)\cdot\nabla\bfpsi_{1}(-\bfy)d\bfy.
 \end{split}
\end{align}
Note that $\bfv(-\bfy)=-\bfv(\bfy)$ and 

\begin{align*}
\nabla\bfpsi_{1}(-\bfy)=G\rho\int_{\calB_{1}}\frac{-\bfy-\bfz}{|-\bfy-\bfz|^{3}}d\bfz=G\rho\int_{\calB_{2}}\frac{\bfz-\bfy}{|\bfz-\bfy|^{3}}d\bfz=-\nabla\bfpsi_{2}(\bfy),
\end{align*}
so \eqref{time derivative ET 2} becomes

\begin{align}\label{time derivative ET 3}
\begin{split}
\frac{d}{dt}\left(\int_{\calB_{1}}\bfpsi_{2}(t,\bfx)d\bfx\right)=2\int_{\calB_{1}}\bfv(\bfx)\cdot\nabla\bfpsi_{2}(\bfx)d\bfx.
\end{split}
\end{align}
Multiplying \eqref{time derivative ET 4} and \eqref{time derivative ET 3} by $2^{-1}$ and adding them to \eqref{time derivative ET 1}, we see that $\frac{d\scE}{dt}=0$. 

Recall also that $\widetilde{\scEtidal}$ is defined as

\begin{align*}
\begin{split}
  \widetilde{\scEtidal}(t):=\frac{1}{2|\calB_1|}\int_{\calB_1}|\bfv|^2d\bfx-\frac{1}{2}|\bfx_1'|^2+\frac{1}{2|\calB_1|}\int_{\calB_1}\bfpsi_1d\bfx+\frac{3GM}{5R}.
\end{split}
\end{align*}
We divide $\widetilde{\scEtidal}$ into its kinetic and potential parts $\widetilde{\scEtidal}:=\scEtidal^0+\scEtidal^1$ where

\begin{align*}
&\scEtidal^0:=\frac{1}{2|\calB_1|}\int_{\calB_1}|\bfv|^2d\bfx-\frac{1}{2}|\bfx_1'|^2,\\
&\scEtidal^1:=\frac{1}{2|\calB_1|}\int_{\calB_1}\bfpsi_1d\bfx+\frac{3GM}{5R}.
\end{align*}


\begin{proof}[Proof of \thm{thm: tidal capture}]

The existence part of the theorem was already proved in \cor{cor: existence}. We now prove \eqref{eq: tidal waves}. First,
\begin{align}\label{kinetic difference 1}
\begin{split}
|\calB_1|\scEtidal^0=&\frac{1}{2}\int_{\calB_{1}}\left(|\bfv|^{2}-|\bfx'_{1}|^{2}\right)d\bfx=\frac{1}{2}\int_{\calB_{1}}\left(|\bfv-\bfx'_{1}+\bfx'_{1}|^{2}-|\bfx'_{1}|^{2}\right)d\bfx\\
=&\frac{1}{2}\int_{\calB_{1}}|\bfv-\bfx'_{1}|^{2}d\bfx+\int_{\calB_{1}}\bfx'_{1}\cdot(\bfv-\bfx'_{1})d\bfx=\frac{1}{2}\int_{\calB_{1}}|\bfv-\bfx'_{1}|^{2}d\bfx\\
=&\frac{1}{2}\int_{\calB_{1}}|\nabla(\bfphi-\circbfphi)|^{2}d\bfx=\frac{1}{2}\int_{\partial\calB_{1}}(\bfphi-\circbfphi)\nabla_{\bfn}(\bfphi-\circbfphi)dS\\
=&-\frac{1}{2}\int_{\partial\calB_{1}}(\bfphi-\circbfphi)\bfn\cdot(\bfv-\bfx'_{1})dS=-\frac{1}{2}\int_{\partial\calB_{1}}(\phi-\rphi)(n\cdot\zeta_{t})dS\\
=&-\frac{1}{4R}\int_{\partial\calB_{1}}(\phi-\rphi)\tilh_{t}dS-\frac{1}{2}\int_{\partial\calB_{1}}(\phi-\rphi)\left(\left(n-\frac{1}{R}\zeta\right)\cdot\zeta_{t}\right)dS\\
=&-\frac{1}{2}\int_{S_R}(\phi-\rphi)h_{t}dS+I,
\end{split}
\end{align}
where

\begin{align*}
\begin{split}
 I:=-\frac{1}{2}\int_{S_R}(\phi-\rphi)h_{t}\left(|N||\sg|^{-1}-1\right)dS -\frac{1}{2R}\int_{\partial\calB_{1}}(\phi-\rphi)h\,h_{t}dS-\frac{1}{2}\int_{\partial\calB_{1}}(\phi-\rphi)\left(\left(n-\frac{1}{R}\zeta\right)\cdot\zeta_{t}\right)dS.
\end{split}
\end{align*}
To understand the contribution of the main term we let $\bfP_\ell$ denote the projection on the $\ell$th eigenspace, $\calY_\ell$, of $\slashed{\Delta}$ and use \eqref{potential formula} to write

\begin{align*}
\begin{split}
\bfD(\phi-\rphi)=-h_t-\frac{1}{R}hh_t+(R^{-1}\zeta-n)\cdot u + (\bfD-\nabla_n)(\phi-\rphi).  
\end{split}
\end{align*}
Therefore,

\begin{align}\label{potential in use}
\begin{split}
\phi-\rphi=&\frac{2}{R}\bfS(I-\bfK)^{-1}(Rh_{t}+hh_t+(Rn-\zeta)\cdot u-R(\bfD-\nabla_n)(\phi-\rphi))\\
=&-\frac{1}{2}\bfP_2(Rh_{t}+hh_t+(Rn-\zeta)\cdot u-R(\bfD-\nabla_n)(\phi-\rphi))\\
&+\frac{2}{R}\bfS(I-\bfK)^{-1}(I-\bfP_2)(Rh_{t}+hh_t+(Rn-\zeta)\cdot u-R(\bfD-\nabla_n)(\phi-\rphi))\\
=&-\frac{R}{2}\partial_th_2+II,
\end{split}
\end{align}
where

\begin{align*}
\begin{split}
II:=&-\frac{1}{2}\bfP_2(hh_t+(Rn-\zeta)\cdot u-R(\bfD-\nabla_n)(\phi-\rphi))\\
&+\frac{2}{R}\bfS(I-\bfK)^{-1}(I-\bfP_2)(Rh_{t}+hh_t+(Rn-\zeta)\cdot u-R(\bfD-\nabla_n)(\phi-\rphi)).
\end{split}
\end{align*}
Combining with \eqref{kinetic difference 1} we get

\begin{align*}
\begin{split}
 |\calB_1|\scEtidal^0=\frac{R}{4}\int_{S_R}|\partial_t h_2 |^2dS+I-\frac{1}{2}\int_{S_R}II\times h_t dS.
\end{split}
\end{align*}
It follows from \props{prop: h ell}, \ref{prop: L2 C1}, \ref{prop: L2 C2}, \ref{prop: a priori}, \cor{cor: Dk H com}, \lem{lem: operators}, and \eqref{potential formula}, that for some universal constant $C_0>0$,

\begin{align}\label{eq: scEtidal0 estimate}
\begin{split}
 \scEtidal^0(t)\approx R^{-2}\|\partial_th\|_{L^2(S_R)}^2\geq C_0\eta^8(t)|v_1(t)|^2.
\end{split}
\end{align}
Here the last estimate is valid for $r\in(r_0,10r_0)$. For $\scEtidal^1$ we argue differently. First note that since $\calB_1$ is a ball of radius $R$ at $t=T_0$,

\begin{align}\label{eq: scEtidal1 identity1}
\scEtidal^1(t)=\int_{T_0}^t\frac{d}{ds}\left(\frac{1}{2|\calB_1|}\int_{\calB_1(s)}\bfpsi_1(s,\bfx)d\bfx\right)ds.
\end{align}
Recall that acceleration of $\calB_1$ due to the self-gravitational force from $\calB_1$ is zero, that is,

\begin{align*}
\int_{\calB_1}\nabla\bfpsi_1d\bfx=0.
\end{align*}
Using this observation, \eqref{time derivative ET 4}, \eqref{eq: psi1},and the curl and divergence free properties of $\bfv$ we compute

\begin{align*}
\frac{d}{dt}\left(\frac{1}{2|\calB_1|}\int_{\calB_1}\bfpsi_1d\bfx\right)=&\frac{1}{|\calB_1|}\int_{\calB_1}\bfv\cdot\nabla\bfpsi_1d\bfx=\frac{1}{|\calB_1|}\int_{\calB_1}(\bfv-\bfx_1')\cdot\nabla\bfpsi_1d\bfx=\frac{1}{|\calB_1|}\int_{\calB_1}\nabla\cdot(\bfpsi_1(\bfv-\bfx_1'))d\bfx\\
=&\frac{1}{|\calB_1|}\int_{\partial\calB_1}\zeta_t\cdot n\psi_1dS=\frac{1}{|\calB_1|}\int_{\calB_1}\zeta_t\cdot n\left(\psi_1+\frac{GM}{R}\right)dS\\
=&\frac{R^{-1}}{2|\calB_1|}\int_{\calB_1}\partial_t\tilh \left(\psi_1+\frac{GM}{R}\right)dS+\frac{1}{|\calB_1|}\int_{\calB_1}\zeta_t\cdot(n-R^{-1}\zeta)\left(\psi_1+\frac{GM}{R}\right)dS\\
=&\frac{1}{|\calB_1|}\int_{\calB_1}\partial_th\left(\psi_1+\frac{GM}{R}\right)dS+\frac{R^{-1}}{|\calB_1|}\int_{\calB_1}h\,\partial_th\left(\psi_1+\frac{GM}{R}\right)dS\\
&+\frac{1}{|\calB_1|}\int_{\calB_1}\zeta_t\cdot(n-R^{-1}\zeta)\left(\psi_1+\frac{GM}{R}\right)dS\\
=&\frac{\txtg}{|\calB_1|}\int_{\calB_1}\partial_th(I-3\Kone)h\,dS+III,
\end{align*}
where

\begin{align*}
III:=&\frac{R^{-1}}{|\calB_1|}\int_{\calB_1}h\,\partial_th\left(\psi_1+\frac{GM}{R}\right)dS+\frac{1}{|\calB_1|}\int_{\calB_1}\zeta_t\cdot(n-R^{-1}\zeta)\left(\psi_1+\frac{GM}{R}\right)dS\\
&+\frac{1}{|\calB_1|}\int_{\calB_1}\frakR_3\partial_th\,dS.
\end{align*}
Letting

\begin{align*}
IV:=\frac{\txtg}{|\calB_1|}\int_{\calB_1}\partial_th(I-3\Kone)h\,dS-\frac{\txtg}{|\calB_1|}\int_{S_R}\partial_th_2(I-3\bfK)h_2\,dS,
\end{align*}
and using \prop{prop: K S2} we arrive at

\begin{align*}
\frac{d}{dt}\left(\frac{1}{2|\calB_1|}\int_{\calB_1}\bfpsi_1d\bfx\right)=\frac{\txtg}{5|\calB_1|}\partial_t\|h_2\|_{L^2(S_R)}^2+III+IV.
\end{align*}
Inserting this into \eqref{eq: scEtidal1 identity1} and applying \prop{prop: h ell}, it follows as in the proof of \eqref{eq: scEtidal0 estimate} that for some universal constant $C_0>0$

\begin{align*}
\scEtidal^1(t)\approx\frac{GM}{R^5}\|h(t)\|_{L^2(S_R)}^2\geq C_0\frac{GM}{R}\eta^6.
\end{align*}
This completes the proof of \eqref{eq: tidal waves}. On the other hand, using the notation introduced in Subsection~\ref{subsec: x_1 J},

\begin{align*}
\begin{split}
 \tilde{\scE}=\lim_{t\to T_0}\tilde{\scE}= \frac{1}{2}c_0^2\frac{GM}{ R}\beta^{-\frac{12}{7}}.
\end{split}
\end{align*}
In view of \prop{prop: r1' v1}, if $t^\ast$ is such that $r_1(t^\ast)=3c_0^2R\beta^{\frac{2}{7}}$, and if $c_0$ is sufficiently small, then $\widetilde{\scEorbital}(t^\ast)=\tilde{\scE}-\widetilde{\scEtidal}(t^\ast)\leq -c_1\frac{GM}{R}\eta^6(t^\ast)$ for some positive constant $c_1>0$, completing the proof of the theorem.
\end{proof}



\appendix


\section{Clifford Analysis and Layered Potentials}\label{app: Clifford}

\subsection{The Clifford Algebra $C\ell_{0,2}(\bbR)$}

In this appendix we recall basic algebraic properties of the Clifford algebra $\calC:=C\ell_{0,2}(\bbR)$. We refer the reader to \cite{GM} for  a much more complete treatment, and to \cite{Wu99, Wu11} and the references therein for earlier use of Clifford analysis in study of incompressible free boundary problems. The algebra $\calC$ is the associative algebra generated by the four basis elements $\{1,e_1,e_2,e_3\}$ over $\bbR$, satisfying the relations

\begin{align}\label{eq: Clifford relations}
\begin{split}
 1e_i=e_i, \quad e_ie_j=-e_je_i,\quad i\neq j,~i,j=1,2,3, \quad e_1e_2=e_3,\quad e_i^2=-1,~i=1,2,3.   
\end{split}
\end{align}
Every element $c\in\calC$ has a unique representation $\xi=\xi^0+\sum_{i=1}^3\xi^ie_i$. Sometimes we also write $e_0$ for $1$ so that  $\xi=\sum_{i=0}^3\xi^ie_i$. We call $\xi^0$ the real part of $\xi$ and denote it by $\Re \xi$, and $\sum_{i=1}^3\xi^ie_i$ the vector part of $\xi$ and denote it by $\vect \xi$. An element $\xi\in\calC$ is referred to as a Clifford vector. If $\Re\xi=0$ we say that $\xi$ is a vector (or pure vector) and if $\vect\xi=0$ we say that $\xi$ is a real number or a scalar (or pure scalar). We identify the real numbers $\bbR$ with Clifford vectors using the relation $a\mapsto a1$, and identify vectors in $\bbR^3$ with Clifford numbers using the relation $\bfv^i\bfe_i\mapsto\bfv^ie_i$. Here $\{\bfe_1,\bfe_2,\bfe_3\}$ is the standard basis for $\bbR^3$. The dot product of two Clifford vectors $\eta$ and $\xi$ is defined as

\begin{align*}
\begin{split}
 \xi\cdot\eta:=\Re\xi \Re\eta+ \vect\xi\cdot\vect\eta=\sum_{i=0}^3\xi^i\eta^i. 
\end{split}
\end{align*}
If $\xi$ and $\eta$ are vectors, then, in view of \eqref{eq: Clifford relations}, the usual (or Clifford) product of $\xi$ and $\eta$ is

\begin{align*}
\begin{split}
 \xi\eta=-\xi\cdot\eta+\xi\cross\eta.
\end{split}
\end{align*}

\subsection{Layered Potentials and the Hilbert Transform}

The Clifford differentiation operator $\calD$, acting on Clifford algebra-valued functions, is defines as

\begin{align*}
\begin{split}
 \calD=\sum_{i=1}^3\partial_{x^i}e_i,
\end{split}
\end{align*}
where $x=(x^1,x^2,x^3)$ are the usual rectangular coordinates in $\bbR^3$. If $\bff$ is a Clifford algebra-valued function we denote the real part of $\bff$ by $\circbff:=\Re \bff$, and the vector part of $\bff$ by $\vecbff:=\vect \bff=\sum_{i=1}^3\bff^ie_i$. The Clifford derivative of $\bff$ then satisfies

\begin{align*}
\begin{split}
 \calD\bff=\nabla\circbff+\nabla\cross\vecbff -\nabla\cdot \vecbff.
\end{split}
\end{align*}
Moreover, by direct computation we see that $\calD^2\bff=-\Delta\bff=-\sum_{i=0}^3\Delta\bff^ie_i$.

Let $\Omega$ be a $C^2$, bounded, and simply-connected domain in $\bbR^3$ with boundary $\Sigma$ and complement $\Omega^c$. We denote the exterior normal vector to $\Sigma$ by $n$ and the induced volume form on $\Sigma$ by $dS$. We say that a function $\bff$ defined on $\Omega$ is Clifford analytic, if $\calD\bff=0$. Note that if $\bff$ is vector-valued, this is equivalent to $\bff$ being curl and divergence free. In general the computation above shows that the components of a Clifford analytic function are harmonic. The following simple observations are used many times in this work. For any function $\bff$ defined in $\overline{\Omega}$,

\begin{align*}
\begin{split}
 n\calD \bff=-n\cdot\nabla\bff+n\cross\nabla\bff. 
\end{split}
\end{align*}
Since the components of a Clifford analytic function are harmonic, it follows that if $f$ is the restriction to $\Sigma$ of a Clifford analytic function $\bff$ in $\Omega$, then

\begin{align}\label{eq: nabla cross dot n}
\begin{split}
 \nabla_nf=n\cross\nabla f, 
\end{split}
\end{align} 
where $\nabla_n$ denotes the Dirichlet-Neumann map of $\Omega$. Moreover, writing $f=\vecf+\circf$, with $\circf=\Re f=\circbff\vert_{\Sigma}$ and $\vecf=\vect f=\vecbff\vert_\Sigma$, since $\nabla\cross\vecbff=\nabla \circbff$ in $\Omega$, we have (the same identity holds if $\Omega$ is unbounded but $\bff$ decays at infinity)

\begin{align}\label{eq: n cross nabla clifford}
\begin{split}
 \frac{1}{|N|}(\xi_\beta \cdot \vecf_\alpha-\xi_\alpha\cdot\vecf_\beta)=\frac{1}{|N|}\xi_\alpha^i\xi_\beta^j(\partial_i{\vecf}^j-\partial_j{\vecf}^i)=n\cdot\nabla \circf = \nabla_n\circf.
\end{split}
\end{align}
In particular, when $f$ is a vector-valued Clifford analytic function

\begin{align}\label{eq: n cross nabla clifford vec}
\begin{split}
 n\cross\nabla f= \frac{1}{|N|}(\xi_\beta\cross f_\alpha-\xi_\alpha\cross f_\beta). 
\end{split}
\end{align}

We next turn to to the definition of the Hilbert transform and layered potentials. Let $\Gamma(x):=-\frac{1}{4\pi|x|}$, $x\in\bbR^3$,  be the fundamental solution of the Laplace equation in $\bbR^3$, and let

\begin{align*}
\begin{split}
 K(x):=-2\calD\Gamma(x)=-\frac{1}{2\pi}\frac{x}{|x|^3},\qquad x\in\bbR^3.
\end{split}
\end{align*}
For a Clifford algebra-valued function $f$ defined on $\Sigma$ we define the Hilbert transform of $f$ as

\begin{align*}
\begin{split}
 H_\Sigma f(\xi)=\pv\int_\Sigma K(\xi'-\xi)n(\xi') f(\xi')dS(\xi'),\qquad \xi\in\Sigma.
\end{split}
\end{align*}
We often use the shorthand $\int_{\Sigma}Kn'f'dS'$ for the integral above. Similarly, the Cauchy integral of $f$ is defined for interior points $\eta\in\Omega$ as

\begin{align*}
\begin{split}
 C_\Sigma f(\eta):= \frac{1}{2}\int_\Sigma K(\xi'-\eta)n(\xi')f(\xi')dS(\xi'),\qquad \eta\in\Omega.
\end{split}
\end{align*}
The Hilbert transform satisfies $H_{\Sigma}^2=\mathrm{Id}$ and $H_\Sigma 1=1$. The following theorem summarizes the relation between the Hilbert transform and Clifford analyticity.


\begin{theorem}\label{thm: Cuachy}[See \cite{GM} Chapter 2 and \cite{Wu99} Remark 1]
If $f$ is the restriction to $\Sigma$ of a Clifford analytic function $\bff$ defined in a neighborhood of $\Omega$, then $\bff(\eta)=C_\Sigma f(\eta)$ for every $\eta\in\Omega$. Conversely, if $f\in C^1(\Sigma,\calC)$, then $C_\Sigma f$ is Clifford analytic in $\Omega$ and continuous on $\barOmega$. Moreover, if $f\in C^1(\Sigma,\calC)$, then 

\begin{align*}
\begin{split}
 C_\Sigma f(\xi)=\frac{1}{2}f(\xi)+\frac{1}{2}H_\Sigma f(\xi) ,\qquad \xi\in\Sigma.
\end{split}
\end{align*}
Finally, $f$ is the restriction to $\Sigma$ of a Clifford analytic function $\bff\in C^0(\barOmega,\calC)$, if and only if $f=H_\Sigma f$. Similarly, $f$ is the restriction to $\Sigma$ of a Clifford analytic function $\bff\in C^0(\overline{\Omega^c},\calC)$ with $\lim_{|x|\to\infty}\bff(x)=0$, if and only if $f=-H_\Sigma f$.
\end{theorem}


The Hilbert transform is also related to the classical layered potentials for the Laplace operator on $\Omega$. Recall that the double-layered potential, $K_\Sigma$, and the single-layered potential, $S_\Sigma$, are defined as

\begin{align}\label{eq: K def}
\begin{split}
 &K_\Sigma f(\xi)=-\pv\int_\Sigma K(\xi'-\xi) \cdot n(\xi') f(\xi')dS(\xi'),\qquad \xi\in\Sigma,\\
 &S_\Sigma f(\xi)=-\frac{1}{4\pi}\int_{\Sigma} \frac{f(\xi')}{|\xi'-\xi|}dS(\xi'),~\qquad\qquad\qquad\quad\quad \xi\in\bbR^3,
\end{split}
\end{align}
where $f$ is a real-valued function defined on $\Sigma$. By direct inspection, we have

\begin{align*}
\begin{split}
 K_\Sigma f = \Re H_\Sigma f\quad\mand\quad \calD S_{\Sigma} f = -C_\Sigma (nf). 
\end{split}
\end{align*}
We define the formal $L^2(\Sigma,dS)$ adjoint of $H_\Sigma$ by $H^\ast_\Sigma:=nH_\Sigma n$. Then the $L^2(\Sigma,dS)$ adjoint of $K_\Sigma$ satisfies

\begin{align}\label{eq: K* def}
\begin{split}
 K^\ast_\Sigma f(\xi)=\pv\int_\Sigma n(\xi)\cdot K(\xi'-\xi)f(\xi')dS(\xi')=\Re H^\ast_\Sigma f(\xi),\qquad \xi\in\Sigma. 
\end{split}
\end{align}
As shown in \cite{Ke1,Ke2,Ver}, the operator $I+ K_\Sigma:L^2(\Sigma,dS)\to L^2(\Sigma,dS)$ and its adjoint $I+ K^\ast_\Sigma:L^2(\Sigma,dS)\to L^2(\Sigma,dS)$ are bounded and invertible. 


\begin{theorem}\label{thm: K inverse} For any bounded domain $\Omega\subseteq \bbR^3$ with Lipschitz boundary $\Sigma$, the operators $I+ K_\Sigma,~I+ K^\ast_\Sigma:L^2(\Sigma,dS)\to L^2(\Sigma,dS)$ are bounded and invertible. Similarly, with $L^2_0(\Sigma,dS)$ denoting the space of $L^2$ functions with zero average on $\Sigma$, the operators $I- K_\Sigma,~I- K^\ast_\Sigma:L^2_0(\Sigma,dS)\to L^2_0(\Sigma,dS)$ are bounded and invertible. Moreover the operator bounds $\|(I\pm K^\ast_\Sigma)^{-1}\|_{2,2}$ and $\|(I\pm K_\Sigma)^{-1}\|_{2,2}$ depend only on the Lipschitz constant for $\Sigma$.

\end{theorem}


The double-layered potential $K_\Sigma$ is related to the Dirichlet problem on $\Omega$ in the following way. If $f\in L^2(\Sigma,dS)$ then the unique solution to the Dirichlet problem

\begin{align*}
\begin{split}
 \Delta u=0\quad in ~\Omega, \qquad u\vert_\Sigma=f, 
\end{split}
\end{align*}
is given by

\begin{align*}
\begin{split}
 u(\eta)=\frac{1}{2\pi}\int_\Sigma \frac{(\xi-\eta)\cdot n(\xi)}{|\xi-\eta|^3}(I+K_\Sigma)^{-1}f(\xi)dS(\xi),\qquad \eta\in\Omega. 
\end{split}
\end{align*}
 See \cite{Ke1}, \cite{Ke2}, or \cite{Ver} for a proof of this fact. Similarly, if  $f\in L^2(\Sigma,dS)$ then the unique solution to the Neumann problem

\begin{align*}
\begin{split}
 \Delta u=0\quad in ~\Omega, \qquad \nabla_n u\vert_\Sigma=f, 
\end{split}
\end{align*}
is given by

\begin{align}\label{eq: K Neumann}
\begin{split}
 u(\eta)=\frac{1}{2\pi}\int_\Sigma \frac{1}{|\xi-\eta|}(I-K^\ast_\Sigma)^{-1}f(\xi)dS(\xi)=-2S_\Sigma(I-K_\Sigma^\ast)^{-1}f(\eta),\qquad \eta\in\Omega. 
\end{split}
\end{align}
See  \cite{Ke1}, \cite{Ke2}, or \cite{Ver} for a proof. Note that since $f=\nabla_nu$, $\int_{\Sigma}fdS=\int_\Omega\Delta udx=0$, so $f$ belongs to the domain of $(I-K_\Sigma^\ast)^{-1}$. In the following lemma we also provide an expression for the Dirichlet-Neumann map $\nabla_n$ of $\Sigma$ in terms of layered potentials.


\begin{lemma}\label{lem: DN K} For any differentiable function $f$ on $\Sigma$, the Dirichlet-Neumann map, $\nabla_n$, satisfies

\begin{align*}
\nabla_n f=(I+K_{\Sigma}^\ast)^{-1}\,\pv\int_{\Sigma}(n\cross K)\cdot (n'\cross\nabla) f' dS'.
\end{align*}
Moreover, if $\bff$ denotes the harmonic extension of $f$ to the interior, $\Omega$, and $\calD f$ the restriction of $\calD \bff$ to $\Sigma$, then

\begin{align*}
\begin{split}
 \calD f=(I+H_\Sigma)(n\,(I-K_\Sigma^\ast)^{-1}\nabla_nf ).
\end{split}
\end{align*}

\end{lemma}


\begin{proof}
See \cite{Wu99}, equation (3.13), for the proof of this statement.
\end{proof}


Finally, we restrict attention to the case where $\Omega$ is a ball in $\bbR^3$, $\Omega=B_1(0)$, with the standard sphere as boundary, $\partial\Omega=\bbS^2$, and present a formula for the double-layered potential $\bbK:=K_{\bbS^2}$ in terms of spherical harmonics. We start by fixing our notation for spherical harmonics. Recall that $L^2(\bbS^2)$ admits a direct sum decomposition $L^2(\bbS^2)=\bigoplus_{\ell=0}^\infty \mathring{\calY}_\ell$ into the eigenspaces of the Laplacian $\ringsDelta$ on $\bbS^2$. In other words, any smooth function $f\in L^2(\bbS^2)$ admits a unique decomposition $f=\sum_{\ell=0}^\infty f_\ell$, $f_\ell\in\mathring{\calY}_\ell$, such that

\begin{align*}
\begin{split}
 -\ringsDelta f_\ell=\ell(\ell+1)f_\ell. 
\end{split}
\end{align*}
The $\ell$th eigenspace $\mathring{\calY}_\ell$ has dimension $2\ell+1$, and using the usual polar coordinates $\xi=(\cos\varphi\sin\theta,\sin\varphi\sin\theta,\cos\theta)$ for $\xi\in\bbS^2$, an orthonormal basis for $\mathring{\calY}_\ell$ is given by

\begin{align*}
\begin{split}
 Y_{\ell}^m(\xi)=Y_\ell^m(\theta,\varphi)=(-1)^m\sqrt{\frac{2\ell+1}{4\pi}\frac{(\ell-m)!}{(\ell+m)!}} \,P_\ell^m(\cos\theta)e^{im\varphi},\qquad m=-\ell,\dots,\ell,
\end{split}
\end{align*}
where $P_\ell^m$ are the Legendre polynomials $P_\ell^m(x)=\frac{1}{2^\ell \ell!}(1-x^2)^{m/2}\frac{d^{\ell+m}}{dx^{\ell+m}}(x^2-1)^\ell$. Using this basis we can decompose a function $f\in L^2(\bbS^2)$ as

\begin{align*}
\begin{split}
 f(\xi)=\sum_{\ell=0}^\infty\sum_{m=-\ell}^\ell f_{m}^\ell Y_\ell^m(\xi),\qquad f_{m}^\ell:=\int_{\bbS^2}f(\xi')\overline{Y_{\ell}^m}(\xi') dS(\xi').
\end{split}
\end{align*}
Our interest is in understanding the action of the double-layered potential $\bbK$ on $\mathring{\calY}_\ell$.


\begin{proposition}\label{prop: K S2} Let $\bbK$ and $\bbS$ be the double-layered and single-layered potentials on $\bbS^2$, respectively. $\bbK$ is self-adjoint, that is, $\bbK^\ast=\bbK$, and for any $f_\ell\in\mathring{\calY}_\ell$

\begin{align*}
\bbK f_\ell=-\bbS f_\ell= (1-4\ringsDelta)^{-\frac{1}{2}}f_\ell:=\frac{1}{2\ell+1} f_\ell.
\end{align*}

\end{proposition}

 
 \begin{proof}
 
 The self-adjointness of $\bbK$ follows by inspection of formulas \eqref{eq: K def} and \eqref{eq: K* def}, and the observation that when $\Sigma=\bbS^2$ the exterior normal is given by $n(\xi)=\xi$. Using this same observation we get
 
 \begin{align*}
\begin{split}
 \bbK f_\ell(\xi)=\frac{1}{4\pi}\int_{\bbS^2}\frac{f_\ell(\xi')}{|\xi'-\xi|}dS(\xi') =-\bbS f_\ell(\xi).
\end{split}
\end{align*}
Since for any $\xi,\xi'\in\bbS^2$, $\frac{d}{dr}|\xi'-r\xi|^{-1}=|\xi'-r\xi|^{-3}(\xi'\cdot\xi-r)\leq 0$, by the dominate convergence theorem
 
 \begin{align*}
\begin{split}
 \bbK f_\ell(\xi)=\lim_{r\to1^+}\frac{1}{4\pi}\int_{\bbS^2}\frac{f_\ell(\xi')}{|\xi'-r\xi|}dS(\xi'). 
\end{split}
\end{align*}
 We now use the following representation for $|\xi'-r\xi|$, valid for any $r>1$:
 
 \begin{align*}
\begin{split}
 \frac{1}{|\xi'-r\xi|} =\frac{4\pi}{r}\sum_{k=0}^\infty\frac{r^{-k}}{2k+1}\sum_{m=-k}^kY_{k}^m(\xi)\overline{Y_{k}^{m}(\xi')}.
\end{split}
\end{align*}
This relation can be verified by direct computation using properties of Legendre polynomials. See for instance \cite{MathPhysBook1}. Plugging this  back into the representation for $\bbK$ above and using the fact that $f_\ell\in \mathring{\calY}_\ell$, with $f_\ell(\xi)=\sum_{m=-\ell}^\ell f^\ell_mY_\ell^m(\xi)$, we get

\begin{align*}
\begin{split}
 \bbK f_\ell(\xi)=\lim_{r\to 1^+}\frac{r^{-\ell-1}}{2\ell+1}\sum_{m=-\ell}^\ell f_m^\ell Y_\ell^m(\xi)=\frac{1}{2\ell+1}f_\ell.
\end{split}
\end{align*}
\end{proof}

The analysis above can be extended to $S_R$, the sphere of radius $R$. We denote the $\ell$th eigenspace of the Laplacian $\slashed{\Delta}$ of $S_R$ by $\calY_\ell$, so that $\slashed{\Delta} f_\ell=-R^{-2}\ell(\ell+1)f_\ell$ for $f_\ell\in\calY_\ell$. Let $\bfK$ and $\bfS$ be the double and single layered potentials on $S_R$, respectively. The following proposition is the analogue of \prop{prop: K S2} for $S_R$.
 
 \begin{proposition}\label{spherical harmonics for SR}
  Let $\bfK$ and $\bfS$ be the double-layered and single-layered potentials on $S_{R}$, respectively. Then $\bfK$ and $\bfS$ are self-adjoint and $\bfS=-R\bfK$. Moreover, for any $f_\ell\in\calY_\ell$ we have
  
  \begin{align*}
  \bfK=\frac{1}{2\ell+1}f_\ell. 
  \end{align*}
\end{proposition}


\begin{proof}
The proof is the same as that of \prop{prop: K S2}.
\end{proof}


\subsection{Commutator Formulas}

We present a number of commutator formulas that are frequently used in this paper. The proofs of these formulas are similar to the ones in the case where the domain is diffeomorphic to $\bbR^2$ where they were derived by Wu in \cite{Wu99, Wu11}. 

The following important identity is from \cite{Wu99} (see equation (3.5) on page 453): Suppose $\xi,\xi',\eta\in\bbR^3$ are arbitrary vectors with $\xi\neq\xi'$, and let $K=K(\xi'-\xi)$. Then in local coordinates $(\alpha,\beta)$ on $\Sigma$ 

\begin{align}\label{eq: Wu2 3.5}
\begin{split}
  -(\eta\cdot\nabla)K (\xi'_{\alpha'}\cross\xi'_{\beta'})+(\xi'_{\alpha'}\cdot\nabla)K(\eta\cross\xi'_{\beta'})+(\xi'_{\beta'}\cdot\nabla)K (\xi'_{\alpha'}\cross\eta)=0.
\end{split}
\end{align}
The proof of this identity is identical to the one in \cite{Wu99} so we omit it. Before stating the commutator formulas for the Hilbert transform $H_\Sigma$ we recall the following notation from \eqref{eq: Q scalar}:

\begin{align*}
\begin{split}
 Q(f,g)=\frac{1}{|N|}(f_\alpha g_\beta-f_\beta g_\alpha), 
\end{split}
\end{align*}
where $(\alpha,\beta)$ are orientation preserving local coordinates on $\Sigma$.


\begin{lemma}\label{lem: H commutators}

Let $f\in C^1(\bbR\times\Sigma)$ be a Clifford algebra-valued function, and let $a\in C^0(\bbR\times\Sigma)$ be real-valued. Then

\begin{align*}
\begin{split}
& [\partial_t,H_{\Sigma}]f=\int_{\Sigma} K(\xi'-\xi)(\xi_t-\xi'_t)\cross (n'\cross\nabla f') dS',\\
&[an\cross\nabla,H_{\Sigma}]f=\int_{\Sigma}K(\xi'-\xi)(an-a'n')\cross(n'\cross\nabla f')dS',\\
&[\partial_t^2,H_{\Sigma}]f=\int_{\Sigma}K(\xi'-\xi)(\xi_{tt}-\xi'_{tt})\cross(n'\cross\nabla f')dS'+2\int_{\Sigma}K(\xi'-\xi)(\xi_{t}-\xi'_{t})\cross(n'\cross\nabla f_t')dS'\\
&\qquad\qquad\quad+\int_{\Sigma} K(\xi'-\xi)(\xi_t-\xi'_t)Q(f',\xi'_t)dS'+\int_{\Sigma}\partial_{t}K(\xi'-\xi)(\xi_{t}-\xi'_{t})\times(n'\times\nabla f')dS.\\
&[\partial_t,\nabla_n]f=(I+K_\Sigma^\ast)^{-1}\Re\left\{-n_t H_{\Sigma}\calD f-n[\partial_t,H_{\Sigma}]\calD f+nH_\Sigma (n_tn\,\calD f)\right\}\\
&\qquad\qquad\quad+(I+K_\Sigma^\ast)^{-1}\int_{\Sigma}n\cross K(\xi'-\xi)\cdot \left(Q(f',u')-\frac{\partial_t|N'|}{|N'|}n'\cross\nabla f'\right)dS'.
\end{split}
\end{align*}

\end{lemma}


\begin{proof}
The proof is the same as the proof of Lemma~1.2 in \cite{Wu11}. The last statement has the same proof as equation (3.16) in Lemma 3.2 of \cite{Wu99}. Note that the analogues of the integration by parts in the proof of Lemma~1.2 in \cite{Wu11} in our case can be carried out using the invariant formulation in \lem{lem: int by parts}.
\end{proof}


\begin{remark}
The expressions of the form $\eta\cross (n\cross\nabla f)$ in the identities above should be understood componentwise. That is,

\begin{align*}
\begin{split}
 \eta\cross(n\cross\nabla f)=\frac{1}{|N|}\left( (\eta\cross\xi_\beta) f_\alpha - (\eta\cross\xi_\alpha) f_\beta\right).
\end{split}
\end{align*} 
\end{remark}


To show that the commutator between $n\cross\nabla$ and $H_{\Sigma}$ is small in the case where $\Sigma$ is close to a round sphere, we need an additional identity. This identity which is from the proof of Proposition~2.2 in \cite{Wu11} states that for any differentiable function $f$

\begin{align*}
\begin{split}
  \int_\Sigma K(\xi'-\xi)(\xi'-\xi)\cross (n'\cross \nabla f')dS'=0.
\end{split}
\end{align*}
We omit the simple proof which uses \eqref{eq: Wu2 3.5}  and integration by parts, and can be found in \cite{Wu11}, Proposition~2.2. Note that when $\Sigma$ is a round sphere the vector $\xi$ points in the same direction as the normal vector, so comparing with the expression for $[n\cross\nabla, H_{\Sigma}]f$ in \lem{lem: H commutators} we see that $[n\cross\nabla, H_{S_R}]f=0$. More generally, for arbitrary $\Sigma$ we get

\begin{align}\label{eq:  n cross nabla H}
\begin{split}
  [n\cross\nabla ,H_{\Sigma}]f=\frac{1}{R}\int_{\Sigma}K(\xi'-\xi)((Rn-\xi)-(Rn'-\xi'))\cross(n'\cross\nabla f')dS'.
\end{split}
\end{align}


\section{Analytic Preparations}\label{app: CM}

In this appendix we collect some general estimates that are used in the paper. In this section $\snabla$ denotes the covariant differentiation operator with respect to the standard metric on $S_R$ and $\Omega_i$ denotes the rotational vectorfield about the axis $\bfe_i$ in $\bbR^3$. $\Omega_i$ is tangent to $S_R$ and in coordinates is given by $\bfx^k\partial_\ell-\bfx^\ell\partial_k$ for some $k$ and $\ell$. Since the three vectorfields $\Omega_1$, $\Omega_2$, and $\Omega_3$ span the tangent space to $S_R$ at each point, for any function $f:S_R\to\bbR$ we have the pointwise estimate

\begin{align*}
\begin{split}
  |\snabla f|\lesssim \frac{1}{R}\sum_{i=1}^3|\Omega_if|.
\end{split}
\end{align*}
The following lemma is just the standard Sobolev estimate on $S_R$ which we record for reference.


\begin{lemma}\label{lem: Sobolev}
For any $f\in H^2(S_R)$ 

\begin{align*}
\begin{split}
 \|f\|_{L^\infty(S_R)}\lesssim \frac{1}{R} \|f\|_{L^2}+\frac{1}{R}\sum_{i=1}^3\|\Omega_i f\|_{L^2(S_R)}+\frac{1}{R}\sum_{i,j=1}^3 \|\Omega_i\Omega_jf\|_{L^2(S_R)}.
\end{split}
\end{align*}
\end{lemma}

When the axis of rotation is not important we simply write $\Omega f$ instead of $\Omega_i f$, and $\Omega^2f$ instead of $\Omega_i\Omega_jf$, etc.  We next turn to estimates on singular integral operators, due to Calderon, Coifman, David, McIntosh, and Meyer (see \cite{Calderon53}, \cite{CM-Adv}, and \cite{CM-Ann}). The general setup is as as follows. Let $J:S_R:\to \bbR^k$, $F:\bbR^k\to\bbR$, and $A:S_R\to\bbR$ be smooth functions. We want to estimate singular and nonsingular integrals of the following forms:

\begin{align}\label{eq: def C1}
\begin{split}
 C_1f(p):=\pv\int_{S_R}F\left(\frac{J(p)-J(q)}{|p-q|}\right) \frac{\prod_{i=1}^N (A_i(p)-A_i(q))}{|p-q|^{N+2}}f(q)dS(q),
\end{split}
\end{align}
where $dS$ denotes the surface measure on $S_R$, and where we assume that the kernel $$k_1(p,q)=F\left(\frac{J(p)-J(q)}{|p-q|}\right) \frac{\prod_{i=1}^N (A_i(p)-A_i(q))}{|p-q|^{N+2}}$$ is odd, that is, $k_1(p,q)=-k_1(q,p)$. similarly, 

\begin{align}\label{eq: def C2}
C_{2}f(p):=\textrm{p.v.}\int_{S_{R}}F\left(\frac{J(p)-J(q)}{|p-q|}\right)\frac{\prod_{i=1}^{N}\left(A_{i}(p)-A_{i}(q)\right)}{|p-q|^{N+1}}\frakD f(q)dS(q),
\end{align}
where we assume that the kernel $$k_2(p,q)=F\left(\frac{J(p)-J(q)}{|p-q|}\right)\frac{\prod_{i=1}^{N}\left(A_{i}(p)-A_{i}(q)\right)}{|p-q|^{N+1}}$$ is even, that is, $k_2(p,q)=k_2(q,p)$.  Here $\frakD f=\Omega f$ where $\Omega$ is a rotational vectorfield defined earlier, but we have chosen the notation $\frakD f$ instead of $\Omega f$ for consistency with the main body of the article. To prove the desired estimates we fix a finite covering $\{U_\alpha\}$ of $S_R$ with geodesic balls of radius $r\ll R$, such that if $U_\alpha\cap U_\beta=\emptyset$, then $\inf_{\{x\in U_\alpha,~ y\in U_\beta\}} |x-y| >cR$ for some absolute constant $c$. We  let $\{\chi_\alpha\}$ be a smooth partition of unity with respect to the finite covering $\{U_\alpha\}$. Before stating the main estimates we remark that with $\bfd$ denoting the geodesic distance on $S_R$ we have

\begin{align*}
\begin{split}
  |p-q|\leq \bfd(p,q)\leq C|p-q|,\qquad \forall p,q\in S_R, 
\end{split}
\end{align*}
and for some absolute constant $C>0$. We now state the estimate on $C_1$. Recall that $\snabla$ denotes the standard covariant differentiation operator on $S_R$.


\begin{proposition}\label{prop: L2 C1} With the same notation as \eqref{eq: def C1}
\begin{align}\label{eq: C1 type 1}
\begin{split}
 \|C_1f\|_{L^2(S_R)}\leq C\prod_{i=1}^N\left(\|\snabla A_i\|_{L_\infty(S_R)}+R^{-1}\|A_i\|_{L^\infty(S_R)}\right)\|f\|_{L^2(S_R)},
\end{split}
\end{align}
and

\begin{align}\label{eq: C1 type 2}
\begin{split}
   \|C_1f\|_{L^2(S_R)}\leq C\left(\|\snabla A_1\|_{L^2(S_R)}+R^{-1}\|A_1\|_{L^2(S_R)}\right)\prod_{i=2}^N\left(\|\snabla A_i\|_{L_\infty(S_R)}+R^{-1}\|A_i\|_{L^\infty(S_R)}\right)\|f\|_{L^\infty(S_R)},
\end{split}
\end{align}
where the constants depend on $F$ and $\|\snabla J\|_{L^\infty}$.
\end{proposition}


\begin{proof}
For simplicity of notation we assume $N=1$. Writing

\begin{align*}
\begin{split}
C_1f(p)=\sum_\beta \pv\int_{U_\beta}  F\left(\frac{J(p)-J(q)}{|p-q|}\right) \frac{ A(p)-A(q)}{|p-q|^{3}}\chi_\beta(q)f(q)dS(q),
\end{split}
\end{align*}
and using the triangle inequality in $L^2(S_R)$ we estimate

\begin{align*}
\begin{split}
  \|C_1f\|_{L^2(S_R)}^2\lesssim& \sum_{\alpha,\beta}\int_{U_\alpha} \chi_\alpha(p) \left(\pv\int_{U_\beta}F\left(\frac{J(p)-J(q)}{|p-q|}\right) \frac{A(p)-A(q)}{|p-q|^{3}}\chi_\beta(q)f(q)dS(q)\right)^2 dS(p)\\
  =:&I_{\alpha\beta}.
\end{split}
\end{align*}
We estimate $I_{\alpha\beta}$ differently according to whether $U_\alpha$ and $U_\beta$ intersect nontrivially. If $U_\alpha\cap U_\beta=\emptyset$ then $|p-q|\gtrsim R$, so we simply have

\begin{align*}
\begin{split}
 I_{\alpha\beta}\lesssim R^{-4}\|A\|_{L^\infty(S_R)}^2\left(\int_{S_R} f(q)dS(q)\right)^2\lesssim R^{-2} \|A\|_{L^\infty(S_R)}^2\|f\|_{L^2(S_R)}^2,
\end{split}
\end{align*}
in accordance with \eqref{eq: C1 type 1}. Similarly

\begin{align*}
\begin{split}
  I_{\alpha\beta}\lesssim R^{-6}\int_{S_R}\|f\|_{L^\infty(S_R)}^2\left(\int_{S_R}A(q)dS(q)+R^2A(p)\right)^2dS(p)\lesssim R^{-2}\|f\|_{L^\infty(S_R)}^2\|A\|_{L^2(S_R)}^2,
\end{split}
\end{align*}
in accordance with \eqref{eq: C1 type 2}. Next, if $U_\alpha\cap U_\beta\neq \emptyset$ then we define $U$  to be a slight enlargement of $U_\alpha\cup U_\beta$ and let $\varphi:\bbB\to U\subseteq \bbR^3$ be a diffeomorphism onto its image, satisfying 

\begin{align}\label{eq: chart estimates}
\begin{split}
 c_1\leq \frac{|\varphi(x)-\varphi(y)|}{|x-y|} \leq c_2,\quad\forall x,y\in\bbB \qquad\mathrm{and}\qquad c_3\leq |\varphi'(x)|\leq c_4,\quad \forall x\in\bbB.
\end{split}
\end{align}
for some absolute constants $c_1, \dots, c_4$, where $\bbB=\{x\in\bbR^2 \mathrm{~s.t.~} |x|\leq r\}$, and $|\varphi'|:=|\partial_1\varphi\cross\partial_2\varphi|$.
We can now use the coordinate function $\varphi$ to write the integrals $I_{\alpha\beta}$ as integrals on $\bbR^2$. For this we first extend $\varphi$ to a map from all of $\bbR^2$ to $\bbR^3$ in such a way that \eqref{eq: chart estimates} holds on all of $\bbR^2$ and $\varphi(x)\notin U_\alpha\cup U_\beta$ for $x\notin\bbB$. We then have

\begin{align*}
\begin{split}
 &I_{\alpha\beta}=\\
 &\int_{\bbR^2} \chi_\alpha(\varphi(x)) \left(\pv\int_{\bbR^2}F\left(\frac{J(\varphi(x))-J(\varphi(y))}{|\varphi(x)-\varphi(y)|}\right) \frac{A(\varphi(x))-A(\varphi(y))}{|\varphi(x)-\varphi(y)|^{3}}\chi_\beta(\varphi(y))f(\varphi(y))|\varphi'(y)|dy\right)^2 |\varphi'(x)|dx.
\end{split}
\end{align*}
Let $\tilA:\bbR^2\to\bbR$, $\tilf:\bbR^2\to\bbR$, and $\tilJ:\bbR^2\to\bbR^{k+1}$ be

\begin{align*}
\begin{split}
 \tilA =A\circ \varphi,\quad \tilJ=(J\circ\varphi,\varphi), \quad \tilf=|\varphi'|f\circ\varphi.
\end{split}
\end{align*}
We also define $\tilF\in C^\infty(\bbR^{k+1},\bbR)$ in such away that outside the interval $|z_{k+1}|\leq \delta$, $\delta\ll c_1$,

\begin{align*}
\begin{split}
 \tilF (z_1,\dots,z_{k+1})=\frac{1}{|z_{k+1}|^3}F\left(\frac{z_1}{|z_{k+1}|},\dots,\frac{z_k}{|z_{k+1}|}\right).
\end{split}
\end{align*}
It follows that

\begin{align*}
\begin{split}
 F\left(\frac{J(\varphi(x))-J(\varphi(y))}{|\varphi(x)-\varphi(y)|}\right) \frac{A(\varphi(x))-A(\varphi(y))}{|\varphi(x)-\varphi(y)|^{3}} =\tilF\left(\frac{\tilJ(x)-\tilJ(y)}{|x-y|}\right)\frac{\tilA(x)-\tilA(y)}{|z-y|^3}.
\end{split}
\end{align*}
Since $|\varphi'|,|\chi_\alpha|, |\chi_\beta|\lesssim1,$ and $\|\nabla\tilA\|_{L^p(\bbR^2)}\lesssim\|\snabla A\|_{L^p(S_R)}$ for $p=2,\infty$, the contribution of $I_{\alpha\beta}$ can be bounded using Proposition~2.6 in \cite{Wu11}.

\end{proof}


For $C_{2}f$ we have the following estimate.


\begin{proposition}\label{prop: L2 C2}
With the same notation as \eqref{eq: def C2}
\begin{align}\label{eq: C2 type 1}
\begin{split}
 \|C_2f\|_{L^2(S_R)}\leq C\prod_{i=1}^N\left(\|\snabla A_i\|_{L_\infty(S_R)}+R^{-1}\|A_i\|_{L^\infty(S_R)}\right)\|f\|_{L^2(S_R)},
\end{split}
\end{align}
and

\begin{align}\label{eq: C2 type 2}
\begin{split}
   \|C_2f\|_{L^2(S_R)}\leq C\left(\|\snabla A_1\|_{L^2(S_R)}+R^{-1}\|A_1\|_{L^2(S_R)}\right)\prod_{i=2}^N\left(\|\snabla A_i\|_{L_\infty(S_R)}+R^{-1}\|A_i\|_{L^\infty(S_R)}\right)\|f\|_{L^\infty(S_R)},
\end{split}
\end{align}
where the constants depend on $F$ and $\|\snabla J\|_{L^\infty}$.
\end{proposition}


\begin{proof}

This follows from Proposition~\ref{prop: L2 C1} and integration by parts. Note that since $\snabla\cdot\Omega=0$, for any differentiable functions $f$ and $g$ defined on $S_R$ we have $\int_{S_R}f\Omega gdS=-\int_{S_{R}} g\Omega f dS$. 
\end{proof}

To estimate the derivatives of integral operators such as $C_1$ and $C_2$ we need to find a convenient expression for these derivatives. We derive such an expression in the next lemma, both in the case where the domain is $S_R$ and when the domain is the free surface boundary $\partial\calB_1$. The latter is of separate interest in the paper.


\begin{lemma}\label{lem: frakD kernel}
\begin{enumerate}
\item Suppose $L=L(\xi,\xi')$ is a vector-valued kernel on $\partial\calB_1\times\partial\calB_1$ such that $|\xi-\xi'|^2L(\xi,\xi')$ is continuous and $L$ is differentiable away from $\xi'=\xi$.  Then for any  Clifford algebra-valued differentiable function $g$ on $\partial\calB_1$,

\begin{align}\label{eq: D H temp 2}
\begin{split}
\frakD\int_{\partial\calB_1} Lg'dS'-\int_{\partial\calB_1}L\frakD'g'dS'=&\int_{\partial\calB_1}\left((\Omega+\Omega')L-\bfe\cross L\right)g'dS'\\
 &+\int_{\partial\calB_1}\frac{Lg'}{|N'||\sg'|^{-1}}\Omega'(|N'||\sg'|^{-1})dS'.
\end{split}
\end{align}

\item Suppose $L=L(\xi,\xi')$ is a vector-valued kernel on $S_R\times S_R$ such that $|\xi-\xi'|^2L(\xi,\xi')$ is continuous and $L$ is differentiable away from $\xi'=\xi$.  Then for any  Clifford algebra-valued differentiable function $g$ on $S_R$,

\begin{align}\label{eq: D H temp 2}
\begin{split}
\frakD\int_{S_R} Lg'dS'-\int_{S_R}L\frakD'g'dS'=&\int_{S_R}\left((\Omega+\Omega')L-\bfe\cross L\right)g'dS'.
\end{split}
\end{align}
\end{enumerate}
\end{lemma}


\begin{proof}

The second part is a consequence of the first in the special case where $\partial\calB_1$ is a round sphere, so we concentrate on the first. If $g'$ is scalar-valued, then by Lemma~\ref{lem: product rule}

\begin{align*}
\begin{split}
  \frakD\int_{\partial\calB_1}Lg'dS'=&\int_{\partial\calB'}(\Omega L-\bfe\cross L)g'dS'\\
  =&\int_{\partial\calB_1}((\Omega+\Omega')L-\bfe\cross L)g'dS'+\int_{\partial\calB_1}L\Omega'g'dS'\\
  &+\int_{\partial\calB_1}\frac{Lg'}{|N'||\sg'|^{-1}}\Omega'(|N'||\sg'|^{-1})dS',
\end{split}
\end{align*}
proving \eqref{eq: D H temp 2} when $g'$ is scalar-valued. If $g'$ is vector-valued, then

\begin{align*}
\begin{split}
 \frakD\int_{\partial\calB_1} L g'dS' =&-\int_{\partial\calB_1}(\Omega L)\cdot g'dS'+\int_{\partial\calB_1}(\Omega L)\cross g'dS'-\int_{\partial\calB_1}\bfe\cross(L\cross g')dS'\\
 =&\int_{\partial\calB_1}((\Omega+\Omega') L)g'dS'-\int_{\partial\calB_1}\bfe\cross(L\cross g')dS'\\
 &+\int_{\partial\calB_1}L\Omega'g'dS'+\int_{\partial\calB_1}\frac{Lg'}{|N'||\sg'|^{-1}}\Omega'(|N'||\sg'|^{-1})dS'\\
 =&\int_{\partial\calB_1}((\Omega+\Omega') L)g'dS'-\int_{\partial\calB_1}\bfe\cross(L\cross g')dS'+\int_{\partial\calB_1}L(\bfe\cross g')dS'\\
 &+\int_{\partial\calB_1}L\frakD'g'dS'+\int_{\partial\calB_1}\frac{Lg'}{|N'||\sg'|^{-1}}\Omega'(|N'||\sg'|^{-1})dS'\\
  =&\int_{\partial\calB_1}((\Omega+\Omega') L)g'dS'-\int_{\partial\calB_1}\bfe\cross(L\cross g')dS'+\int_{\partial\calB_1}g'\cdot(\bfe\times L)dS'\\
 &+\int_{\partial\calB_{1}}(L\times\bfe)\times g'dS'-\int_{\partial\calB_{1}}\bfe\times(g'\times L)dS'\\&+\int_{\partial\calB_1}L\frakD'g'dS'+\int_{\partial\calB_1}\frac{Lg'}{|N'||\sg'|^{-1}}\Omega'(|N'||\sg'|^{-1})dS'\\
  =&\int_{\partial\calB_1}((\Omega+\Omega') L)g'dS'-\int_{\partial\calB_1}(\bfe\cross L)g'dS'\\
 &+\int_{\partial\calB_1}L\frakD'g'dS'+\int_{\partial\calB_1}\frac{Lg'}{|N'||\sg'|^{-1}}\Omega'(|N'||\sg'|^{-1})dS',
\end{split}
\end{align*}
proving \eqref{eq: D H temp 2} when $g'$ is vector-valued.

\end{proof}


Combining the previous lemma (in the case where $\partial\calB_1$ is $S_R$ and $\xi$ is the identity mapping) with Sobolev estimates and Proposition~\ref{prop: L2 C1} we get the following $L^\infty$ estimate.


\begin{proposition}\label{prop: Linfty C1}

There is a constant $C=C(F, \|\snabla J\|_{L^{\infty}}, \|\snabla^2J\|_{L^{\infty}}, \|\snabla^3J\|_{L^{\infty}})$ such that

\begin{align*}
\|C_{1}f\|_{L^{\infty}(S_{R})}\leq CR^{-N}\prod_{i=1}^{N}\left(\sum_{k=0}^{3}R^{k}\|\snabla^{k}A_{i}\|_{L^{\infty}(S_{R})}\right)\left(\sum_{k=0}^{2}R^{k}\|\snabla^{k}f\|_{L^{\infty}(S_{R})}\right).
\end{align*}

\end{proposition}


\begin{proof}
The proof is immediate from Lemmas~\ref{lem: Sobolev} and~\ref{lem: frakD kernel} and Proposition~\ref{prop: L2 C1}. Note that we have used the embedding $L^\infty(S_R)\hookrightarrow L^2(S_R)$ to replace the $L^2$ norms on the right-hand side in Proposition~\ref{prop: L2 C1} by $L^\infty$ norms.
\end{proof}


In applications we often encounter integrals similar to $C_1$ and $C_2$ which are defined on $\partial\calB_1$ rather than $S_R$. The estimates in \props{prop: L2 C1}--\ref{prop: Linfty C1} can be transferred to this case if we parameterize $\partial\calB_1$ by $\xi$. More precisely, we have the following corollary.


\begin{corollary}\label{cor: Dk H com}
Suppose for some $\ell\geq5$, $\xi$ satisfies

\begin{align*}
\sup_{p,q\in S_R}\frac{|p-q|}{|\xi(p)-\xi(q)|}\leq c_0,\quad\sum_{|\alpha|\leq\ell}\|\frakD^\alpha(|N||\sg|^{-1})\|_{L^2(S_R)}\leq c_1,\quad
\sum_{|\alpha|\leq \ell}\|\frakD^\alpha\xi\|_{L^2(S_R)}\leq c_2,
\end{align*}
where $N=\xi_\alpha\cross\xi_\beta$, for some constants $c_0$, $c_1$, and $c_2$.

\begin{align*}
h(\xi)=\int_{\partial\calB_1}K(\xi'-\xi)(g(\xi')-g(\xi))f(\xi')dS(\xi')
\end{align*}
for some given functions $f$ and $g$. Then there is a constant $C=C(c_0,c_1,c_2)$ such that for all $k\leq \ell$

\begin{align*}
\begin{split}
  \|\frakD^kh\|_{L^2(\partial\calB_1)}\leq \frac{C}{R}\sum_{j\leq \max\{k,3\} } \|\frakD^jg\|_{L^2(\partial\calB_1)} \sum_{ j\leq\max\{k-1,3\}} \|\frakD^jf\|_{L^2(\partial\calB_1)}.
\end{split}
\end{align*}
\end{corollary}


\begin{proof}
To be able to use \props{prop: L2 C1} and \ref{prop: L2 C2} we write $F$ (identifying $h\circ\xi$ with $h$ as usual) as

\begin{align*}
\begin{split}
 h(p)=-\frac{1}{2\pi}\int_{S_R}\frac{|p'-p|^3}{|\xi(p')-\xi(p)|^3}\frac{(\xi(p')-\xi(p))(g(p')-g(p))}{|p'-p|^3}\frac{|N(p')|}{|\sg(p')|}f(p')dS(p'). 
\end{split}
\end{align*}
Similarly for the $L^2(\partial\calB_1)$ norms we write

\begin{align*}
\begin{split}
 \|\frakD^jh\|_{L^2(\partial\calB_1)}^2=\int_{S_R}|\frakD^jh(p)|^2\frac{|N(p)|}{|\sg(p)|}dS(p). 
\end{split}
\end{align*}
The statement now follows from \props{prop: L2 C1} and \ref{prop: L2 C2}, together with \lem{lem: frakD kernel}, and the Sobolev embedding. We refer the reader to Lemma~2.5 in \cite{BMSW1} for a similar argument in dimension two.
\end{proof}


\bibliographystyle{plain}
\bibliography{ref}

 \bigskip

\centerline{\scshape Shuang Miao}
\smallskip
{\footnotesize
 \centerline{B\^atiment de Math\'ematiques, \'Ecole Polytechnique F\'ed\'erale de Lausanne}
\centerline{Station 8, CH-1015, Lausanne, Switzerland}
\centerline{\email{shuang.miao@epfl.ch}}
} 

 \medskip

\centerline{\scshape Sohrab Shahshahani}
\medskip
{\footnotesize
 \centerline{Department of Mathematics and Statistics, University of Massachusetts}
\centerline{Lederle Graduate Research Tower, 710 N. Pleasant Street,
Amherst, MA 01003-9305, U.S.A.}
\centerline{\email{sohrab@math.umass.edu}}
}

\end{document}